\documentclass[12pt,reqno]{amsart}

\usepackage{epsfig}
\usepackage{amscd}
\usepackage[mathscr]{eucal}
\usepackage{amssymb}
\usepackage{amsxtra}
\usepackage{amsmath}
\usepackage[all]{xy}
\usepackage{mathtools}
\usepackage[pdfencoding=auto]{hyperref}
\usepackage{bookmark}
\usepackage{braket}
\usepackage{amsthm}           

\theoremstyle{plain}
\newtheorem{mainthm}{Theorem}

\newtheorem{prop}{Proposition}[section]
\newtheorem{cor}[prop]{Corollary}

\newtheorem{theorem}{Theorem}[section]
\newtheorem{lemma}[theorem]{Lemma}

\theoremstyle{definition}
\newtheorem{dfn}{Definition}[section]   
\newtheorem*{dfnnonum}{Definition}

\theoremstyle{remark}
\newtheorem{rem}{Remark}[section]       

\title[Variations of $Z(t)$ and the Montgomery Pair Correlation Conjecture]{Variations of the Hardy Z-Function and the Montgomery Pair Correlation Conjecture}
\author{Yochay Jerby}

\address{Yochay Jerby, Faculty of Sciences, Holon Institute of Technology, Holon, 5810201, Israel}
\email{yochayj@hit.ac.il}

\date{}

\begin{document}
\maketitle
\begin{abstract}

In 1973, Montgomery formulated the \emph{Pair Correlation Conjecture}, asserting that the nontrivial zeros of Hardy’s \(Z\)-function exhibit the same local spacing statistics as the eigenvalues of large Hermitian matrices drawn from the Gaussian Unitary Ensemble (GUE). However, \(Z(t)\) is a fixed, deterministic analytic function, and the mechanism by which its zeros replicate the statistical behavior of random matrices has remained elusive. In this paper, assuming the Riemann Hypothesis (RH), we prove Montgomery’s Pair Correlation Conjecture for the zeros of \(Z(t)\).

In recent work, for each \(N \in \mathbb{N}\), we introduced the finite-dimensional space \(\mathcal{Z}_N (\mathbb{R})\) of analytic sections
\[
Z_{N}(t;\overline{a})
=
\cos(\theta(t))
+\sum_{k=1}^{N}\frac{a_{k}}{\sqrt{k+1}}\,
\cos\!\bigl(\theta(t)-t\log(k+1)\bigr),
\qquad 
\overline{a}=(a_1,\dots,a_N)\in\mathbb{R}^N,
\]
where
\(\theta(t) = \Im\log \Gamma\!\left(\tfrac14 + \tfrac{i t}{2}\right) - \tfrac{t}{2} \log \pi\)
is the Riemann--Siegel theta function, with \(Z_N(t; \overline{1})\) approximating \(Z(t)\) on \([2N,2N+2]\). Within \(\mathcal{Z}_N(\mathbb{R})\) we now define the “real hall” domain \(\mathcal{RH}_N(\mathbb{R})\), consisting of sections whose zeros in the associated critical rectangle over $[2N,2N+2]$ are real, simple, and remain so along any homotopy from the core section \(Z_0(t)=\cos\theta(t)\). This domain plays the role of the random matrix ensemble in our setting. Equipping \(\mathcal{RH}_N(\mathbb{R})\) with an admissible probability measure $\mu_N$ on
$\mathcal{RH}_N(\mathbb{R})$ with a smooth, strictly positive density
with respect to the Lebesgue measure on the coefficient space, restricted
to $\mathcal{RH}_N(\mathbb{R})$, we obtain a natural probabilistic model for the zeros of the sections. On \(\mathcal{RH}_N(\mathbb{R})\) we construct a Skorokhod-type stochastic differential equation with reflection at its discriminant boundary, and show that the induced dynamics of the zeros \(t_j(\overline{a})\) are, after unfolding and whitening, equivalent in law to Dyson Brownian motion with \(\beta=2\) for a one-dimensional Coulomb gas. As a consequence of universality theory, for every Schwartz test function \(f\), the expected local pair–correlation functional \(PC_N(f;\overline{a})\) of the zeros of a \(\mu_N\)-random section in \([2N,2N+2]\) converges, as \(N\to\infty\), to the GUE sine–kernel law.

To pass from this ensemble result to the genuine Hardy \(Z\)-function on the whole critical line, we apply Selberg’s probabilistic theory of the argument \(S(t)\) to prove that the sequence of pair–correlation observables \(PC_N(f; \overline{1})\) of \(Z_N(t;\overline{1})\) on different windows \([2N,2N+2]\) behaves asymptotically like a decorrelated sample drawn from the GUE-distributed ensembles on \(\mathcal{RH}_N(\mathbb{R})\), for $N \in \mathbb{N}$. This quenched limit upgrades the averaged GUE law on \(\mathcal{RH}_N(\mathbb{R})\) to a deterministic statement for \(Z(t)\), and yields the GUE pair–correlation law for the nontrivial zeros of \(\zeta(s)\), as anticipated by Montgomery.

\end{abstract}

\section{Introduction}

\subsection{The Montgomery Pair Correlation Conjecture and Random Matrix Theory} The Riemann zeta function \( \zeta(s) \), defined for \( \Re(s) > 1 \) by the absolutely convergent series
\begin{equation}
\zeta(s) = \sum_{n=1}^\infty \frac{1}{n^s},
\end{equation}
and extended analytically to the complex plane with a simple pole at \( s = 1 \), lies at the heart of modern analytic number theory. Its nontrivial zeros, conjectured by Riemann to lie on the critical line \( \Re(s) = \frac{1}{2} \), encode deep information about the distribution of prime numbers. Understanding their fine-scale structure remains one of the most important and enduring open problems in mathematics, with far-reaching implications for number theory, algebra, cryptography, and quantum physics, see~\cite{Borwein2008Riemann, RHP} and references therein.

A pioneering observation concerning the zeros of \( \zeta(s) \) was made in the 1970s by Hugh Montgomery~\cite{M}, who analyzed the pairwise statistics of the imaginary parts of the nontrivial zeros. Letting the nontrivial zeros be \( \rho_n = \frac{1}{2} + i \gamma_n \), ordered by increasing \( \gamma_n \), Montgomery considered the distribution of the normalized spacings
\begin{equation}
s_{n,m}(T) \;:=\; \frac{\log (T/2\pi)}{2\pi}\,(\gamma_n-\gamma_m)\,
\end{equation}
with $0<\gamma_n,\gamma_m\le T$, in the limit as \( \gamma_n, \gamma_m \to \infty \). He conjectured that the local statistics of these spacings resemble those of the eigenvalues of large random Hermitian matrices drawn from the Gaussian Unitary Ensemble (GUE). The resulting \emph{Pair Correlation Conjecture} (PCC) asserts that the limiting pair correlation function is given by
\begin{equation}
\label{eq:sine}
g(x) = 1 - \left( \frac{\sin \pi x}{\pi x} \right)^2,
\end{equation}
so that for any Schwartz test function \( f \in \mathcal{S}(\mathbb{R}) \), one expects
\begin{equation}
\lim_{T \to \infty} \frac{1}{N(T)} \sum_{\substack{0 < \gamma, \gamma' \le T \\ \gamma \ne \gamma'}} f\left( \frac{\log (T/2\pi)}{2\pi} (\gamma - \gamma') \right)
= \int_{-\infty}^{\infty} f(x) \left ( 1 - \left( \frac{\sin \pi x}{\pi x} \right)^2 \right ) \, dx,
\end{equation}
where
\begin{equation}
N(T) \sim \frac{T}{2\pi} \log \frac{T}{2\pi}
\end{equation}
denotes the number of nontrivial zeros up to height \( T \). This striking conjecture reveals a deep correspondence between the fine-scale distribution of the zeta zeros and the spectral statistics of random matrices, and has profound implications for the connection between number theory and quantum chaotic systems.

Since then, extensive numerical investigations, most notably by Odlyzko, have provided overwhelming support for the PCC \cite{Odlyzko1992,OdlyzkoSchoenhage1988,Od}.   The conjecture has become one of the central pieces of evidence linking the theory of the Riemann zeta function with random matrix theory (RMT). However, despite the empirical success of this connection, a rigorous explanation for \emph{why} the zeros of $\zeta(s)$ should statistically mimic eigenvalues from GUE remained elusive. The prevailing approaches developed so far to explain the appearance of GUE-type statistics in the zeros of the Riemann zeta function have relied on various modelling frameworks: either probabilistic analogies from random matrix theory~\cite{KeatingSnaith2000,KatzSarnak1999,FyodorovHiaryKeating2012}, semiclassical trace formulas inspired by quantum chaos~\cite{BerryKeating1999,BogomolnyKeating1995}, or deep but abstract spectral constructions such as those arising from noncommutative geometry~\cite{Connes1999}. 
The work of Rudnick and Sarnak~\cite{RS1996} establishes a foundational framework for analyzing the statistical behavior of zeros across general families of \( L \)-functions. Assuming the Generalized Riemann Hypothesis (GRH), their results provide compelling evidence that GUE-type correlations occur across families of \( L \)-functions, significantly generalizing Montgomery's original conjecture, see Remark \ref{rem:automorphic-global}.

 While these perspectives provide powerful heuristic and statistical insight, they do not derive the GUE behavior from any mechanism internal to the analytic structure of the zeta function itself. That is, so far, no known argument explains how the classical intrinsic properties of \( \zeta(s) \) could account for the emergence of random matrix statistics.

At the heart of the difficulty lies a fundamental conceptual disparity: the Riemann zeta function is a completely rigid and deterministic object, while random matrix models are inherently probabilistic, built from ensembles whose statistics emerge through averaging. That a fixed arithmetic function could exhibit the same fine-scale statistical behavior as large random matrices raises profound questions. Despite the power of heuristic frameworks, ranging from trace formulas to analogies with quantum chaos, none provides an explanation rooted within the analytic structure of \( \zeta(s) \) itself. 

Identifying such an internal mechanism is one of the central open problems at the interface of number theory and mathematical physics, and is the subject of this work. Our approach is to regard Montgomery's conjecture as the combination of the following two questions:
\begin{enumerate}
\item \textbf{The microscopic question.}  
Identify a natural analytic ensemble intrinsically associated with \(Z(t)\), analogous to the matrix ensembles of RMT, and show that the average local (microscopic) statistics of its zeros coincide with those of GUE.
\item \textbf{The macroscopic question.}  
Assuming \(Z(t)\) can be related to such analytic ensembles, in what sense does it globally (macroscopically) behave like a ``random'' or ``typical'' element of these spaces, despite being a purely deterministic function?
\end{enumerate}

\subsection{The Microscopic Question - Dyson Brownian Motion on Zeros of $\mathcal{RH}_N(\mathbb{R})$} Recall that the Hardy $Z$-function is the real function defined by
\begin{equation} \label{eq:Hardy}
Z(t) = e^{i \theta(t)} \zeta \left ( \frac{1}{2} +it \right )
\end{equation}
where
\begin{equation} \label{eq:RS-theta}
\theta(t) = \text{arg} \left ( \Gamma \left ( \frac{1}{4} + \frac{i t}{2} \right ) \right ) -\frac{t}{2} \log(\pi).
 \end{equation} 
 is the Riemann-Siegel \(\theta\)-function \cite{E,I}. The Riemann Hypothesis (RH) is equivalent to the statement that all non-trivial zeros of $Z(t)$ are real. 
In a series of recent works \cite{J,J4,J5,J3}, we introduced the finite-dimensional space \(\mathcal{Z}_N( \mathbb{R}) \) of analytic sections
\begin{equation}
Z_{N}(t; \overline{a})
\;=\;
\cos(\theta(t))
+\sum_{k=1}^{N}\frac{a_{k}}{\sqrt{k+1}}\,
\cos\bigl(\theta(t)-t\log(k+1)\bigr),
\end{equation}
parametrized by real coefficients \(a_k \in \mathbb{R}\) and refer to 
\begin{equation}
Z_0(t):=Z_N(t; \bar{0})=\cos (\theta(t))
\end{equation} as the core function of $Z(t)$. Consider:
\begin{dfnnonum}
Let \(\mathcal{RH}_N(\mathbb{R}) \subset \mathcal{Z}_N(\mathbb{R})\) denote the subset of all functions \(Z_N(t;\bar a)\)
for which there exists a continuous homotopy \(a:[0,1]\to\mathbb{R}^N\) with \(a(0)=\bar 0\) and \(a(1)=\bar a\) such that,
for every \(s\in[0,1]\), the zeros of \(Z_N(t;a(s))\) in the critical rectangle
\begin{equation}
2N \le \Re t \le 2N+2,\qquad -\tfrac12 \le \Im t \le \tfrac12
\end{equation}
remain real and simple, and the number of zeros in the rectangle stays constant along the homotopy.
\end{dfnnonum}
The boundary $\partial \mathcal{RH}_N(\mathbb{R}) \subset \mathcal{RH}_N(\mathbb{R})$ is described in terms of the analytical discriminants introduced in \cite{J3}, see Remark \ref{rem:boundary}. It is known that the following approximation holds
\begin{equation} 
\label{eq:approx}
Z(t) = Z_{N(t)}(t ; \overline{1})+ O \left ( \frac{1}{\sqrt[4]{t}} \right ),
\end{equation}
where $N(t) = \left [ \frac{t}{2} \right]$. Although, a priori, the error term could contain essential information about the zeros, we showed in \cite{J,J5,J3} that, unlike the classical Riemann–Siegel formula, the approximation \eqref{eq:approx} is sufficiently sensitive and the Riemann Hypothesis is essentially equivalent to showing that
\begin{equation}
Z_N(t;\overline{1})\in\mathcal{RH}_N(\mathbb{R}),
\end{equation}
for all $N \in \mathbb{N}$, refer to Remark \ref{rem:window-flexibility} regarding the flexibility of window boundaries.  

Let $\widetilde{\Phi}_N:\mathcal{RH}_N(\mathbb{R})\to\mathbb{R}^{M_N}$ denote the map sending a section $Z_N(t;\bar a)$ to its ordered unfolded zero configuration in the
window $[2N,2N+2]$, that is 
\begin{equation}
\widetilde{\Phi}_N(\bar a):=\bigl(\widetilde t_1(\bar a),\dots,\widetilde t_{M_N}(\bar a)\bigr),
\end{equation}
normalized by 
\begin{equation}
\widetilde t_j(\bar a):=\frac{\log N}{2\pi}\,t_j(\bar a)
\end{equation}
where $t_1(\bar a)<\cdots<t_{M_N}(\bar a)$ are the zeros of $Z_N(t;\bar a)$ in $[2N,2N+2]$. We will refer to a probability measure $\mu_N$ on $\mathcal{RH}_N(\mathbb{R})$ as admissible if
there exist constants $c,C>0$, independent of $N$, such that for every $\bar a_0\in\mathcal{RH}_N(\mathbb{R})$,
\begin{equation}\label{eq:admissible}
\mu_N\!\left(\left\{\bar a\in\mathcal{RH}_N(\mathbb{R}) :
\bigl\|\widetilde{\Phi}_N(\bar a)-\widetilde{\Phi}_N(\bar a_0)\bigr\|\le N^{-C}\right\}\right)
\le N^{-c},
\end{equation}
with respect to the Euclidean norm  $\|\cdot\|$ on $\mathbb{R}^{M_N}$. Our first main result in this work is:

\begin{mainthm}[GUE law for the ensemble $\mathcal{RH}_N(\mathbb{R})$]\label{thm:A}
As \(N \to \infty\), the average pair
correlation of the zeros of sections 
\(Z_N(t;\bar a) \in \mathcal{RH}_N(\mathbb{R})\), taken with respect to 
any admissible probability measure $\mu_N$ on
$\mathcal{RH}_N(\mathbb{R})$ with a smooth, strictly positive density
with respect to the Lebesgue measure on the coefficient space, restricted
to $\mathcal{RH}_N(\mathbb{R})$, converges to the GUE distribution. That is, for any Schwartz test 
function \(f \in \mathcal{S}(\mathbb{R})\),
\begin{equation}\label{eq:GUE-ensemble}
\lim_{N\to\infty}
\mathbb{E}_{\mu_N}\!\left[
\frac{1}{M_N}
\sum_{\substack{t_j,t_k\in[2N,2N+2]\\ j\ne k}}
f\!\left(\frac{\log N}{2\pi}\bigl(t_j(\bar a)-t_k(\bar a)\bigr)\right)
\right]
=
\int_{\mathbb{R}} f(x)
\left(1-\left(\frac{\sin\pi x}{\pi x}\right)^2\right)\,dx,
\end{equation}
where \(M_N \sim \tfrac{1}{\pi}\log\!\bigl(\tfrac{N}{\pi}\bigr)\) is the 
number of zeros \(\{t_j(\bar a)\}\) of \(Z_N(t;\bar a)\) in \([2N,2N+2]\).
\end{mainthm}

Theorem \ref{thm:A} provides the first purely analytic framework in which random matrix behavior emerges intrinsically from the theory of the \(Z\)-function itself. To illustrate this phenomenon, Fig.~\ref{fig:f1} displays the histogram of nearest-neighbor spacings between the $164$ zeros of a section \( Z_{N}(t;\bar{a}) \) with \( N = 300 \), computed in the interval \( 755 \le t \le 965 \), whose coefficients \( \bar{a} \in \mathbb{R}^N \) were chosen randomly within the real submanifold \( \mathcal{RH}_N(\mathbb{R}) \), ensuring that all zeros in the region are real. Superimposed in red is the Wigner surmise for the Gaussian Unitary Ensemble,
\begin{equation}
P_{\mathrm{GUE}}(s) = \frac{32}{\pi^2}s^2 e^{-4s^2/\pi},
\end{equation}
which is known to approximate the universal spacing distribution of eigenvalues in large Hermitian matrices. The close agreement between the empirical spacing histogram and \( P_{\mathrm{GUE}}(s) \) illustrates that the zeros of \( Z_N(t;\bar{a}) \) display behaviour anticipated for GUE-type spacings, as guaranteed by Theorem~\ref{thm:A}.

\begin{figure}[ht!]
  \centering
  \includegraphics[scale=0.5]{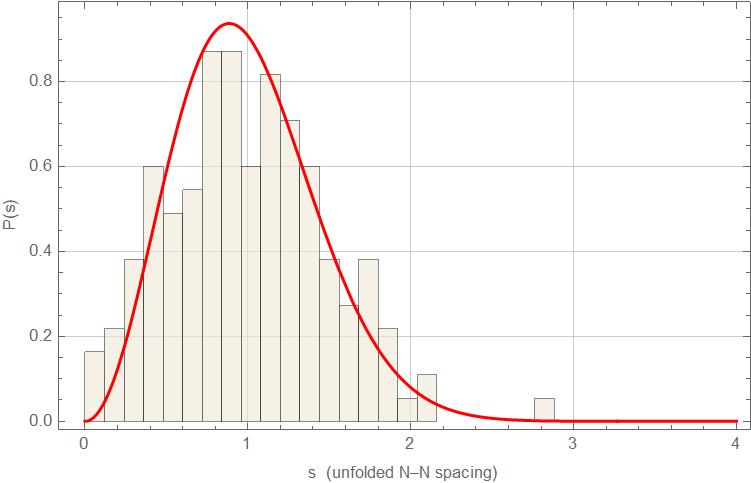}
  \caption{\small Histogram of nearest-neighbour spacings of the $164$ zeros of
    $Z_{N}(t;\bar{a})$, with  $N=300$, in
    $755\le t\le 965$ where $\bar{a} \in \mathbb{R}^N$ is randomized so that its zeros in this region are real, and the GUE Wigner surmise 
    $P_{\text{GUE}}(s)=\tfrac{32}{\pi^{2}}s^{2}e^{-4s^{2}/\pi}$
    (red). 
  }%
  \label{fig:f1}
\end{figure}

The key mechanism underlying Theorem~\ref{thm:A} is the introduction of a Skorokhod-type stochastic differential equation (SDE) flow on the real hall \(\mathcal{RH}_N(\mathbb{R})\), with reflection at the boundary, which drives a diffusion of the section parameters \(\overline{a}(t)\) while keeping the flow inside \(\mathcal{RH}_N(\mathbb{R})\). This SDE induces, via the analytic dependence of zeros on the coefficients, a coupled system of SDEs for the deformed zeros \(t_n(\overline{a}(t))\). 

This construction is directly analogous to the classical random matrix setting, where one considers an SDE on the space of Hermitian matrices and studies the induced dynamics of the eigenvalues. In that case the eigenvalue process is Dyson Brownian motion (DBM) with parameter \(\beta=2\), and a central result of Dyson~\cite{Dyson} shows that DBM is ergodic with the GUE eigenvalue distribution as its unique invariant measure, with convergence to equilibrium as time tends to infinity. The proof of Theorem~\ref{thm:A} rests on showing that the induced zero-dynamics on \(\mathcal{RH}_N(\mathbb{R})\), after unfolding and whitening, share the same asymptotic local properties as DBM.

To pass from the obtained structural DBM on $\mathcal{RH}_N(\mathbb{R})$ to the full GUE sine–kernel law, we appeal to the modern universality theory for Dyson Brownian motion and related log–gases. Whereas Dyson’s original work~\cite{Dyson} showed that DBM is ergodic and converges in law to the global GUE equilibrium for a specific matrix Ornstein–Uhlenbeck process, later developments, due to Erd{\H{o}}s, Schlein, Yau, Bourgade and others, established that the local correlation functions of DBM at fixed energy are universal. That is, for a wide class of initial distributions and potentials, the eigenvalues rapidly acquire GUE local statistics essentially independent of the initial law, see ~\cite{AGZ,BourgadeErdosYau2014,DeiftBook,ErdosSchleinYau2011,ErdosYau2012,ForresterBook,LandonSosoeYau2017,PB}. Applying these universality results to the zero–dynamics on $\mathcal{RH}_N(\mathbb{R})$, once coupled at the microscopic scale to a suitable DBM, we obtain that the local pair–correlation functionals $PC_N(f;\bar a)$ of $\mu_N$–random sections converge to the sine–kernel law stated in Theorem~\ref{thm:A}. Consequently, the ensemble of zeros of a $\mu_N$-random section in $\mathcal{RH}_N(\mathbb{R})$ exhibits the same GUE sine–kernel pair–correlation behavior in the window $[2N,2N+2]$, resolving the microscopic question.

\subsection{The Macroscopic Question - the Ergodicity of the Pair-Correlations of $Z_N(t; \bar 1)$}
Theorem~\ref{thm:A} establishes GUE pair--correlation at the level of the real hall ensemble
\(\mathcal{RH}_N(\mathbb{R})\), but by itself this does not yet say anything about any
\emph{specific} element of the ensemble.
In fact, in view of Theorem~\ref{thm:A}, the Pair--Correlation Conjecture can now be reformulated as a question about the realization of the Hardy function within the real halls \(\mathcal{RH}_N(\mathbb{R})\), for \(N\in\mathbb{N}\). Namely, whether the GUE universality established for the ensemble on \(\mathcal{RH}_N(\mathbb{R})\) is in fact realized by the specific, fixed analytic function \(Z(t)\) itself. In light of the approximation \eqref{eq:approx}, we regard \(Z(t)\) on each window \([2N,2N+2]\) as being represented by the canonical section
$Z_N(t;\overline{1}) \in \mathcal{RH}_N(\mathbb{R})$, so that the problem becomes one of understanding the statistical behavior of the deterministic sequence of sections \(\{Z_N(t;\overline{1})\}_{N\in\mathbb{N}}\) inside the ensembles \(\mathcal{RH}_N(\mathbb{R})\), for $N \in \mathbb{N}$.

Before considering the function \(Z(t)\) itself, we introduce a randomized analogue 
defined on the whole real line.  
For each \(N \in \mathbb{N}\), let 
\(Z_N(t; \bar a^{(N)}) \in \mathcal{RH}_N(\mathbb{R})\) 
be chosen independently according to the Gaussian measure \(\mu_N\) 
restricted to \(\mathcal{RH}_N(\mathbb{R})\).  
We define the \emph{randomized $Z$--function} $Z(t ; \{ \bar a^{(N)} \})$ by setting 
\begin{equation}
Z(t ; \{ \bar a^{(N)} \}) = Z_N(t; \bar a^{(N)}),
\end{equation}
for $t \in [2N,2N+2]$, so that \(Z(t ; \{ \bar a^{(N)} \})\) coincides with the section 
\(Z_N(t; \bar a^{(N)})\) on each corresponding window. Theorem~\ref{thm:B1} upgrades the averaged GUE result for sections
\(Z_N(t;\overline{a}) \in \mathcal{RH}_N(\mathbb{R})\) on individual windows
\([2N,2N+2]\) to a genuine GUE pair--correlation law for a single realization of a
randomized $Z$--function on the entire real line.

\begin{mainthm}[GUE Pair--Correlation for randomized $Z$--functions]
\label{thm:B1}
Let \(Z(t ; \{ \bar a^{(N)} \})\) be a randomized $Z$--function as above.  
Then the real zeros of \(Z(t ; \{ \bar a^{(N)} \})\) obey the 
Gaussian Unitary Ensemble (GUE) pair--correlation law.  
That is, for every Schwartz test function \(f \in \mathcal{S}(\mathbb{R})\),
\begin{equation}\label{eq:PCC-thmB}
\lim_{T \to \infty}
\frac{1}{N(T)}
\!\!\sum_{\substack{0<\gamma,\gamma'\le T\\\gamma\ne\gamma'}}
f\!\left(
\frac{\log(T/2\pi)}{2\pi}\,(\gamma-\gamma')
\right)
=
\int_{\mathbb{R}} 
f(x)
\!\left(
1-\left(\frac{\sin\pi x}{\pi x}\right)^{\!2}
\right)
dx,
\end{equation}
where \(\gamma\) and \(\gamma'\) denote the real zeros of 
\(Z(t ; \{ \bar a^{(N)} \})\), and 
\(N(T)\sim \tfrac{T}{2\pi}\log\!\tfrac{T}{2\pi}\) 
is their counting function.
\end{mainthm}

   Finally, the following theorem provides an affirmative answer and constitutes a proof of Montgomery’s original conjecture. It is proven by showing that the Hardy 
\(Z\)–function constructed from the deterministic sections 
\(\overline{a}^{(N)} = \overline{1}\) reproduces, in the limit, essentially the same  
pair–correlation statistics as the randomized 
\(Z\)–functions considered in Theorem~\ref{thm:B1}. 

\begin{mainthm}[Montgomery Pair--Correlation Theorem for Hardy’s $Z$--function]\label{thm:B}
Assuming the Riemann Hypothesis, the nontrivial zeros of Hardy’s $Z$--function 
obey the Gaussian Unitary Ensemble (GUE) pair--correlation law.  
Equivalently, Montgomery’s Pair--Correlation Conjecture holds.  
That is, for every Schwartz test function $f \in \mathcal{S}(\mathbb{R})$,
\begin{equation}\label{eq:PCC-thmB}
\lim_{T \to \infty}
\frac{1}{N(T)}
\!\!\sum_{\substack{0<\gamma,\gamma'\le T\\\gamma\ne\gamma'}}
f\!\left(
\frac{\log(T/2\pi)}{2\pi}\,(\gamma-\gamma')
\right)
=
\int_{\mathbb{R}} 
f(x)
\!\left(
1-\left(\frac{\sin\pi x}{\pi x}\right)^{\!2}
\right)
dx,
\end{equation}
where $\gamma$ and $\gamma'$ denote the real zeros of~$Z(t)$, 
and $N(T)\sim \tfrac{T}{2\pi}\log\!\tfrac{T}{2\pi}$ is their counting function. 
\end{mainthm}

In the classical random matrix setting, a single
Hermitian matrix need not exhibit GUE--like statistics at all. For example, a diagonal
matrix with integer entries has a completely rigid spectrum with evenly spaced eigenvalues,
far from the sine--kernel behavior. An analogous phenomenon occurs in our analytic
framework, where the core section \(Z_0(t)=\cos\theta(t)\in\mathcal{RH}_N(\mathbb{R})\) has zeros
on each window \([2N,2N+2]\) that are essentially uniformly spaced, with no visible level
repulsion. Thus, in order to prove Theorem~\ref{thm:B} for the Hardy \(Z\)–function itself,
we must show that \(Z(t)\), represented on each window by the canonical section
\(Z_N(t;\overline{1})\), lies on the “random’’ side of the ensemble, similar to the randomized sections of Theorem \ref{thm:B1}, rather than behaving
like such singular, highly structured cases. 

This is achieved by combining Selberg’s probabilistic theory of the argument function
\cite{Selberg1946,Selberg1947,Selberg1989CLT,SelbergCollectedPapers1989},
\begin{equation}
S(t)=\frac{1}{\pi}\arg\zeta\Bigl(\tfrac12+it\Bigr),
\end{equation}
with the microscopic universality established in Theorem~\ref{thm:A}. We refer to 
\begin{equation}
\mathrm{PC}_N(f;\overline{a})
:=\frac{1}{M_N}\!\!\sum_{\substack{t_j(\overline{a}),\,t_k(\overline{a})\in[2N,2N+2]\\ j\neq k}}
f\!\left(\frac{\log N}{2\pi}\bigl(t_j(\overline{a})-t_k(\overline{a})\bigr)\right),
\end{equation}
as the local pair–correlation functional of $Z_N(t; \overline{a})$ in the window $[2N,2N+2]$. 
Assuming RH, we first use the approximation
\begin{equation}
Z(t)\approx Z_N(t;\overline{1}) \in \mathcal{RH}_N(\mathbb{R})
\qquad t\in[2N,2N+2],
\end{equation}
to express the global Montgomery pair–correlation sum of the Hardy $Z$–function,
\begin{equation}
PC^{Z}(f; T)
:=\frac{1}{N(T)}
\!\!\sum_{\substack{\gamma,\gamma' \leq T\\ \gamma\ne\gamma'}}
f\!\left(\frac{\log(T/2\pi)}{2\pi}\,(\gamma-\gamma')\right),
\end{equation} 
as an average of local pair–correlation
functionals \(PC_N(f;\overline{1})\) over successive windows.
 
Selberg’s $L^2$ and covariance
estimates for \(S(t)\) and its increments then imply that, on the scale of these windows,
the fluctuations of \(S(t)\) behave like a Gaussian process with rapidly decaying correlations
between well–separated windows. Using Selberg's theory, we prove in Theorem~\ref{thm:Sel-dec} that for every fixed Schwartz test function $f$, the sequence $\{PC_N(f;\overline{1})\}_{N \in \mathbb{N}}$ has uniformly bounded variance and
satisfies the weak decorrelation condition
\begin{equation}
\label{eq:Sel-dec}
\mathrm{Cov}\!\left(PC_N(f;\overline{1}),PC_{N'}(f;\overline{1})\right)\to 0,
\end{equation}
as $|N-N'|\to\infty$. This gives a law of large numbers–type result for the
sequence \(\{PC_N(f;\overline{1})\}_{N\in\mathbb{N}}\): after suitable normalization, their average
converges to its ensemble mean as \(N\to\infty\).

Combining this macroscopic decorrelation with the microscopic GUE universality of
Theorem~\ref{thm:A},
we conclude that the deterministic canonical sections \(Z_N(t;\overline{1})\) have the same
limiting pair–correlation statistics as a “typical’’ random section in \(\mathcal{RH}_N(\mathbb{R})\), for $N \in \mathbb{N}$.
In this sense the sequence \(\{Z_N(t;\overline{1})\}_{N\in\mathbb{N}}\) behaves ergodically inside the
ensembles, and the global pair–correlation of the zeros of Hardy’s \(Z\)–function coincides
with the GUE law, proving Montgomery’s conjecture.

\subsection{Organization of the Paper} The rest of the paper is organized as follows.
\begin{enumerate}
  \item \emph{Review of GUE and Dyson’s 1D Coulomb gas.} In Section~\ref{s:2} we recall the GUE joint eigenvalue law, the Coulomb–gas interpretation, Dyson Brownian motion (DBM) and universality theory, fixing notation used later for correlation statistics and stochastic evolutions.

  \item \emph{Skorokhod SDE on $\mathcal{RH}_N(\mathbb{R})$ and induced zero dynamics.} In Section~\ref{s:4} we recall the $A$-variation space $\mathcal{Z}_N(\mathbb{R})$ and introduce $\mathcal{RH}_N(\mathbb{R}) \subset \mathcal{Z}_N(\mathbb{R})$. We define the reflected Skorokhod SDE on $\mathcal{RH}_N(\mathbb{R})$, derive the Itô dynamics for the real zeros in the window in \eqref{eq:zeros-Ito}, and describe the boundary/reflection geometry.

  \item \emph{Quadratic co–variation of first–order noises.} Section~\ref{s:5} computes the bracket matrix of the first–order terms $\beta^{(n)}_t$ and shows, in Theorem~\ref{prop:beta-bracket-asympt}, that off–diagonal covariations are asymptotically negligible, yielding an almost–diagonal driver on fixed time windows. This leads to the normalization $\hat{\beta}^{(n)}_t$ and the normalized SDE \eqref{eq:Xtilde-SDE-1}:
  \begin{equation}
  \label{eq:Xtilde-SDE-2}
  \begin{aligned}
  d\widetilde X_n(t)
  &= d\hat\beta^{(n)}_t
     + \frac{1}{2\big\|\nabla t_n(A_t)\big\|}\,
       \big\langle dB_t,\; \nabla^2 t_n(A_t)[\,dB_t\,]\big\rangle
     + b^{\mathrm{Sko}}_n(t)\,dt,
  \end{aligned}
  \end{equation}
  where $b^{\mathrm{Sko}}_n(t)\,dt=\frac{1}{\|\nabla t_n(A_t)\|}\,\langle \nabla t_n(A_t),\,dL_t\rangle$ is the Skorokhod reflection term.

  \item \emph{The Hadamard product and Coulomb–type repulsion.} In Section~\ref{s:6} we analyze the second–order Itô correction $\langle dB_t,\,\nabla^2 t_n(A_t)[dB_t]\rangle$ and show, in Proposition~\ref{prop:Ito-from-6.2}, that it produces a singular Coulomb-type drift with state-dependent metric weights. 
This leads to the representation of the SDE \eqref{eq:Xtilde-SDE-2}  stated in Corollary~\ref{prop:Xtilde-SDE}, which is later written as
  \begin{equation}\label{eq:16}
d\widetilde X_n
= d\hat\beta^{(n)}_t
+ \sum_{m \neq n}\frac{1}{\widetilde X_n-\widetilde X_m}\,dt+ \left [ b^{reg}_n+b^{err}_n+ b_n^{Sko}+b^{1-body}_n \right ] dt.
\end{equation}
in \eqref{eq:DBM-reg-err} with the terms $ b^{reg}_n\, ,\, b^{err}_n \, ,\, b_n^{Sko}$ and $b^{1-body}_n$ defined in  Definition~\ref{lem:cutoff-decomp}.

  \item \emph{Asymptotic independence of the normalized noises.} In Section~\ref{s:7} we prove, in Theorem~\ref{cor:levy-hatbeta}, that the normalized first-order noises $\hat{\beta}^{(n)}_t$ are asymptotically independent Brownian motions.

  \item \emph{Universality and reduction to DBM.}
  Section~\ref{s:8} is devoted to reducing \eqref{eq:16} to the classical DBM
  \begin{equation}
  d\widetilde X_n(t)
  = dB^{(n)}_t + \sum_{m\neq n}\frac{1}{\widetilde X_n(t)-\widetilde X_m(t)}\,dt,
  \qquad n\in I_N,
  \end{equation}
  without changing the local statistics. We show via Corollary~\ref{cor:breg-harmless}, using established DBM universality results, that $b^{\mathrm{reg}}$ does not affect bulk local statistics and may therefore be dropped, replace $d\hat\beta_t$ by independent Brownian increments $dB_t$ via the whitening Theorem~\ref{prop:cov-replacement}, and remove $b^{\mathrm{err}}$ with Corollary~\ref{thm:girsanov-clean} and Theorem~\ref{thm:girsanov-handoff}, showing vanishing relative entropy. Finally, Corollary~\ref{cor:A} shows that the Skorokhod term $b_n^{Sko}$ and the one-body term $b^{1-body}_n$ are negligible in the bulk, completing the proof of Theorem~\ref{thm:A}.

\item \emph{GUE for randomized $Z$--functions.}  
In Section~\ref{s:9.1} we introduce both the local and global pair--correlation functionals and establish Theorem~\ref{thm:B1}, which proves that the zeros of randomized $Z$--functions 
$Z_N(t;\{\bar a^{(N)}\})$ obey the GUE pair--correlation law.  
We further explain the conceptual transition from the annealed setting of Theorem~\ref{thm:A}, where the averaging is taken over the analytic ensemble 
$\mathcal{RH}_N(\mathbb{R})$, to the quenched setting of Theorems~\ref{thm:B1} and~\ref{thm:B}, which concern fixed randomized and deterministic realizations.  
 
\item \emph{Ergodic realization of the Pair--Correlation Conjecture.} 
In Section~\ref{s:9} we prove
Theorem~\ref{thm:B}, completing the proof of Montgomery’s Pair–Correlation Conjecture
for the Hardy function. Selberg’s statistical theory
of $S(t)$ is used to show in Theorem \ref{thm:Sel-dec} that, for each fixed Schwartz test function $f$, the
sequence $\{PC_N(f;\overline{1})\}_{N\in\mathbb{N}}$ has uniformly bounded variance and
satisfies a weak decorrelation property, giving a law of large numbers–type
convergence of their empirical averages to the ensemble mean. Combining this
macroscopic decorrelation with the microscopic GUE universality of
Theorem~\ref{thm:A}, we conclude that the zeros of Hardy’s $Z$–function itself
satisfy the GUE pair–correlation law.

 \item \emph{Summary and concluding remarks.} 
Section~\ref{s:10} reviews the main logical progression 
(zeros $\;\Rightarrow\;$ SDE $\;\Rightarrow\;$ DBM universality 
$\;\Rightarrow\;$ ensemble GUE of $\mathcal{RH}_N$ 
$\;\Rightarrow\;$ GUE for $Z(t)$), 
and discusses the resulting consequences and open directions.
\end{enumerate}

\section{Review of GUE Eigenvalue Statistics and Universality for Dyson's 1D Coulomb Gas Model}
\label{s:2} 

In this section, we review fundamental concepts from the theory of GUE eigenvalue statistics and Dyson's one-dimensional Coulomb gas model, see \cite{M,Dyson,PB,AGZ} for comprehensive treatments. While these constructions are classical within the framework of random matrix theory, they serve as benchmark analogues for the deformations studied in our variation space of \( Z(t) \). This section thus also provides the technical groundwork necessary for developing the stochastic mechanisms that emerge in our analytic number theory framework in the subsequent sections.

\subsection{Random Matrices and 1D Coulomb Gas}
\label{ss:2.1}
Let
\begin{equation}
\mathcal H_M \;=\;
\Bigl\{\,H\in\operatorname{Mat}_{M}(\mathbb C)\;\bigm|\;  H_{ji} = \overline{H_{ij}}  \Bigr\}
\end{equation}
be the real vector space of Hermitian \(M \times M \) matrices, equipped with the inner product
\begin{equation}
\langle A,B\rangle=\operatorname{Tr}(AB).
\end{equation}  
We say that \( H \) is a random \( M \times M \) Hermitian matrix drawn from the Gaussian Unitary Ensemble, if its entries are defined as follows:
\begin{enumerate}
\item The diagonal entries \( H_{ii} \) are independent normally distributed real Gaussians \( \mathcal{N} \left (0, 1 \right ) \) with mean zero and variance \(1\). 

\item The off-diagonal entries \( H_{ij} \) for \( i < j \) are independent complex Gaussians with real and imaginary parts \( \mathcal{N} \left (0, \tfrac{1}{2} \right ) \).
\end{enumerate}
The resulting probability density on the space of Hermitian matrices is
\begin{equation}
P(H) = (2\pi)^{-M^2/2}
\exp\!\left(-\frac{1}{2}\,\operatorname{Tr}(H^2)\right),
\end{equation}
defined with respect to the Lebesgue measure on the independent real and imaginary parts of the matrix entries.

The eigenvalues \( \lambda_1, \dots, \lambda_M \) of \( H \) are real and can be studied via a change of variables from matrix entries to eigenvalues and eigenvectors. Integrating over the Haar measure on the unitary group gives the joint probability density function (JPDF) for the eigenvalues:
\begin{equation}
\label{eq:stat}
P(\lambda_1, \dots, \lambda_M)
= \frac{1}{\hat{Z}_M}
  \prod_{1 \le i < j \le M} |\lambda_i - \lambda_j|^{2}
  \exp\!\left(-\frac{1}{2}\sum_{i=1}^M \lambda_i^{2}\right),
\end{equation}
where the normalization constant is given explicitly by
\begin{equation}
\hat{Z}_M
= (2\pi)^{M/2}\,\prod_{j=1}^M j! .
\end{equation}
The expression \eqref{eq:stat} has a natural interpretation in terms of statistical mechanics. Define the associated energy functional
\begin{equation}
\label{eq:CoulombEnergy}
E(\lambda_1, \dots, \lambda_M)
= - \sum_{1 \le i < j \le M} \log |\lambda_i - \lambda_j|
  + \frac{1}{4}\sum_{i=1}^M \lambda_i^2 .
\end{equation}
Then the JPDF takes the form of a Boltzmann distribution at inverse temperature \( \beta = 2 \):
\begin{equation}
P(\lambda_1, \dots, \lambda_M) = \frac{1}{Z_M}\, e^{-\beta\, E(\lambda_1, \dots, \lambda_M)}.
\end{equation}
This viewpoint, introduced by Dyson~\cite{Dyson}, reveals a deep analogy with statistical physics, according to which the eigenvalues of a GUE matrix behave like a one-dimensional Coulomb gas, a system of $M$ identical charged particles confined to the real line, repelling one another via a logarithmic potential and held in place by an external quadratic potential. The repulsion term \( -\log|\lambda_i - \lambda_j| \) corresponds to the two-dimensional Coulomb interaction, since \( \log|x| \) is the Green's function of the Laplacian in two dimensions, while the confining term \( \tfrac{1}{4}\sum \lambda_i^2 \), equivalently, the quadratic weight in \eqref{eq:stat}, prevents the spectrum from escaping to infinity. The interplay between repulsion and confinement gives an equilibrium in which the eigenvalues are neither rigid nor independent, but exhibit universal statistical correlations characteristic of GUE statistics, such as the sine kernel~\eqref{eq:sine}, appearing in Montgomery’s pair correlation conjecture.

\subsection{Itô’s Calculus and Dyson Brownian Motion for Hermitian Matrices}
Let $H(t)\in\mathcal H_M$ evolve from $H(0)$ by
\begin{equation}\label{eq:DBM-matrix}
  dH(t)=dB(t),
\end{equation}
where $B(t)\in\mathcal H_M$ is a Hermitian matrix Brownian motion (BM) with independent
increments, Hermitian symmetry $dB_{ji}(t)=\overline{dB_{ij}(t)}$, and quadratic
variation
\begin{equation}\label{eq:BM-cov}
  \mathbb E\!\left[dB_{ij}(t)\,\overline{dB_{kl}(t)}\right]
  = \delta_{ik}\delta_{jl}\,dt,
  \qquad
  \mathbb E\!\left[dB_{ij}(t)\,dB_{kl}(t)\right]=0,
\end{equation}
for all $i\le j,\;k\le l$. Equivalently, the diagonal entries are real standard BMs and the off-diagonals have
independent real/imaginary parts with variance $dt/2$. From \eqref{eq:BM-cov},
for all $1\le i,j\le M$,
\begin{equation}\label{eq:entry-variance}
  \mathbb E\Bigl[\,\bigl|H_{ij}(t)-H_{ij}(0)\bigr|^{2}\Bigr]=t.
\end{equation}
 For each fixed $t$ let
\begin{equation} 
H(t)=U(t)\,\Lambda(t)\,U(t)^\dagger,
\qquad
\Lambda(t)=\operatorname{diag} \!\bigl(\lambda_1(t),\dots,\lambda_M (t)\bigr)
\end{equation} 
be an eigen-decomposition with $U(t)$ unitary. If the ordered eigenvalues of \(H\in\mathcal H_M\) are
\(\lambda_1(H)\le\cdots\le\lambda_M(H)\), call the spectrum
\emph{simple} when the inequalities are strict.
\begin{equation}
\mathcal H_M^{\mathrm{simp}}
 \;=\;
 \bigl\{H\in\mathcal H_M
        \;\bigm|\;
        \lambda_1(H)<\cdots<\lambda_M(H)\bigr\}.
\end{equation}
Its complement
\(\Sigma_{\mathrm{deg}}=\mathcal H_M\setminus\mathcal H_M^{\mathrm{simp}}\)
has (real) codimension\,\(\ge 2\) and Lebesgue measure \(0\).
On \(\mathcal H_M^{\mathrm{simp}}\) the ordered eigenvalues are
\emph{smooth} functions; at points of \(\Sigma_{\mathrm{deg}}\) they are
only Lipschitz and lack derivatives.

For a twice continuously differentiable scalar function $F: \mathcal{H}_M \rightarrow \mathbb{R} $ Itô’s formula for \(\mathcal H_M\)–valued semimartingales reads
\begin{equation}
\label{eq:Ito-der}
  dF(H_t)
    =\;\bigl\langle\nabla F(H_t),\,dH_t\bigr\rangle
     \;+\;
     \frac{1}{2}\,
     \bigl\langle dH_t,\;\nabla^{2}F(H_t)[\,dH_t\,]\bigr\rangle
\end{equation}
where the second inner product uses the fact that
\((dB_t)^2=dt\). Let us set 
\begin{equation} F(H)=\lambda_n(H). 
\end{equation} 
Let \(v_n(H)\) be the normalised eigenvector of \(H\) for \(\lambda_n(H)\).
Classical Rayleigh–Schrödinger theory gives the following deterministic perturbation formulas
\begin{equation}
\label{eq:RSpert}
\nabla\lambda_n(H)=P_n,
\qquad
\nabla^{2}\lambda_n(H)[X]
 =2\!\sum_{m\neq n}
   \frac{\langle v_m,X v_n\rangle}{\lambda_n-\lambda_m}\,
   \langle v_n,X v_m\rangle ,
\end{equation}
where 
\(P_n:=v_n v_n^{\!\dagger}\), is the rank-one noise operator. 
Inserting \eqref{eq:RSpert} into
Itô's formula \eqref{eq:Ito-der} gives
\begin{equation}
\label{eq:Ito-matrix}
d\lambda_n(t)
=\Bigl\langle P_n,\,dB_t\Bigr\rangle
   +
   \sum_{m\neq n}
     \frac{|\,\langle v_m,dB_t v_n\rangle|^{2}}
          {\lambda_n-\lambda_m}.
\end{equation}
The two components of \eqref{eq:Ito-matrix} satisfy: 
\begin{enumerate}
\item The diagonal pieces
\begin{equation}
\label{eq:BnRMT}
d\widetilde{\beta}^{(n)}_t:=\Bigl\langle P_n, \,dB_t\Bigr\rangle =\langle v_n,dB_tv_n\rangle
\end{equation}
are independent standard Brownian motions thanks to the covariance
choice, see Remark \ref{prop:5.1} below for further details. 
\item  For $m\neq n$ the off–diagonal bracket
      \(
        \big\langle v_m,\,dB_t\,v_n\big\rangle
      \)
      is a centred complex Gaussian increment whose conditional second moment is given by 
      \begin{equation}
        \mathbf E\!\Bigl[\,
          \bigl|\!\bigl\langle v_m,\,dB_t\,v_n\bigr\rangle\bigr|^{2}
          \,\bigm|\,\mathcal F_t
        \Bigr]
        \;=\;dt,
      \end{equation}
      where 
      \begin{equation}
\mathcal{F}_t
\;:=\;
\sigma\!\bigl(\,H_s : 0\le s\le t\bigr)
\end{equation}
is the natural filtration of the matrix process $H_t$, see Remark \ref{rem:RMT2} for further details. 
\end{enumerate}

We have thus obtained Dyson Brownian motion (DBM) coupled system of SDE
\begin{equation}
\label{eq:DBM}
  d\lambda_n(t)
   = d\widetilde{\beta}^{(n)}_t
    + \sum_{m\neq n}
        \frac{1}{\lambda_n(t)-\lambda_m(t)}\,dt,
  \quad n=1,\dots,M.
\end{equation}
For the eigenvalue vector 
\(\lambda(t)=(\lambda_1(t),\dots,\lambda_M(t))\)
the coupled SDE \eqref{eq:DBM} induces the following Fokker–Planck equation for its density
\(p_t(\lambda)\):
\begin{equation}
\partial_t p_t(\lambda)
  \;=\;\frac{1}{2}\sum_{n=1}^M\partial_{\lambda_n}^2 p_t(\lambda)
  \;-\;\sum_{n=1}^M \partial_{\lambda_n}
        \!\left[\left(\sum_{m\ne n}\frac{1}{\lambda_n-\lambda_m}\right)
              p_t(\lambda)\right].
\end{equation}

 If \(H(0)= c \cdot I\), for $c \in \mathbb{R}$, then for each fixed \(t>0\) the law of \(H(t)\) is Gaussian with variance \(t\), and the eigenvalue density is
\begin{equation}
\label{eq:pt-GUE}
p_t(\lambda_1,\ldots,\lambda_M)
 \;=\; \frac{1}{Z_{M,t}}\,
       \prod_{1\le i<j\le M}|\lambda_i-\lambda_j|^{2}
       \exp\!\left(-\frac{1}{2t}\sum_{k=1}^M(\lambda_k-c)^{2}\right),
\end{equation}
with \(Z_{M,t}=(2\pi t)^{M/2}\prod_{j=1}^M j!\).
In particular, after recentring and unfolding by the instantaneous density, the bulk correlation statistics are known to be given by the GUE sine kernel.

\begin{rem}[Collisions occur with probability zero] \label{rem:2}
Because the repulsion term blows up as
\(\lambda_n\to\lambda_m\), the eigenvalues cannot cross and the collision set is polar.
Formally, starting from a simple spectrum,
\begin{equation}
\mathbb{P}\left(
  \text{there exists } t > 0 \text{ such that } \lambda_n(t) = \lambda_m(t)
\right) = 0, 
\qquad \text{for } n \ne m.
\end{equation}
so the DBM path remains in the smooth chamber \(\mathcal H_M^{\mathrm{simp}}\) for all \(t\ge0\).
Thus the derivation above is valid almost surely and the matrix Itô
lemma provides a rigorous bridge from Brownian motion on
\(\mathcal H_M\) to the Dyson SDE without ever encountering
non-differentiable points.
\end{rem}

\subsection{Modern universality theory for DBM} \label{ss:univ}
Dyson’s original analysis of \eqref{eq:DBM} identifies \eqref{eq:pt-GUE} as the unique invariant density and shows that, for the matrix Ornstein–Uhlenbeck dynamics started from a Gaussian law, the eigenvalue process is ergodic and converges in distribution to the global GUE equilibrium~\cite{Dyson}. In recent years this picture has been vastly extended by the universality theory for Dyson Brownian motion and related log–gases. It is now known that, for a very broad class of initial eigenvalue configurations and confining potentials, the short–time evolution of \eqref{eq:DBM} drives the system to GUE local statistics at fixed energy, essentially independent of the microscopic details of the initial law. In particular, the bulk $k$–point correlation functions and gap distributions converge to those of the sine–kernel determinantal process after a relaxation time that is much shorter than the global equilibration time. See, for example, the works of Erd{\H{o}}s–Schlein–Yau and Bourgade–Erd{\H{o}}s–Yau~\cite{AGZ,BourgadeErdosYau2014,DeiftBook,ErdosSchleinYau2011,ErdosYau2012,ForresterBook,LandonSosoeYau2017,PB}.

In particular, let
\(\lambda^{(N)}_1,\dots,\lambda^{(N)}_{M_N}\)
denote the eigenvalues in a fixed bulk window for an $M_N\times M_N$ log–gas or, equivalently, the solution of \eqref{eq:DBM} at a fixed time $t=t_N>0$, with initial law $\mu_N$ satisfying the usual regularity and rigidity assumptions. Define the eigenvalue pair–correlation functional
\begin{equation}
PC_N^{\mathrm{eig}}(f)
 \;:=\;
 \frac{1}{M_N}
 \sum_{\substack{j,k:\\ j\neq k}}
 f\!\left(\frac{\log M_N}{2\pi}\,\bigl(\lambda^{(N)}_j-\lambda^{(N)}_k\bigr)\right),
\end{equation}
where $f\in\mathcal{S}(\mathbb{R})$ is a Schwartz function. Then universality theory for Dyson Brownian motion and log–gases asserts that, for any admissible sequence of initial measures $\mu_N$ (in the sense of \eqref{eq:admissible}) and suitable times $t_N$, one has
\begin{equation}
\label{eq:DBM-univ-PC}
\lim_{N\to\infty}\,
\mathbb{E}_{\mu_N}\bigl[PC_N^{\mathrm{eig}}(f)\bigr]
 \;=\;
\int_{\mathbb{R}} f(x)\Bigl(1-\Bigl(\tfrac{\sin\pi x}{\pi x}\Bigr)^2\Bigr)\,dx.
\end{equation}
In words, once a DBM flow has been identified, the unfolded local pair–correlation of eigenvalues converges, in expectation and for any reasonable initial law~$\mu_N$, to the GUE sine–kernel functional. This is precisely the mechanism we shall use in Theorem~\ref{thm:A}. After unfolding and coupling the zero–dynamics on $\mathcal{RH}_N(\mathbb{R})$ to \eqref{eq:DBM}, we may apply \eqref{eq:DBM-univ-PC} with $\mu_N$ any probability measure on $\mathcal{RH}_N(\mathbb{R})$ with smooth, strictly positive density, yielding the GUE pair–correlation limit for $\mathbb{E}_{\mu_N}\bigl[PC_N(f;\bar a)\bigr]$.

DBM universality is viewed as a mixing/relaxation statement in the sense that for positive times $t_N>0$ the flow
locally forgets most details of the initial data and converges to the universal
sine--kernel statistics, provided the initial law is admissible and not excessively concentrated on a
microscopically exceptional set. Without such an admissibility condition assumption, one can choose initial laws supported overwhelmingly near highly rigid
configurations, in which case the pair--correlation at time
$t=0$ is non--GUE and the DBM input \eqref{eq:DBM-univ-PC} does not apply.
Intuitively, the admissibility condition \eqref{eq:admissible} ensures that the DBM evolution explores a genuine local
neighborhood in the microscopic variables, rather than remaining trapped near an
exceptional configuration. Averaging with respect to such laws is then governed by the universal local equilibrium theory.

Let us note that in the context of Theorem \ref{thm:A}, one should typically have in mind measures $\mu_N$ induced by a random choice of coefficients according to a sufficiently regular distribution on the coefficient space, restricted to $\mathcal{RH}_N(\mathbb{R})$, which generically satisfy such admissibility conditions.

\section{The Skorokhod SDE on $\mathcal{RH}_N (\mathbb{R})$ and the Induced SDEs on zeros}
\label{s:4} 

In our previous works~\cite{J,J4,J5,J3}, we introduced a novel analytic framework for the Hardy \( Z \)-function, based on an approximate functional equation (AFE) with exponentially decaying error. We proved that
\begin{equation} 
\label{eq:acc}
Z(t) = \cos(\theta(t)) + \sum_{k=1}^{N(t)} a^{\mathrm{acc}}_{k}(N(t)) \frac{\cos( \theta(t)- t \log(k+1) )}{\sqrt{k+1}} + O \left( e^{- \omega t} \right),
\end{equation} 
where \( N(t) = \left\lfloor \frac{t}{2} \right\rfloor \), and the coefficients \( a^{\mathrm{acc}}_k(N) \) are given explicitly by
\begin{equation}
\label{eq:acc-co}
a^{\mathrm{acc}}_{k}(N) := \sum_{n=k}^{N} \frac{1}{2^{n+1}} \binom{n}{k},
\end{equation}  
with \( \omega > 0 \) a fixed constant.

For any \( N \in \mathbb{N} \), we define \( \mathcal{Z}_N (\mathbb{R}) \) to be the \( N \)-th variation space of the Hardy \( Z \)-function, whose elements are generalized sections of the form
\begin{equation} 
Z_N(t; \overline{a}) = \cos(\theta(t)) + \sum_{k=1}^{N} \frac{a_k}{\sqrt{k+1}} \cos\left( \theta(t) - t \log(k+1) \right),
\end{equation}
with \( \overline{a} = (a_1, \dots, a_N) \in \mathbb{R}^N \). This space \( \mathcal{Z}_N (\mathbb{R}) \) includes two distinguished elements:
\begin{enumerate}
\item The core approximation \( Z_0(t) = \cos(\theta(t)) \) corresponding to \( \overline{a} = \overline{0} \), whose zeros $t_n \in \mathbb{R}$ are fully understood and are given as the unique solutions of 
\begin{equation}
\theta(t_n)=\pi \left ( n +\frac{1}{2} \right ).  
\end{equation}
\item The function \( Z_N^{\mathrm{acc}}(t) := Z_{N}(t; \overline{a}^{\mathrm{acc}}(N)) \), where \( \overline{a}^{\mathrm{acc}} \approx \overline{1} \) as defined in~\eqref{eq:acc-co}, which provides a remarkably accurate approximation to \( Z(t) \) in the range \( 2N \leq t \leq 2N+2 \), as guaranteed by~\eqref{eq:acc}. 
\end{enumerate}
This leads us to define:
\begin{dfn}[The $N$-th real hall] 
Let \(\mathcal{RH}_N(\mathbb{R}) \subset \mathcal{Z}_N(\mathbb{R})\) denote the subset of all functions \(Z_N(t;\bar a)\)
for which there exists a continuous homotopy \(a:[0,1]\to\mathbb{R}^N\) with \(a(0)=\bar 0\) and \(a(1)=\bar a\) such that,
for every \(s\in[0,1]\), the zeros of \(Z_N(t;a(s))\) in the critical rectangle
\begin{equation}
2N \le \Re t \le 2N+2,\qquad -\tfrac12 \le \Im t \le \tfrac12
\end{equation}
remain real and simple, and the number of zeros in the rectangle stays constant along the homotopy.
\end{dfn}

\begin{rem}[$\mathcal{RH}_N(\mathbb{R})$ and the RH] The Riemann Hypothesis can essentially be considered as equivalent to showing $Z_N (t;\overline{a}^{\mathrm{acc}}(N)) \in \mathcal{RH}_N(\mathbb{R})$ for all $N>0$, see Remark \ref{rem:window-flexibility}. It should be noted that this central feature, which enables the construction of our
$\mathcal{RH}_N(\mathbb{R})$ as a meaningful object, stands in sharp contrast to the
classical Riemann–Siegel formula, whose main truncation up to
$\bigl\lfloor \sqrt{t/2\pi}\,\bigr\rfloor$ is known to violate the
Riemann Hypothesis infinitely often. The various distinctions between our approximation
\eqref{eq:acc} and the setting of the classical Riemann–Siegel formula are discussed in detail
in \cite{J,J4,J5,J3}, to which we refer the reader for further background.

\end{rem}

As discussed in Section~\ref{s:2}, when the entries of a Hermitian matrix $H$ are further assumed to evolve stochastically under matrix-valued Brownian motion, Itô's lemma gives the DBM system of stochastic differential equations for the eigenvalues \eqref{eq:DBM}. We now introduce a parallel stochastic framework in our setting of $\mathcal{RH}_N (\mathbb{R}) \subset \mathcal{Z}_N (\mathbb{R})$. 

\begin{rem}[The geometry of the boundary \( \partial \mathcal{RH}_N (\mathbb{R}) \)]\label{rem:boundary}
Let 
\begin{equation}
\mathcal P_N:=\{\,t\in\mathbb C:\;2N\le \Re t\le 2N+2,\;-\tfrac12\le \Im t\le \tfrac12\,\}.
\end{equation} be the critical $N$-th rectangle. The boundary \(\partial\mathcal{RH}_N (\mathbb{R}) \) of \( \mathcal{RH}_N (\mathbb{R}) \subset \mathcal Z_N(\mathbb R)\) decomposes into two natural pieces:
\begin{equation}
\partial\mathcal{RH}_N(\mathbb{R}) \;=\; \Sigma_N \,\cup\, \Gamma_N,
\end{equation}
where
\begin{equation}
\Sigma_N := \bigl\{\,\bar a:\ Z_N(t;\bar a)=0 \text{ and } \partial_t Z_N(t;\bar a)=0 \text{ for some } t\in\mathcal P_N \,\bigr\}.
\end{equation}
is the \emph{discriminant (collision) set} of parameters at which two, or more, zeros in \(\mathcal P_N\) coalesce, introduced and studied in \cite{J3}, and
\begin{equation} 
\Gamma_N := \bigl\{\,\bar a:\ Z_N(t;\bar a)=0\ \text{for some } t\in\partial\mathcal P_N ,\bigr\}.
\end{equation}
is the \emph{entry/exit set} of parameters at which a real zero lies on the boundary of the rectangle, so the number of zeros inside \(\mathcal P_N  \) changes by crossing \(\Gamma_N\). In particular, the regular part of \( 
\partial\mathcal{RH}_N(\mathbb{R}) \) is a smooth hypersurface and the singular locus of higher multiplicities 
\(
\partial\mathcal{RH}^{sing}_N(\mathbb{R}) \subset 
\partial\mathcal{RH}_N(\mathbb{R}) \) is a real–analytic subvariety of codimension at least \(2\).

The local geometry of the boundary \(\partial\mathcal{RH}_N(\mathbb{R})\) was studied in \cite{J5,J3}. When negative coordinates \(a_i<0\) are allowed, additional zeros may “arrive from infinity’’ by crossing the horizontal sides \(\Im t=\pm \tfrac{1}{2}\). In the positive cone \(a\ge 0\), a Rouché’s theorem argument shows that entry/exit can occur only at the endpoints \(t=2N,\,2N+2\) of the real segment inside the rectangle $\mathcal{P}_N$. More concretely, if \(a:[0,1]\to(\mathbb{R}^+)^N\) is a homotopy with \(a(0)=\bar 0\) and \(a(1)=\bar a\), then zeros can leave the real line along the homotopy if and only if a collision between two consecutive real zeros occurs, that is if $a(s) \in \Sigma_N$ for some $s \in [0,1]$. In such a case the pair of zeros can depart off the real axis together. Hence, in the positive cone \(a\ge 0\), no complex zeros can come from “infinity”, therefore the number of zeros in \(\mathcal{P}_N\) can change only via crossings at the real endpoints of \(\mathcal{P}_N\).

\end{rem}

To ensure that the stochastic process remains inside $\mathcal{RH}_N $ for all time, we introduce a reflecting term along the boundary $\partial \mathcal{RH}_N$. This leads naturally to a reflected stochastic differential equation, known as a \emph{Skorokhod SDE}, which describes the evolution of a process that behaves like Brownian motion in the interior of a domain but is \emph{reflected} at the boundary to prevent it from exiting, see \cite{Skorokhod1961,Skorokhod1962,Tanaka1979,
LionsSznitman1984,IkedaWatanabe1981}.

More precisely, if $D \subset \mathbb{R}^N$ is an open domain with smooth boundary $\partial D$, then a reflected Brownian motion in $D$ is a continuous process $A_t$ satisfying
\begin{equation}
dA_t = dB_t + dL_t,
\end{equation}
where $B_t$ is standard Brownian motion in $\mathbb{R}^N$, and $L_t$ is a finite-variation process supported on the boundary $\partial D$ that pushes $A_t$ inward whenever it hits the boundary. The vector $dL_t$ points in the direction of the inward normal to $\partial D$ at the point $A_t$ and increases only when $A_t$ is on the boundary. 

When $D$ has a smooth boundary, this reflection is well-defined and pushes the process back into the domain in the direction of the inward normal vector. 
However, in many applications, including queueing theory, interacting particle systems, and stochastic networks, one encounters domains with non-smooth boundaries, such as polyhedral, convex, or Lipschitz domains, or boundaries with singularities of lower codimension. The analysis of Skorokhod problems has been extended to such settings by various authors, establishing existence, uniqueness, and regularity of reflected processes even when the boundary is non-smooth or includes singular strata, see for instance \cite{Saisho1987, DupuisIshii1991, Williams1995, DaiWilliams1995, BurdzyChenSylvester2004, Grabiner1999}. This extended theory applies to our setting where the domain of interest is $\mathcal{RH}_N \subset \mathcal{Z}_N(\mathbb{R})$. We thus have, by direct application of the multidimensional Itô formula \eqref{eq:Ito-der} to $t_n(A_t)$:

\begin{prop}
Let $A_t$ evolve as reflected Brownian motion in $\mathcal{RH}_N \subset \mathbb{R}^N$, satisfying the Skorokhod SDE:
\begin{equation}
dA_t = dB_t + dL_t, 
\end{equation}
where $B_t$ is standard Brownian motion in $\mathcal{Z}_N$, and $L_t$ is a finite-variation process pushing $A_t$ inward when it reaches the boundary $\mathcal{D}_N = \partial \mathcal{RH}_N$, where the reflection vector is defined almost always along the inward normal to $\mathcal{D}_N \setminus \mathcal{D}_N^{sing}$. Then, 
\begin{equation}
\label{eq:zeros-Ito}
d t_n(A_t) = \langle \nabla t_n(A_t), dB_t \rangle + \langle \nabla t_n(A_t), dL_t \rangle + \frac{1}{2} \langle dB_t, \nabla^2 t_n(A_t)[dB_t] \rangle. 
\end{equation}
where $t_n(A_t)$ is the $n$-th real zero of $Z_N(t; A_t)$ within the window $[2N,2N+2]$. 
\end{prop}
We will show in the following sections that the SDE system for the zeros \(t_n(A_t)\) in \eqref{eq:zeros-Ito} satisfies the following properties:
\begin{enumerate}
\item The first–order martingale terms \(\langle \nabla t_n(A_t),\, dB_t\rangle\) provide the leading diffusive noise and become asymptotically independent.
\item The Itô second–order correction \(\tfrac{1}{2}\,\langle dB_t,\,\nabla^{2} t_n(A_t)[\,dB_t\,]\rangle\) generates a Coulomb–type repulsion, in direct analogy with random matrix theory.
\item The Skorokhod reflection term keeps the process in \(\mathcal{RH}_N\) and prevents collisions of zeros, but its contribution to the resulting SDE is negligible.
\end{enumerate}

Consequently, the system \eqref{eq:zeros-Ito} reduces to Dyson Brownian motion \eqref{eq:DBM} for the GUE eigenvalues \(\lambda_n(H_t)\). Hence, in the limit \(N\to\infty\), the zeros \(t_n(A_t)\) exhibit GUE–type local statistics. In practice we work with the normalized form \eqref{eq:Xtilde-SDE-1} of \eqref{eq:zeros-Ito}, introduced in Section~\ref{s:5}.

\begin{rem}[RMT analogy of $\mathcal{RH}_N (\mathbb{R}) \subset \mathcal{Z}_N(\mathbb{R})$]
\label{rem:3.3}
The open chamber structure of 
\begin{equation}
\mathcal{RH}_N (\mathbb{R}) \subset \mathcal{Z}_N(\mathbb{R}) \subset \mathcal{Z}_N(\mathbb{C}),
\end{equation} where the zeros of $Z_N(t; \bar{a})$ remain real and simple, has a natural analogue in random matrix theory. The complex Ginibre ensemble $\mathrm{Gin}_\mathbb{C}(M)$ consists of all $M \times M$ complex matrices with independent Gaussian entries, see \cite{Gin}. The subset of Hermitian matrices $\mathcal{H}_M$ forms a real-algebraic submanifold of codimension $M^2$. Within this Hermitian chamber, the eigenvalues are always real, generically simple, and vary smoothly with the matrix entries. This is precisely the setting where the classical Dyson Brownian Motion (DBM) is defined. Let us denote
\begin{equation}
\mathcal{CP}_M(\mathbb{R}) = \left\{ A \in \mathbb{C}^{M \times M} \mid p_A(\lambda) \in \mathbb{R}[\lambda] \right\}
\end{equation}
as the set of complex matrices whose characteristic polynomial $p_A(\lambda)$ is real. Thus, just as $\mathcal{H}_M$ forms a spectral chamber within the larger space $\mathcal{CP}_M(\mathbb{R})$, bounded by the discriminant locus where eigenvalues leave the real line or coalesce, the region $\mathcal{RH}_N$ in our model defines a spectral chamber within the full parameter space $\mathcal{Z}_N(\mathbb{R})$, bounded by collisions of real zeros of $Z_N(t; \bar{a})$. In both cases, the analytic control of eigenvalues or zeros, and the associated variational calculus, is valid only within this open stratum.
\end{rem}

\begin{rem}[Analytic ensembles of analytic functions]
In view of the above Remark \ref{rem:3.3} the probabilistic structure on the $A$-variation space $\mathcal{Z}_N(\mathbb{R})$ is related in spirit
to the Gaussian analytic ensembles studied by Sodin and Tsirelson and in later
works on random complex zeroes and Gaussian analytic functions
\cite{HKPV,NazarovSodin,SodinTsirelsonI,SodinTsirelsonII}. In their setting,
one considers Gaussian analytic, or entire, functions and regards the zero set
in $\mathbb{C}$ as a random point field. The resulting local statistics of
zeros are of Ginibre type as after appropriate scaling they are shown to fall into the
same universality class as the complex Ginibre ensemble $\mathrm{Gin}_\mathbb{C}(M)$, i.e. the two–dimensional Coulomb gas with logarithmic interaction.

There is, however, a fundamental difference in scope. For Gaussian analytic
functions, a typical realization has zeros spread throughout $\mathbb{C}$, and
no constraint is imposed that all zeros in a given region lie on the real line.
There is in particular no finite–dimensional parameter domain on which the
zeros in a fixed window are constrained to be real, simple, and ordered.

By contrast, the real hall $\mathcal{RH}_N(\mathbb{R})$ is defined precisely so
that for each section $Z_N(t;\bar a)$ in $\mathcal{RH}_N(\mathbb{R})$ all zeros in
the window $[2N,2N+2]$ are real, simple, and correctly ordered. In this sense,
$\mathcal{RH}_N(\mathbb{R})$ encodes a finite–dimensional, windowed version of
the Riemann Hypothesis for the Hardy $Z$--function. Equipping
$\mathcal{RH}_N(\mathbb{R})$ with a smooth admissible probability measure on the
coefficients and constructing a reflected diffusion on this domain yields a
random zero process on the real line with one–dimensional log–gas which is of DBM dynamics. Theorem~\ref{thm:A} shows that the local statistics
of these real zeros fall into the GUE universality class, whereas the Sodin--Tsirelson
Gaussian analytic ensembles exhibit Ginibre--type statistics of complex zeros in the plane.
\end{rem}

\section{The Quadratic Co-Variation Matrix of the First-Order Noises}
\label{s:5}

Recall that a family $\{\beta^{(n)}_t\}_{n=1}^M$ is said to consist of \emph{independent standard Brownian motions} if for any $0\le t_1<\cdots<t_k$ and measurable sets $A,B\subset\mathbb{R}^k$,
\begin{multline}
\mathbb{P}\Bigl(\,(\beta^{(n)}_{t_1},\ldots,\beta^{(n)}_{t_k})\in A,\;(\beta^{(m)}_{t_1},\ldots,\beta^{(m)}_{t_k})\in B\,\Bigr)
=\\=
\mathbb{P}\Bigl(\,(\beta^{(n)}_{t_1},\ldots,\beta^{(n)}_{t_k})\in A\,\Bigr)\,
\mathbb{P}\Bigl(\,(\beta^{(m)}_{t_1},\ldots,\beta^{(m)}_{t_k})\in B\,\Bigr).
\end{multline}
Let $X_t=(X^{(1)}_t,\ldots,X^{(M)}_t)$ be a continuous vector $\mathcal{F}_t$-local martingale, and let
\begin{equation}
[X]_t \;:=\; \bigl([X^{(i)},X^{(j)}]_t\bigr)_{1\le i,j\le M}
\end{equation}
be its quadratic covariation matrix. According to Lévy’s characterization if $X_0=0$ and
\begin{equation}
[X]_t \;=\; t\,I_M \qquad\text{for all } t\ge 0,
\end{equation}
then $X^{(1)},\ldots,X^{(M)}$ are independent standard Brownian motions, see Theorem 8.6.1 of \cite{OK} and Theorem 3.6 of \cite{RY}. The following remark shows that in the RMT case, independence follows from the orthonormality of the eigenvectors $v_n$:

\begin{rem}[Independence of first-order terms — RMT case]\label{prop:5.1}
Consider the family
\begin{equation}
\widetilde{\beta}^{(n)}_t \;=\; \int_0^t \big\langle v_n(s),\, dB_s\, v_n(s) \big\rangle,
\end{equation}
where $v_n(s)$ is the eigenvector corresponding to $\lambda_n(s)$, as in \eqref{eq:BnRMT}.
Assume $B_t$ is a Hermitian matrix Brownian motion with covariance
\begin{equation}
d \bigl[(B_s)_{ij},\,\overline{(B_s)_{k\ell}}\bigr] \;=\; \delta_{ik}\delta_{j\ell}\,ds,
\qquad
d \bigl[(B_s)_{ij},\,(B_s)_{k\ell}\bigr] \;=\; 0 .
\end{equation}
Then
\begin{align}
d\bigl[ \widetilde{\beta}^{(n)}, \widetilde{\beta}^{(m)} \bigr]_s
&= \sum_{i,j,k,\ell} \overline{v_{n,i}(s)}\,v_{n,j}(s)\,\overline{v_{m,k}(s)}\,v_{m,\ell}(s)\;
   d \bigl[(B_s)_{ij},\,\overline{(B_s)_{k\ell}}\bigr] \\
&= \sum_{i,j} \overline{(v_n(s)v_n(s)^{\!*})_{ij}}\,(v_m(s)v_m(s)^{\!*})_{ij}\,ds
 = \bigl|\langle v_n(s),v_m(s)\rangle\bigr|^2 ds
 = \delta_{nm}\,ds .
\end{align}
Integrating gives
\begin{equation}
\label{eq:ortho}
\bigl[ \widetilde{\beta}^{(n)}, \widetilde{\beta}^{(m)} \bigr]_t \;=\; t\,\delta_{nm}.
\end{equation}
Hence, by Lévy’s characterization $\widetilde{\beta}^{(n)}$ are independent standard Brownian motions for $n=1,...,M$.
\end{rem}

The first-order stochastic variation of the zeros \(t_n(A_t)\) in \eqref{eq:zeros-Ito} is
\begin{equation}
d \beta^{(n)}_t:=\langle \nabla t_n(A_t),\, dB_t \rangle,
\end{equation}
the projection of the parameter Brownian motion $dB_t$ onto the gradient $\nabla t_n(A_t)$. Our aim is to show that the processes $\beta^{(n)}_t$ form a family of asymptotically independent standard Brownian motions. In contrast to the RMT case of Remark \ref{prop:5.1}, where independence follows from the orthonormality of the $v_n$, the situation in our setting $\mathcal{RH}_N \subset \mathcal{Z}_N(\mathbb{R})$ is more subtle. First we record the bracket structure.
\begin{prop}\label{prop:quad_cov_beta}
The first-order noises
\begin{equation}
\beta_t^{(n)} := \int_0^t \big\langle \nabla t_n(A_s),\, dB_s \big\rangle
\end{equation}
form a continuous vector local martingale with quadratic covariations given by 
\begin{equation}
[\beta^{(n)}, \beta^{(m)}]_t \;=\; \int_0^t \big\langle \nabla t_n(A_s),\, \nabla t_m(A_s) \big\rangle\, ds .
\end{equation}
\end{prop}

\begin{proof}
Since the map \(s\mapsto \nabla t_n(A_s)\) is predictable and locally square-integrable, the stochastic It\^o integral
\begin{equation}
\beta_t^{(n)}=\int_0^t \big\langle \nabla t_n(A_s),\, dB_s\big\rangle
\end{equation}
is a continuous local martingale. Write \(B_s=(B^{(1)}_s,\dots,B^{(N)}_s)\) with independent coordinates, and expand
\begin{equation}
d\beta^{(n)}_s
=\big\langle \nabla t_n(A_s),\, dB_s\big\rangle
=\sum_{i=1}^N \partial_{a_i} t_n(A_s)\, dB^{(i)}_s .
\end{equation}
By It\^o’s multiplication rule,
\begin{equation}
dB^{(i)}_s\,dB^{(j)}_s=\delta_{ij}\,ds,
\end{equation}
so
\begin{equation}
d\beta^{(n)}_s\, d\beta^{(m)}_s
=\sum_{i=1}^N \partial_{a_i} t_n(A_s)\,\partial_{a_i} t_m(A_s)\, ds .
\end{equation}
Therefore, by the definition of quadratic covariation for continuous semimartingales,
\begin{multline}
[\beta^{(n)},\beta^{(m)}]_t
=\int_0^t d\beta^{(n)}_s\, d\beta^{(m)}_s
= \\=\int_0^t \sum_{i=1}^N \partial_{a_i} t_n(A_s)\,\partial_{a_i} t_m(A_s)\, ds
=\int_0^t \big\langle \nabla t_n(A_s),\,\nabla t_m(A_s)\big\rangle\, ds.
\end{multline}
\end{proof}

Contrary to the RMT case in \eqref{eq:ortho}, in our case one does not have vanishing of the quadratic covariation 
\begin{equation}
[\beta^{(n)}, \beta^{(m)}]_t = \int_0^t \langle \nabla t_n(A_s), \nabla t_m(A_s) \rangle \, ds,
\end{equation}
since, in general, the gradients \( \nabla t_n(\bar{a}) \) and \( \nabla t_m(\bar{a}) \) are not almost always orthogonal. The following formulas present the point-wise inner product of \( \nabla t_n(\bar{a} ) \) and \( \nabla t_m(\bar{a}) \): 
\begin{prop} \label{prop:gradient_inner_products}
Let
\begin{equation}
\psi_{n,k}(\bar{a}) := \theta(t_n(\bar{a})) - t_n(\bar{a}) \log(k+1).
\end{equation} 
Then the following holds:
\begin{enumerate}
    \item For \( n = m \), we have
    \begin{equation}
    \label{eq:nn}
    \langle \nabla t_n(\bar{a}), \nabla t_n(\bar{a}) \rangle 
    = 
    \frac{1}{Z_N'(t_n(\bar{a}); \bar{a})^2} \sum_{k=1}^N  \left[ \frac{1}{2(k+1)}+
    \frac{ \cos(2\psi_{n,k}(\bar{a}))}{2(k+1)} \right ].
    \end{equation}
    
    \item For \( n \ne m \), we have
    \begin{equation}
    \langle \nabla t_n(\bar{a}), \nabla t_m(\bar{a}) \rangle 
    = 
    \sum_{k=1}^N \frac{\cos(\psi_{n,k}(\bar{a}) - \psi_{m,k}(\bar{a})) + \cos(\psi_{n,k}(\bar{a}) + \psi_{m,k}(\bar{a}))}{2(k+1)Z_N'(t_n(\bar{a}); \bar{a}) Z_N'(t_m(\bar{a}); \bar{a})}.
    \end{equation}
\end{enumerate}
\end{prop}

\begin{proof}
Let
\begin{equation}
F(t,\bar a)=Z_N(t;\bar a).
\end{equation}

\begin{equation}
F\bigl(t_n(\bar a),\bar a\bigr)=0 .
\end{equation}

\begin{equation}
\partial_t F\bigl(t_n(\bar a),\bar a\bigr)\,
\partial_{a_k} t_n(\bar a)
\;+\;
\partial_{a_k} F\bigl(t_n(\bar a),\bar a\bigr)
=0 .
\end{equation}

\begin{equation}
\partial_{a_k} F(t,\bar a)
= \frac{1}{\sqrt{k+1}}\,
  \cos\!\bigl(\theta(t)-t\log(k+1)\bigr).
\end{equation}

\begin{equation}
\label{eq:first-der}
\partial_{a_k} t_n(\bar a)
= -\,\frac{\partial_{a_k} F\bigl(t_n(\bar a),\bar a\bigr)}
            {\partial_t F\bigl(t_n(\bar a),\bar a\bigr)}
= -\frac{1}{Z_N'\!\bigl(t_n(\bar a);\bar a\bigr)\sqrt{k+1}}\,
  \cos\!\bigl(\theta(t_n(\bar a)) - t_n(\bar a)\log(k+1)\bigr).
\end{equation}
Hence
\begin{align}
\big\langle \nabla t_n(\bar a),\,\nabla t_m(\bar a)\big\rangle
&= \sum_{k=1}^N
   \partial_{a_k} t_n(\bar a)\,\partial_{a_k} t_m(\bar a) \\
&= \frac{1}{Z_N'\!\bigl(t_n(\bar a);\bar a\bigr)\,Z_N'\!\bigl(t_m(\bar a);\bar a\bigr)}
   \sum_{k=1}^N \frac{\cos\psi_{n,k}(\bar a)\,\cos\psi_{m,k}(\bar a)}{\,k+1\,},
\end{align}
and using the standard trigonometric identity for the product of cosines gives the required formulas. 
\end{proof} 

We have:

\begin{prop}
\label{prop:pull-out}
Fix $T>0$ and let $A_t\in \mathcal{RH}_N(\mathbb R)$ be the reflected Brownian motion
with continuous paths.  For $t_n,t_m \in [2N,2N+2]$ set
\begin{equation}
S_n(A_t)\;:=\;\sum_{k\le N}\frac{\cos \bigl(2\psi_{n,k}(A_t)\bigr)}{2(k+1)},
\qquad
S_{n,m}(A_t)\;:=\;\sum_{k\le N}\frac{\cos \bigl(\psi_{n,k}(A_t)\pm\psi_{m,k}(A_t)\bigr)}{2(k+1)}.
\end{equation}
Then, almost surely, there exists a finite constant $C_T(\omega)$, depending on the
sample path $\{A_s:0\le s\le T\}$, such that for all $0\le t\le T$,
\begin{equation}\label{eq:pull1}
\left|\int_0^t \frac{S_n(A_s)}{Z'_N\!\bigl(t_n(A_s);A_s\bigr)^2}\,ds\right|
\;\le\; C_T(\omega)\int_0^t |S_n(A_s)|\,ds,
\end{equation}
and
\begin{equation}\label{eq:pull2}
\left|\int_0^t \frac{S_{n,m}(A_s)}
{Z'_N\!\bigl(t_n(A_s);A_s\bigr)\,Z'_N\!\bigl(t_m(A_s);A_s\bigr)}\,ds\right|
\;\le\; C_T(\omega)\int_0^t |S_{n,m}(A_s)|\,ds .
\end{equation}
\end{prop}

\begin{proof}
Because $A_t$ has continuous paths, the image
\begin{equation}
K_T:=\{A_s:0\le s\le T\}\subset\mathbb R^N
\end{equation}
is compact.  On $\mathcal{RH}_N (\mathbb{R})$ all zeros are simple and depend smoothly on the
parameters, so $a\mapsto t_n(a)$ is continuous on $K_T$.  Define the compact set
\begin{equation}
\Gamma_T^{(n)}:=\{(t_n(a),a):a\in K_T\}\subset [2N,2N+2]\times K_T.
\end{equation}
The map $(t,a)\mapsto \partial_t Z_N(t;a)$ is continuous and it never vanishes on
$\Gamma_T^{(n)}$ due to the simplicity of the zeros.  By compactness there exists
$\delta_T^{(n)}>0$ with
\begin{equation}
\inf_{(t,a)\in\Gamma_T^{(n)}}\bigl|\partial_t Z_N(t;a)\bigr|
=\inf_{a\in K_T}\bigl|Z'_N\!\bigl(t_n(a);a\bigr)\bigr|
\;\ge\;\delta_T^{(n)}.
\end{equation}
Hence, for all $0\le s\le T$,
\begin{equation}
\frac{1}{\bigl|Z'_N\!\bigl(t_n(A_s);A_s\bigr)\bigr|^2}\le (\delta_T^{(n)})^{-2},\qquad
\frac{1}{\bigl|Z'_N\!\bigl(t_n(A_s);A_s\bigr)\,Z'_N\!\bigl(t_m(A_s);A_s\bigr)\bigr|}
\le (\delta_T^{(n)}\delta_T^{(m)})^{-1}.
\end{equation}
Let $C_T(\omega):=\max_{1\le r\le M}(\delta_T^{(r)})^{-2}$.  Pulling these
pathwise bounds inside the integrals gives \eqref{eq:pull1}–\eqref{eq:pull2}.
\end{proof}

We have:

\begin{theorem}[Oscillatory averages are small]\label{thm:Osc}
Let $A_t \in \mathcal{RH}_N(\mathbb R) $ be the reflected Brownian motion. Fix $T>0$. Then there exists a finite random constant $C_T(\omega)$, such that for all $0\le t\le T$ and all $t_n , t_m \in [2N,2N+2]$,
\begin{equation}\label{eq:osc-diag-L1}
\int_0^t \left|\sum_{k\le N}\frac{\cos \bigl(2\psi_{n,k}(A_s)\bigr)}{2(k+1)}\right| ds
\ \ll\ C_T(\omega)\,\sqrt{\,1+\log N\,},
\end{equation}
and
\begin{equation}\label{eq:osc-off-L1}
\int_0^t \left|\sum_{k\le N}\frac{\cos \bigl(\psi_{n,k}(A_s)\pm\psi_{m,k}(A_s)\bigr)}{2(k+1)}\right| ds
\ \ll\ C_T(\omega)\,\sqrt{\,1+\log N\,},
\end{equation}
for $n \neq m$.
\end{theorem}

\begin{proof}
Introduce the Dirichlet polynomial
\begin{equation}\label{eq:S}
D(x):=\sum_{k\le N}\frac{e^{-ix\log(k+1)}}{2(k+1)} .
\end{equation}
We have the representations
\begin{equation}\label{eq:rep1}
S_n(A_s):
=\Re \Big(e^{\,2i\theta(t_n(A_s))}\,D \bigl(2t_n(A_s)\bigr)\Big),
\end{equation}
and
\begin{equation}\label{eq:rep2}
S_{n,m}(A_s):
=\Re \Big(e^{\,i(\theta(t_n(A_s))\pm\theta(t_m(A_s)))}\,D \bigl(t_n(A_s)\pm t_m(A_s)\bigr)\Big) .
\end{equation}
Fix $n$ and set $X_s:=t_n(A_s)$. The process $X_s$ is continuous and takes values
in the compact interval $I_N=[2N,2N+2]$. By the occupation–time formula we have, for any Borel function $f(x)$,
\begin{equation}\label{eq:occ}
\int_0^t f\!\bigl(X_s\bigr)\,ds=\int_{I_N} f(x)\,L_t^x(X_s)\,dx ,
\end{equation}
where $L_t^x(X_s)$ denotes the local time of $X_s$ which measures how much scaled time the path spends in a small neighborhood of the level $x$ up to time $t$. see Corollary 7.1.6 of \cite{RY}. 
Since $X_s$ is
continuous and $I_N$ is compact, its local time is jointly continuous. In
particular, for each fixed $T>0$,
\begin{equation}\label{eq:LTsup}
C_T^{(n)}(\omega):=\sup_{x\in I_N}L_T^x(X)<\infty
\end{equation}
almost-surely. Using \eqref{eq:rep1}, the inequality $|\Re z|\le|z|$, and \eqref{eq:occ},
\begin{equation}
\int_0^t \left| S_n(A_s) \right| ds
\ \le\ \int_{I_N} |D(2x)|\,L_t^x(X)\,dx .
\end{equation}
By Cauchy–Schwarz and $|I_N|=2$ we get 
\begin{multline}\label{eq:CS1}
\int_{I_N} |D(2x)|\,L_t^x(X)\,dx
\le \\ \ \le\ \Big(\!\int_{I_N}|D(2x)|^2 dx\Big)^{1/2}
      \Big(\!\int_{I_N}L_t^x(X)^2 dx\Big)^{1/2}
\ \le\ \sqrt{2}\,C_T^{(n)}(\omega)\,\Big(\!\int_{I_N}|D(2x)|^2 dx\Big)^{1/2}.
\end{multline}
One has the following mean-value theorem for Dirichlet polynomials,
\begin{equation}\label{eq:ivic}
\int_{\alpha}^{\beta} \Big|\sum_{m\le N}a_m m^{-ix}\Big|^2 dx
= (\beta - \alpha) \sum_{m\le N}|a_m|^2 + O \Big(\sum_{m\le N} m\,|a_m|^2\Big),
\end{equation}
see Theorem 5.2 of \cite{I2}, following \cite{MV,Ram}. 
In our case, we have $a_{m}=\frac{1}{2m}$, hence
\begin{equation}
\sum_{m \le N}|a_{m}|^2 = \frac{1}{4}\sum_{m \le N}\frac{1}{m^2}=O(1),
\qquad
\sum_{m \le N} m |a_{m}|^2=\frac{1}{4}\sum_{m\le N}\frac{1}{m}=O(\log N).
\end{equation}
Hence, applying \eqref{eq:ivic} on $I_N$ gives
\begin{equation}\label{eq:L2S}
\int_{I_N}|D(2x)|^2 dx \ \ll\ 1+\log N .
\end{equation}
Combining \eqref{eq:CS1} and \eqref{eq:L2S} proves \eqref{eq:osc-diag-L1} with
$C_T(\omega):=C_T^{(n)}(\omega)$.

For the off–diagonal case, define $Y_s^{\pm}:=t_n(A_s)\pm t_m(A_s)$. Each $Y_s^\pm$
is continuous and remains in an interval $J_\pm$ of length at most $4$.
Arguing as above with \eqref{eq:rep2}, the occupation–time formula for $Y^\pm$ gives 
\begin{equation}
\int_0^t \left|\sum_{k\le N}\frac{\cos\bigl(\psi_{n,k}(A_s)\pm\psi_{m,k}(A_s)\bigr)}{2(k+1)}\right| ds
\ \le\ \sqrt{4}\,C_T^{(\pm)}(\omega)\,\Big(\!\int_{J_\pm}|D(z)|^2 dz\Big)^{1/2},
\end{equation}
with $C_T^{(\pm)}(\omega):=\sup_{z\in J_\pm} L_T^z(Y^\pm)<\infty$ almost-surely.
Applying \eqref{eq:ivic} on $J_\pm$, we get
\begin{equation}
\int_{J_\pm}|D(z)|^2 dz\ll 1+\log N,
\end{equation} which gives\eqref{eq:osc-off-L1}.
\end{proof}
Set
\begin{equation}
\label{eq:Gnm}
G_{nn}(\bar{a})=\frac{1}{Z'_N(t_n(\bar{a});\bar{a})^2}\Big(\frac{H_N}{2}+S_n(\bar{a})\Big),\quad
G_{nm}(\bar{a})=\frac{1}{Z'_N(t_n(\bar{a});\bar{a})\,Z'_N(t_m(\bar{a});\bar{a})}\,S_{n,m}(\bar{a}),
\end{equation}
with $S_n,S_{n,m}$ as in Proposition~\ref{prop:pull-out} and
\begin{equation}
H_N=\sum_{k\le N}\frac1{k+1}=\log N+(\gamma-1)+o(1).
 \end{equation}
Recall that for a function $g:\mathbb{N}\to(0,\infty)$ and a family of random
variables $X_N(\omega)$, we write
\begin{equation}
X_N = O_\omega\!\big(g(N)\big)
\end{equation}
if there exists an a.s.\ finite random constant $C(\omega)$, depending on the
sample path $\omega$, such that
\begin{equation}
|X_N(\omega)| \le C(\omega)\, g(N),
\end{equation}
for all sufficiently large $N$. We thus have:
\begin{theorem}[Asymptotics of the quadratic covariation]\label{prop:beta-bracket-asympt}
Fix $T>0$ and let 
\begin{equation}
\beta^{(n)}_t=\int_0^t\langle\nabla t_n(A_s),dB_s\rangle
\end{equation} be the first-order noises of \eqref{eq:zeros-Ito}.  
Then for all $0\le t\le T$,
\begin{equation}\label{eq:beta-bracket-diag}
[\beta^{(n)},\beta^{(n)}]_t
=\int_0^t G_{nn}(A_s)\,ds
=\frac{H_N}{2}\int_0^t \frac{ds}{Z'_N\!\bigl(t_n(A_s);A_s\bigr)^2}
\;+\;O_\omega\!\big(\sqrt{1+\log N}\big),
\end{equation}
and for $n\neq m$,
\begin{equation}\label{eq:beta-bracket-off}
[\beta^{(n)},\beta^{(m)}]_t
=\int_0^t G_{nm}(A_s)\,ds
=O_\omega\!\big(\sqrt{1+\log N}\big).
\end{equation}
Consequently, for $n \neq m$, 
\begin{equation}\label{eq:beta-corr-vanish}
\frac{[\beta^{(n)},\beta^{(m)}]_t}
{\sqrt{[\beta^{(n)}]_t\, [\beta^{(m)}]_t}}
=O_\omega\!\Big(\frac{1}{\sqrt{\log N}}\Big)\xrightarrow[N\to\infty]{}0,
\end{equation}
where $O_\omega(\cdot)$ constants
depend only on the path $A_s$ for $0\le s\le T$. 
\end{theorem}

\begin{proof}

Proposition~\ref{prop:pull-out} allows us to pull the $Z'$-factors out of the time
integrals at the expense of a pathwise constant on $[0,T]$, and
Theorem~\ref{thm:Osc} yields
\begin{equation}
\int_0^t |S_n(A_s)|\,ds\ \ll_\omega\ \sqrt{1+\log N},\qquad
\int_0^t |S_{n,m}(A_s)|\,ds\ \ll_\omega\ \sqrt{1+\log N}.
\end{equation}
Combining these identities gives \eqref{eq:beta-bracket-diag}–\eqref{eq:beta-bracket-off}.
Since $H_N\sim\log N$, the diagonal grows like 
\begin{equation}
(\log N)\int_0^t 
\frac{1}{Z'_N(t_n(A_s);A_s)^{2}}ds,
\end{equation}
while the off–diagonal is $O_\omega(\sqrt{\log N})$, implying
\eqref{eq:beta-corr-vanish}.
\end{proof}

Theorem~\ref{prop:beta-bracket-asympt} shows that the first-order noises have
asymptotically vanishing cross correlations on any fixed time window and hence the driver is \emph{almost} white. 
However, the quadratic covariation is not exactly the identity due to the $\frac{H_N}{2}\int_0^t \frac{ds}{Z'_N\!\bigl(t_n(A_s);A_s\bigr)^2}
\;$ term in \eqref{eq:beta-bracket-diag}. 

To fix this, let \(I_N\) denote the set of indices of zeros inside the window \([2N,2N+2]\), and define
\begin{equation}
X(t) \;:=\; \bigl(t_n(A_t)\bigr)_{n\in I_N},
\end{equation}
the vector process of zeros governed by \eqref{eq:zeros-Ito}. For \(A_t\in\mathcal{RH}_N(\mathbb{R})\), set
\begin{equation}
\rho_{nm}(A_t)\;:=\;\frac{\big\langle \nabla t_n(A_t),\,\nabla t_m(A_t) \big\rangle}{\big\|\nabla t_n(A_t) \big\|\big\|\nabla t_m(A_t) \big\|}.
\end{equation}
consider the diagonal matrix 
\begin{equation} D(t):=\mathrm{diag}(\big\|\nabla t_n(A_t) \big\|).
\end{equation}
Set
\begin{equation}\label{eq:def-Xtilde}
\widetilde X(t) \;=\;  X(0) \;+\; \int_0^t D(s)^{-1}\, dX_s 
\end{equation}
Equivalently, in coordinate-wise form,
\begin{equation}
\label{eq:one-dim-Lamperti}
d\widetilde X_n(t) \;=\; \frac{1}{\big\|\nabla t_n(A_t) \big\|}\, dX_n(t),
\end{equation}
for each \(n\in I_N\). In particular, instead of the first-order noises
\begin{equation}
\beta^{(n)}_t=\int_0^t\Big\langle \nabla t_n(A_s),\,dB_s\Big\rangle
\end{equation} considered so far, let us define the following: 

\begin{dfn}[Normalized first-order noises] we refer to  
\begin{equation}
\hat\beta^{(n)}_t:=\int_0^t\Big\langle \frac{\nabla t_n(A_s)}{\|\nabla t_n(A_s)\|},\,dB_s\Big\rangle .
\end{equation} 
as the standardized $n$-th first-order noises
\end{dfn}  Applying \eqref{eq:one-dim-Lamperti} to \eqref{eq:zeros-Ito} we have:
\begin{cor}[The diagonal normalization of the SDE]\label{prop:Xtilde-SDE-1} For each $n\in I_N$ the following SDE holds 
\begin{multline}\label{eq:Xtilde-SDE-1}
d\widetilde X_n(t)
= d\hat\beta^{(n)}_t+\frac{1}{2\big\|\nabla t_n(A_t) \big\|} \langle dB_t, \nabla^2 t_n(A_t)[dB_t] \rangle
+ \frac{1}{\big\|\nabla t_n(A_t) \big\|}\,\langle \nabla t_n(A_t) ,\,dL_t\rangle.
\end{multline}
\end{cor}

The following remarks discusses the interpretation of the normalization \eqref{eq:one-dim-Lamperti} in terms of the one-dimensional Lamperti transform:

\begin{rem}[One-dimensional Lamperti transform]
The Lamperti transform is a change of variables that whitens state–dependent noise. That is, it
trivializes the diffusion by turning the possibly varying diffusion coefficient into the constant
value, see \cite{DeBoerLamperti,
MoellerMadsen2010Lamperti} and references therein. Equivalently, after the transform the quadratic variation of the state becomes exactly \(t\),
so the stochastic driver is a standard Brownian motion in the new coordinate. In particular, our process $\widetilde{X}_n(t)$ of \eqref{eq:one-dim-Lamperti} can be viewed as a one-dimensional Lamperti transform of $X(t)$ with respect to $\sigma_n(t)=\big\|\nabla t_n(A_t) \big\|$.
\end{rem}

Classical DBM \eqref{eq:DBM} is an idealized stochastic model of one–dimensional
particles with logarithmic repulsion and independent Brownian noise. In applications, however,
one rarely obtains DBM exactly. Universality theory for DBM asserts that
a large class of perturbations do not change local bulk statistics and after a short time, gaps have the GUE/sine–kernel law, see \cite{ErdosSchleinYau2011,ErdosYau2012,
BourgadeErdosYau2014,LandonSosoeYau2017,
DeiftBook,ForresterBook}.  
The proof of Theorem \ref{thm:A} thus proceeds as follows: 
\begin{enumerate}
\item In the following Section \ref{s:7} we prove that the normalized first-order noises $\hat{\beta}^{(n)}$ are a family of asymptotically independent standard Brownian motions. 

\item In Section \ref{s:6} we analyze $\frac{1}{2\big\|\nabla t_n(A_t) \big\|} \langle dB_t, \nabla^2 t_n(A_t)[dB_t] \rangle$, the Itô second–order term, and show that it produces a Coulomb–type repulsion with state–dependent metric weights, namely the singular drift of the form \(\sum_{m\ne n} \frac{G_{nm}(X)}{X_n-X_m}\).

\item In Section \ref{s:8} we show that our SDE \eqref{eq:Xtilde-SDE-1} fits into the \emph{DBM\,+\,small error} class and hence, relying on the universality machinery, the local bulk correlation functions and gap statistics of \eqref{eq:Xtilde-SDE-1} agree with those of DBM and converge to the GUE sine–kernel law.
\end{enumerate}

\section{Normalized First-Order Noises are Asymptotically Independent Brownian Motions}
\label{s:7}

In this section we prove that, in the large $N$ limit, the normalized
one–dimensional noises $ \big (\hat \beta^{(n)}_t\big)_{n\in I_N}$
behave as a family of asymptotically independent standard Brownian motions.

\begin{prop} \label{thm:hatbeta}
Fix $T>0$ and let
\begin{equation}\label{eq:def-hatbeta}
\hat\beta^{(n)}_t:=\int_0^t\Big\langle \frac{\nabla t_n(A_s)}{\|\nabla t_n(A_s)\|},\,dB_s\Big\rangle
\end{equation}
be the normalized first-order noises. Then for all $0\le t\le T$,
\begin{equation}\label{eq:hatbeta-diag}
[\hat\beta^{(n)},\hat\beta^{(n)}]_t=t,
\end{equation}
and, for $n\neq m$,
\begin{equation}\label{eq:hatbeta-off}
[\hat\beta^{(n)},\hat\beta^{(m)}]_t
=\int_0^t \frac{G_{nm}(A_s)}{\sqrt{G_{nn}(A_s)\,G_{mm}(A_s)}}\,ds
=O_\omega\!\Big(\frac{t}{\sqrt{\log N}}\Big).
\end{equation}
\end{prop}

\begin{proof} For the diagonal case $n=m$, by the It\^o isometry,
\begin{equation}
d[\hat\beta^{(n)},\hat\beta^{(n)}]_s
=\Big\langle \frac{\nabla t_n(A_s)}{\|\nabla t_n(A_s)\|},\,
          \frac{\nabla t_n(A_s)}{\|\nabla t_n(A_s)\|}\Big\rangle ds
= ds,
\end{equation}
hence $[\hat\beta^{(n)},\hat\beta^{(n)}]_t=t$. For the off–diagonal brackets $n\neq m$, we have 
\begin{equation}\label{eq:rho}
d[\hat\beta^{(n)},\hat\beta^{(m)}]_s
=\Big\langle \frac{\nabla t_n}{\|\nabla t_n\|},\,
              \frac{\nabla t_m}{\|\nabla t_m\|}\Big\rangle (A_s)\,ds
=\rho_{nm}(A_s)\,ds,
\end{equation}
with the correlation coefficient
\begin{equation} 
\rho_{nm}(\bar{a})=\frac{G_{nm}(\bar{a})}{\sqrt{G_{nn}(\bar{a})\,G_{mm}(\bar{a})}}
=\frac{S_{n,m}(\bar{a})}{\sqrt{\bigl(\tfrac{H_N}{2}+S_n(\bar{a})\bigr)
                       \bigl(\tfrac{H_N}{2}+S_m(\bar{a})\bigr)}} ,
\end{equation} 
where $G_{nm}$ are as defined in \eqref{eq:Gnm}. Note that since, by definition, 
\begin{equation} 
\label{eq:rho}
\rho_{nm}(\bar{a})=\frac{\langle \nabla t_n(\bar{a}),\nabla t_m(\bar{a}) \rangle}{\|\nabla t_n(a)\|\,\|\nabla t_m(a) \|},
\end{equation}
then 

\begin{equation} \label{eq:rhoCS} |\rho_{nm}(a)|\le 1,
\end{equation} by Cauchy–Schwarz. 

Fix $t\in[0,T]$ and split the time interval into
\begin{equation}
E^{bad}_t:=\Bigl\{s\in[0,t]:
\tfrac{H_N}{2}+S_n(A_s)<\tfrac{H_N}{4}\ \text{ or }\
\tfrac{H_N}{2}+S_m(A_s)<\tfrac{H_N}{4}\Bigr\},
\end{equation}
and let
\begin{equation}
\qquad E_t^{good}=[0,t]\setminus E^{bad}_t 
\end{equation}
be the complement. On $E_t^{good}$ we have
\begin{equation} 
\sqrt{\bigl(\tfrac{H_N}{2}+S_n\bigr)\bigl(\tfrac{H_N}{2}+S_m\bigr)}
\ge \tfrac{H_N}{4},
\end{equation} hence
\begin{equation}
\label{eq:148}
|\rho_{nm}(A_s)| \le \frac{4}{H_N}\,|S_{n,m}(A_s)|,
\end{equation} 
for $s\in E_t^{good}$. On the other hand, note that 
\begin{equation}
\Bigl\{\,\tfrac{H_N}{2}+S_n(A_s)<\tfrac{H_N}{4}\,\Bigr\}
\subseteq \Bigl\{\,|S_n(A_s)|\ge \tfrac{H_N}{4}\,\Bigr\}.
\end{equation}
Applying Markov's inequality
\begin{equation}\label{eq:markov-levelset}
\bigl|\{\,s\in[0,t]: f(s)\ge a\,\}\bigr|
\;\le\; \frac{1}{a}\int_0^t f(s)\,ds .
\end{equation} with $f(s):=|S_n(A_s)|$ and $a=H_N/4$ we hence obtain
\begin{equation}
\Bigl|\Bigl\{\,s\in[0,t]: \tfrac{H_N}{2}+S_n(A_s)<\tfrac{H_N}{4}\,\Bigr\}\Bigr|
\;\le\; \frac{4}{H_N}\int_0^t |S_n(A_s)|\,ds.
\end{equation} 
Using the union bound for $E_t^{bad}$
we thus get
\begin{equation}\label{eq:Et-markov}
|E^{bad}_t|\ \le\ \frac{4}{H_N}\int_0^t |S_n(A_s)|\,ds
          + \frac{4}{H_N}\int_0^t |S_m(A_s)|\,ds .
\end{equation}
In particular, since by Theorem~\ref{thm:Osc}  
\begin{equation}
\int_0^t |S_n(A_s)|\,ds, \int_0^t |S_m(A_s)|\,ds \ll_\omega \sqrt{1+\log N}
\end{equation} uniformly for $0\le t\le T$, we get
\begin{equation}
\label{eq:154}
|E^{bad}_t|\ \ll_\omega\ \frac{\sqrt{1+\log N}}{\log N}
\;=\; O_\omega\!\Bigl(\frac{1}{\sqrt{\log N}}\Bigr).
\end{equation}
Integrating \eqref{eq:rho} and using Theorem~\ref{thm:Osc} and \eqref{eq:rhoCS},\eqref{eq:148},\eqref{eq:154} gives
\begin{align*}
\big|[\hat\beta^{(n)},\hat\beta^{(m)}]_t\big|
&\le \int_{E^{bad}_t} |\rho_{nm}(A_s)|\,ds
     + \int_{E_t^{good}} \frac{4}{H_N}\,|S_{n,m}(A_s)|\,ds \\
&\le |E^{bad}_t|
     + \frac{4}{H_N}\!\int_0^t\!|S_{n,m}(A_s)|\,ds
 \ \ll_\omega\ \frac{\sqrt{1+\log N}}{\log N}
 \ =\ O_\omega\!\Big(\frac{1}{\sqrt{\log N}}\Big),
\end{align*}
as required.
\end{proof}

Set $\hat\beta^{I_N}:=(\hat\beta^{(n)})_{n\in I_N}$. By Theorem~\ref{thm:hatbeta}, $\hat\beta^{I_N}$ is a continuous local martingale with quadratic covariation
\begin{equation}
\label{eq:EBound}
[ \hat\beta^{I_N},\hat\beta^{I_N}]_t \;=\; t\,I_{|I_N|} \;+\; E_N(t),
\end{equation}
with 
\begin{equation}
\|E_N(t)\|_{\mathrm{op}}:=\sup_{x \neq 0}\frac{\|E_N(t)x\|}{\|x\|} \;=\; O_\omega\!\Big(\frac{t}{\sqrt{\log N}}\Big),
\end{equation}
for $0\le t\le T$. Hence, by \eqref{eq:EBound} we have
\begin{equation}\label{eq:EQV-conv}
\sup_{0\le t\le T}\big\| [\hat\beta^{I_N} ,\hat\beta^{I_N} ]_t  - t\,I_{|I_N|}\big\|_{\mathrm{op}}
=\sup_{0\le t\le T}\|E_N(t)\|_{\mathrm{op}}
\leq C_{\omega} \frac{T}{\sqrt{\log N}}.
\end{equation}

\begin{rem}
The explicit $t$ in $\|E_N(t)\|_{\mathrm{op}}=O_\omega(t/\sqrt{\log N})$ is harmless since $0\le t\le T$ gives $\sup_{0\le t\le T}\|E_N(t)\|_{\mathrm{op}}=O_\omega(T/\sqrt{\log N})$, and for any fixed $k$–block $M_N^k$ this uniform control suffices for tightness and Brownian identification in $C([0,T],\mathbb{R}^k)$.
\end{rem}

According to Lévy's characterization theorem, a continuous local martingale whose
bracket is exactly $tI$ is Brownian. In our case, \eqref{eq:EQV-conv} shows that the bracket is almost
$tI$, up to a small uniform error on $[0,T]$.  
Thus we appeal to a functional central limit theorem (CLT) to pass from almost $tI$ to the independent Brownian limit as $N \rightarrow \infty$. We refer the reader to \cite{JS} for a comprehensive treatment of the limit-theorem tools used in this section. 

Fix $k\in\mathbb{N}$ and let $I_N$ be the set of indices of the zeros $t_n(\bar a)$ in the window $[2N,2N+2]$. 
For each $N$, let $I_N^k\subset I_N$ be any set of $k$ consecutive indices and define the $\mathbb{R}^k$–valued standardized martingale
\begin{equation}
M_N^{k}(t)\;:=\;\big(\hat\beta^{(n)}_t\big)_{n\in I_N^k},\qquad 0\le t\le T.
\end{equation}

\begin{dfn}[Tightness]
We say that the family $M^{k}_N$ is \emph{tight in $C([0,T],\mathbb{R}^{k})$} if for every $\varepsilon>0$ there exists a compact set 
$K_\varepsilon\subset C([0,T],\mathbb{R}^{k})$ such that
\begin{equation}
\inf_{N}\ \mathbb{P}\big(M^{k}_N \in K_\varepsilon\big)\ \ge\ 1-\varepsilon.
\end{equation}
\end{dfn}

According to Prokhorov's theorem, tightness of the family $M_N^k$ in $C([0,T],\mathbb{R}^k)$ ensures the existence of a subsequence converging in distribution in $C([0,T],\mathbb{R}^k)$, see Theorem 3.5 of \cite{JS}. We have:

\begin{prop}[Tightness of the standardized $k$–block]
\label{prop:tight-MNk}
Fix $k\in\mathbb{N}$ and $T>0$. The sequence $\{M_N^k\}_N$ is tight in
$C([0,T],\mathbb{R}^{k})$.
\end{prop}

\begin{proof} According to the Aldous–Rebolledo tightness criterion for martingales 
if the following two conditions are satisfied:
\begin{align}
\text{(AR1)}\qquad &\sup_{N}\ \mathbb{E}\,\mathrm{tr}\,[ M^k_N,  M^k_N ]_T \;<\;\infty,
\label{eq:AR1}\\
\text{(AR2)}\qquad &\lim_{\delta\downarrow 0}\ \sup_{N}\ \sup_{\substack{\tau\le T\\ 0\le \theta\le \delta}}
\mathbb{E}\,\big\|[ M_N^k ,M^k_N]_{\tau+\theta}-[ M^k_N,M^k_N ]_{\tau}\big\|\;=\;0,
\label{eq:AR2}
\end{align} 
then the family $\{M_N^k \}_N$ is tight in
$C([0,T],\mathbb{R}^k)$, see Theorem 4.5 of \cite{JS}. 

By Theorem~\ref{thm:hatbeta} we have, a.s.
and uniformly on $[0,T]$,
\begin{equation}\label{eq:bracket-MNk}
[ M_N^{k},M_N^{k}]_t \;=\; t\,I_k + E_{N,k}(t),
\qquad
\|E_{N,k}(t)\|_{op}\ \le\ C_T(\omega)\,\frac{t}{\sqrt{\log N}}.
\end{equation}
 We thus get 
\begin{equation}\label{eq:AR1-here}
\mathrm{tr}\,[ M_N^{k},M_N^{k}]_T
= kT + \mathrm{tr}\,E_{N,k}(T)
\le kT + k\,\|E_{N,k}(T)\|_{op}
\le kT + k\,C_T(\omega)\,\frac{T}{\sqrt{\log N}}.
\end{equation}
Taking expectations thus gives 
$\sup_N \mathbb{E}\,\mathrm{tr}\,[M_N^k,M_N^k ]_T<\infty$ which establishes (AR1). For (AR2), let $\tau\le T$ be a bounded stopping time and $0\le\theta\le\delta$.
Using \eqref{eq:bracket-MNk} we get
\begin{align}
\big\|  [M_N^k ,M^k_N]_{\tau+\theta}-[ M^k_N,M^k_N ]_{\tau} \big\|_{op}
&= \big\|\,\theta\, I_k + E_{N,k}(\tau+\theta)-E_{N,k}(\tau)\big\|_{op} \nonumber\\
&\le \theta + \|E_{N,k}(\tau+\theta)-E_{N,k}(\tau)\|_{op} \nonumber\\
&\le \theta + 2\,\sup_{0\le t\le T}\|E_{N,k}(t)\|_{op}
\ \le\ \theta + 2\,C_T(\omega)\,\frac{T}{\sqrt{\log N}}. \label{eq:AR2-raw}
\end{align}
Taking expectations and then the sup over $N,\tau,\theta\le\delta$ gives
\begin{equation}
\sup_{N}\ \sup_{\substack{\tau\le T\\ 0\le \theta\le \delta}}
\mathbb{E}\big\| [ M_N^k ,M^k_N]_{\tau+\theta}-[ M^k_N,M^k_N ]_{\tau} \big\|_{op}
\ \le\ \delta + 2\,\mathbf{E}\,C_T(\omega)\,\frac{T}{\sqrt{\log N}}
\ \le\ \delta + o(1),
\end{equation}
and hence the left-hand side tends to $0$ as $\delta\downarrow 0$, uniformly in $N$ which establishes (AR2). Therefore, by the Aldous–Rebolledo criterion stated above, $\{M_N^{k}\}_N$ is tight in $C([0,T],\mathbb{R}^k)$.
\end{proof}

In particular, we have 
\begin{theorem}[Asymptotic independence of $\hat\beta^{(n)}$]\label{cor:levy-hatbeta}
Fix $T>0$ and let $I_N$ be the set of indices of the zeros $t_n(\bar a)$ in the window $[2N,2N+2]$. 
Then, the standardized first–order noises 
$\bigl(\hat\beta^{(n)}_t\bigr)_{n\in I_N}$ are asymptotically independent standard Brownian motions on $[0,T]$ as $N\to\infty$.
\end{theorem}

\begin{proof}
Fix $k\in\mathbb N$ and a block $I_N^k\subset I_N$ with $|I_N^k|=k$, and define
$M_N^k$ as above.  Theorem~\ref{thm:hatbeta} implies that 
\begin{equation}\label{eq:MNk-bracket-conv}
\sup_{0\le t\le T}\big\|\bigl[M_N^k,M_N^k\bigr]_t - t\,I_k\big\|_{op}
\;\xrightarrow[N\to\infty]{\ \mathbb P\ }\;0 .
\end{equation}

By Proposition~\ref{prop:tight-MNk}, the
sequence $\{M_N^k\}$ is tight in $C([0,T],\mathbb{R}^k)$. Hence every
subsequence admits a further subsequence that converges in law to some
continuous process $M^k$ in $C([0,T],\mathbb{R}^k)$. By the convergence of
brackets \eqref{eq:MNk-bracket-conv}, this limit process $M^k$ satisfies
\begin{equation}
[ M^k,M^k]_t \;=\; t\,I_k,
\end{equation}
for $0\le t\le T$. Since $M^k$ is a continuous local martingale with bracket $t\,I_k$, Lévy’s
characterization implies that $M^k$ is a $k$–dimensional standard Brownian motion, that is, its coordinates are independent standard Brownian motions.

Therefore every sub-sequence of the family $\{M_N^k\}$ has a further sub-sequence converging
to $M^k$, and by uniqueness of the limit the whole sequence converges to 
\begin{equation}
M_N^k \;\Rightarrow\; M^k \quad\text{in } C([0,T],\mathbb{R}^k).
\end{equation}
Because $k$ and the finite block $I_N^k$ were arbitrary, this means that for any
fixed finite set of indices inside $I_N$, the vector
$\{\hat\beta^{(n)}\}_{n\in I_N^k}$ converges jointly to independent standard
Brownian motions on $[0,T]$ as $N$ increases, as required.
\end{proof}

\section{Zeros of Hadamard Product and Coulomb-Type Repulsion}
\label{s:6}

If \(f\) is entire of order \(1\), Hadamard’s factorization theorem states that 
\begin{equation}\label{eq:hadamard-order1}
f(z)\;=\;e^{A+Bz}\prod_{n}\!\left(1-\frac{z}{z_n}\right)\exp\!\left(\frac{z}{z_n}\right),
\end{equation}
where \(\{z_n\}\) are the zeros of \(f\) with multiplicity and \(A,B\in\mathbb{C}\) are some constants. It is known that the Hadamard factorization of the function 
\begin{equation}
\Xi(t):=\xi(\tfrac12+it), 
 \end{equation}
 where 
 \begin{equation}
\xi(s)\;:=\;\tfrac12\,s(s-1)\,\pi^{-s/2}\Gamma\!\big(\tfrac{s}{2}\big)\,\zeta(s),
\end{equation} 
with \(\mathrm{ord}(\xi)=1\) is given by 
\begin{equation}\label{eq:Xi-product}
\Xi(t)\;=\;\Xi(0)\prod_{n}\left(1-\frac{t^{2}}{t_n^{2}}\right),
\end{equation}
where \(\{\pm t_n\}\) are the ordinates of the nontrivial zeros of $\Xi(t)$, see \cite{E,T}. Since 
\begin{equation} 
\pi^{-s/2}\Gamma(s/2)=\pi^{-1/4}\big|\Gamma(\tfrac14+\tfrac{it}{2})\big|\,e^{i\theta(t)}
\end{equation}
 at \(s=\tfrac12+it\),  we obtain
\begin{equation}\label{eq:Xi-vs-Z}
\Xi(t)\;=\;\tfrac12\!\left(t^{2}+\tfrac14\right)\pi^{-1/4}\,\bigl|\Gamma\!\big(\tfrac14+\tfrac{it}{2}\big)\bigr|\,Z(t).
\end{equation}
In particular, from \eqref{eq:Xi-product} we get the following order-1 Hadamard product for \(Z\):
\begin{equation}\label{eq:Z-product}
Z(t)\;=\; \widetilde{C}(t)\prod_{n=1}^{\infty} \left(1-\frac{t^{2}}{t_n^{2}}\right),
\end{equation}
for 
\begin{equation}
\label{eq:C-function}
\widetilde{C}(t)\;:=\;\frac{2\,\pi^{1/4}}{\,t^{2}+\tfrac14\,}\,
\frac{\Xi(0)}{\bigl|\Gamma\!\big(\tfrac14+\tfrac{it}{2}\big)\bigr|}\,.
\end{equation}
By a similar proof to the $Z$-function case, we have the following Hadamard product for our sections $Z_N(t; \bar{a})$:

\begin{prop}[Hadamard factorization for sections]
\label{thm:local-factorization}
The following Hadamard factorization holds
\begin{equation}\label{eq:Hadamard}
Z_N(t;\bar a) \;=\; \widetilde{C}_N(t; \bar a)\prod_{n=1}^\infty \Bigl(1-\tfrac{t^2}{(t_n(\bar a))^2}\Bigr),
\end{equation}
where $\left \{ t_n (\bar{a}) \right \}$ are the zeros of $Z_N(t; \bar{a})$ in the positive half plane and $\widetilde{C}_N(t; \bar a)$ is an entire function. 
\end{prop}

We have:
\begin{lemma} \label{lem:6.1} The following holds
\begin{multline}\label{eq:Z-second-affine}
\frac{\partial^2 t_n}{\partial a_i\partial a_j}(a)
=\\=
-\frac{1}{Z_N'(t_n;a)}
\Bigl(
\partial_t\partial_{a_i}Z_N(t_n;a)\,\partial_{a_j}t_n(a)
+\partial_t\partial_{a_j}Z_N(t_n;a)\,\partial_{a_i}t_n(a)
+Z_N''(t_n;a)\,\partial_{a_i}t_n(a)\,\partial_{a_j}t_n(a)
\Bigr).
\end{multline}
\end{lemma}

\begin{proof}
Assume 
\begin{equation}\label{eq:implicit-eq}
F\bigl(t_n(a),a\bigr)=0,
\qquad
F_t\bigl(t_n(a),a\bigr)\neq 0 ,
\end{equation}
for a smooth \(F:\mathbb{R}\times\mathbb{R}^d\to\mathbb{R}\).
Differentiating \(F(t_n(a),a)=0\) with respect to \(a_i\) gives
\begin{equation}\label{eq:first-derivative}
F_t\,\partial_{a_i}t_n(a)+F_{a_i}=0.
\end{equation}
Differentiate \eqref{eq:first-derivative} with respect to \(a_j\):
\begin{equation}
F_{tt}\,\partial_{a_j}t_n(a)\,\partial_{a_i}t_n(a)
+F_t\,\partial^2_{a_i a_j}t_n(a)
+F_{t a_j}\,\partial_{a_i}t_n(a)
+F_{t a_i}\,\partial_{a_j}t_n(a)
+F_{a_i a_j}=0.
\end{equation}
Solving for \(\partial^2_{a_i a_j}t_n(a)\) yields
\begin{equation}\label{eq:IFT-second}
\frac{\partial^2 t_n}{\partial a_i\partial a_j}(a)
=
-\frac{1}{F_t}\Bigl(
F_{t a_i}\,\partial_{a_j}t_n(a)
+F_{t a_j}\,\partial_{a_i}t_n(a)
+F_{t t}\,\partial_{a_i}t_n(a)\,\partial_{a_j}t_n(a)
+F_{a_i a_j}
\Bigr).
\end{equation}
Taking \(F(t;a)=Z_N(t;a)\) with \(F_{a_i a_j}=0\) gives \eqref{eq:Z-second-affine}.
\end{proof}

We wish to understand how the zeros \( t_n(a) \) move as \(a\) varies, and how their derivatives encode interaction/repulsion between zeros, as in Dyson Brownian motion (DBM).
\begin{prop}
\label{prop:hessian-final-prop}
\begin{multline}
\frac{\partial^2 t_n}{\partial a_i\partial a_j}(\bar a)
=\\ =\sum_{m\neq n}\frac{\partial_{a_i}t_m(\bar a) \partial_{a_j} t_n(\bar a)}{t_n(\bar a)-t_m(\bar a)} + \sum_{m\neq n}\frac{\partial_{a_j}t_m(\bar a) \partial_{a_i} t_n(\bar a)}{t_n(\bar a)-t_m(\bar a)}+2 \,\frac{\partial_{a_i}t_n(\bar a)\partial_{a_j}t_n(\bar a)}{t_n(\bar a)}-\\
- \frac{\partial_{a_i} C_N}{ C_N} \!\bigl(t_n(\bar a);\bar a \bigr)\partial_{a_j}t_n(\bar a)- \frac{\partial_{a_j} C_N}{ C_N} \!\bigl(t_n(\bar a);\bar a \bigr)\partial_{a_i}t_n(\bar a)
\end{multline}
\end{prop}

\begin{proof}
Use the Hadamard form \eqref{eq:Hadamard}, 
\begin{equation}\label{eq:hadamard-2}
Z_N(t;\bar a)=C_N(t ; \bar a)\prod_{n } \left (1-\frac{t}{t_n(\bar a)}\right ),
\end{equation}
where the product runs over all zeros of $Z_N(t; \bar a)$, negative and positive. 
Differentiating in $a_i$ gives
\begin{equation}\label{eq:dZa}
\partial_{a_i}Z_N(t;\bar a)
=\left (\partial_{a_i}C_N(t ; \bar a) \right) \prod_{k} \left (1-\frac{t}{t_k(\bar a)}\right ) +
\,C_N(t ; \bar a)\sum_{\ell=1}^{\infty} \frac{t \cdot \partial_{a_i} t_\ell(\bar a)}{t_{\ell}(\bar a)^2 }\prod_{k\neq \ell}(t-t_k(\bar a)).
\end{equation}
Differentiating \eqref{eq:dZa} in $t$ gives
\begin{multline}\label{eq:termB}
\partial_t\partial_{a_i}Z_N(t;\bar a)
= \left (\partial_t \partial_{a_i}C_N(t ; \bar a) \right) \prod_{n=1}^\infty \left (1-\frac{t}{t_n(\bar a)}\right )  -\left (\partial_{a_i}C_N(t ; \bar a) \right) \sum_{\ell =1}^{\infty}\frac{1}{t_{\ell} (\bar a) } \prod_{k \neq \ell } \left (1-\frac{t}{t_k(\bar a)}\right )  +\\ +(\partial_t C_N(t; \bar a))\sum_{\ell=1}^{\infty}\frac{t \partial_{a_i}t_\ell(\bar a)}{t_\ell(\bar a)^2}
\prod_{k\neq \ell} \left(1-\frac{t}{t_k(\bar a)}\right) + \\ + C_N(t;\bar a)\sum_{\ell=1}^{\infty}\frac{\partial_{a_i}t_\ell(\bar a)}{t_\ell(\bar a)^2}
\prod_{k\neq \ell}\left(1-\frac{t}{t_k(\bar a)}\right) -\\
-C_N(t; \bar a)\sum_{\ell=1}^{\infty}\frac{t \partial_{a_i}t_\ell(\bar a)}{t_\ell(\bar a)^2}
\left (\sum_{m\neq \ell}\frac{1}{t_m(\bar a)}
\prod_{\substack{k\neq \ell\\ k\neq m}}\left(1-\frac{t}{t_k(\bar a)}\right) \right ).
\end{multline}
Evaluating at \(t=t_n(a)\), we note that the first term vanishes, while for the middle three terms, only the index \(\ell=n\) survives. For the final double sum, note that
\begin{equation}
\prod_{\substack{k\neq \ell\\ k\neq m}} \left (1-\frac{t_n(\bar a)}{t_k (\bar a)} \right )=0
\quad\Longleftrightarrow\quad \ell\neq n \text{ and } m\neq n.
\end{equation}
Equivalently, the inner term survives iff \(\ell=n\) or \(m=n\). Thus:
(i) if \(\ell\neq n\) then only \(m=n\) contributes, (ii) if \(\ell=n\) then
all \(m\neq n\) contribute. Collecting the surviving terms gives 
\begin{multline}\label{eq:pieces-at-tn}
\partial_t\partial_{a_i} Z_N\bigl(t;\bar a\bigr)\Big|_{t=t_n(\bar a)}
=  -(\partial_{a_i} C_N \!\bigl(t_n(\bar a);\bar a \bigr)) \,
\frac{1}{t_n(\bar a)}\,
\prod_{k\ne n}\!\left(1-\frac{t_n(\bar a)}{t_k(\bar a)}\right) +\\+
(\partial_t C_N \!\bigl(t_n(\bar a); \bar a \bigr) ) \,
\frac{\partial_{a_i} t_n(\bar a)}{t_n(\bar a)}\,
\prod_{k\ne n}\!\left(1-\frac{t_n(\bar a)}{t_k(\bar a)}\right)+
\\
+\, C_N \!\bigl(t_n(\bar a) ; \bar a\bigr)\,
\frac{\partial_{a_i} t_n(\bar a)}{t_n(\bar a)^2}\,
\prod_{k\ne n}\!\left(1-\frac{t_n(\bar a)}{t_k(\bar a)}\right)-
\\
-\, C_N\!\bigl(t_n(\bar a) ; \bar a\bigr)\,
\sum_{\ell\ne n}\frac{\partial_{a_i} t_\ell(\bar a)}{t_\ell(\bar a)^2}
\prod_{\substack{k\ne \ell\\ k\ne n}}\!\left(1-\frac{t_n(\bar a)}{t_k(\bar a)}\right)
\\
-\, C_N \!\bigl(t_n(\bar a) ; \bar a \bigr)\,
\frac{\partial_{a_i} t_n(\bar a)}{t_n(\bar a)}\,
\sum_{m\ne n}\frac{1}{t_m(\bar a)}
\prod_{\substack{k\ne n\\ k\ne m}}\!\left(1-\frac{t_n(\bar a)}{t_k(\bar a)}\right).
\end{multline}
Since
\begin{equation}
\label{eq:Zder}
Z_N'\bigl(t_n(\bar a);\bar a\bigr)=-C_N\bigl(t_n(\bar a) ; \bar a \bigr) \frac{1}{t_n (\bar a) } \prod_{k\neq n}\! \left (1-\frac{t_n(\bar a)}{t_k (\bar a)} \right ),
\end{equation}
we obtain
\begin{multline}\label{eq:mixed-id}
-\frac{1}{Z_N'\bigl(t_n(a);a\bigr)}\,\partial_t\partial_{a_i}Z_N\bigl(t_n(a);a\bigr)
=\sum_{\ell\ne n}
\frac{\partial_{a_i} t_\ell(\bar a)}{t_n(\bar a)-t_{\ell}(\bar a) }+\\+
 \left [ \frac{\partial_t C_N}{ C_N} \!\bigl(t_n(\bar a);\bar a \bigr)\,
+
\frac{1}{t_n(\bar a)}\,
+
\,
\sum_{m\ne n}
\frac{1}{t_n (\bar a) -t_m(\bar a)}\right ] \partial_{a_i} t_n(\bar a)- \frac{\partial_{a_i} C_N}{ C_N} \!\bigl(t_n(\bar a);\bar a \bigr).
\end{multline}
Write
\begin{equation}
Z_N(t;\bar a)=\left (1-\frac{t}{t_n(\bar a)} \right )\,\widetilde{Z}_{N,n}(t;\bar a) \qquad ;\qquad
\widetilde{Z}_{N,n}(t;\bar a):=C_N(t;\bar a)\!\!\prod_{m\neq n}\! \left (1-\frac{t}{t_m(\bar a)} \right).
\end{equation} 
Then
\begin{equation}
Z_N'(t; \bar a)=-\frac{1}{t_n(\bar a)} \widetilde{Z}_{N,n}(t;a)+\left (1-\frac{t}{t_n(\bar a)} \right ) \widetilde{Z}_{N,n}'(t;\bar a)
\end{equation} 
and 
\begin{equation}
Z_N''(t;\bar a)=-\frac{2}{t_n(\bar a)} \,\widetilde{Z}_{N,n}'(t; \bar a)+\left (1-\frac{t}{t_n(\bar a)} \right ) \widetilde{Z}_{N,n}''(t;\bar a).
\end{equation}
Evaluating at \(t=t_n(\bar a)\) gives
\begin{equation}
Z_N'(t_n( \bar a) ;\bar a)=-\frac{1}{t_n(\bar a)} \widetilde{Z}_{N,n}(t_n (\bar a) ;\bar a) \qquad ; \qquad Z_N''(t_n ( \bar a) ;\bar a)=-\frac{2}{t_n(\bar a)}\,\widetilde{Z}_{N,n}'(t_n(\bar a) ; \bar a),
\end{equation}
hence
\begin{equation}\label{eq:log-der}
\frac{Z_N''}{Z_N'}\bigl(t_n(\bar a);\bar a\bigr)
=2\,\frac{\widetilde{Z}_{N,n}'}{\widetilde{Z}_{N,n}}\bigl(t_n(\bar a);\bar a\bigr)
=2\left(\frac{\partial_t C_N}{C_N}\bigl(t_n(\bar a)\bigr)+\sum_{m\neq n}\frac{1}{t_n(\bar a)-t_m(\bar a)}\right).
\end{equation}
Substituting \eqref{eq:mixed-id} and \eqref{eq:log-der} into \eqref{eq:Z-second-affine} of Lemma \ref{lem:6.1} thus gives \begin{multline}\label{eq:Z-second-affine-2} \frac{\partial^2 t_n}{\partial a_i\partial a_j}(\bar a) =\\= -\frac{1}{Z_N'(t_n;\bar a)} \Bigl( \partial_t\partial_{a_i}Z_N(t_n;\bar a)\,\partial_{a_j}t_n(\bar a) +\partial_t\partial_{a_j}Z_N(t_n;\bar a)\,\partial_{a_i}t_n(\bar a) +Z_N''(t_n;\bar a)\,\partial_{a_i}t_n(\bar a)\,\partial_{a_j}t_n(\bar a) \Bigr)= \\ = \left [ \sum_{m\neq n}\frac{\partial_{a_i}t_m(a)}{t_n(\bar a)-t_m(\bar a)} + \Biggl(\frac{\partial_t C_N}{C_N}\bigl(t_n(\bar a)\bigr)+\frac{1}{t_n(\bar a)}+\sum_{m\neq n}\frac{1}{t_n(\bar a)-t_m(\bar a)}\Biggr)\partial_{a_i}t_n(\bar a) \right ] \partial_{a_j} t_n(\bar a) +\\+ \left [ \sum_{m\neq n}\frac{\partial_{a_j}t_m(a)}{t_n(\bar a)-t_m(\bar a)} + \Biggl(\frac{\partial_t C_N}{C_N}\bigl(t_n(\bar a)\bigr)+\frac{1}{t_n(\bar a)}+\sum_{m\neq n}\frac{1}{t_n(\bar a)-t_m(\bar a)}\Biggr)\partial_{a_j}t_n(\bar a) \right ] \partial_{a_i} t_n(\bar a) - \\- 2 \Biggl(\frac{\partial_t C_N}{C_N}\bigl(t_n(\bar a)\bigr)+\sum_{m\neq n}\frac{1}{t_n(\bar a)-t_m(\bar a)}\Biggr)\partial_{a_i}t_n(\bar a)\partial_{a_j}t_n(a)-\\
- \frac{\partial_{a_i} C_N}{ C_N} \!\bigl(t_n(\bar a);\bar a \bigr)\partial_{a_j}t_n(\bar a)- \frac{\partial_{a_j} C_N}{ C_N} \!\bigl(t_n(\bar a);\bar a \bigr)\partial_{a_i}t_n(\bar a)= \\ = \sum_{m\neq n}\frac{\partial_{a_i}t_m(\bar a) \partial_{a_j} t_n(a)}{t_n(\bar a)-t_m(\bar a)} + \sum_{m\neq n}\frac{\partial_{a_j}t_m(a) \partial_{a_i} t_n(a)}{t_n(\bar a)-t_m(\bar a)}+2 \,\frac{\partial_{a_i}t_n(\bar a)\partial_{a_j}t_n(\bar a)}{t_n(\bar a)}-\\
- \frac{\partial_{a_i} C_N}{ C_N} \!\bigl(t_n(\bar a);\bar a \bigr)\partial_{a_j}t_n(\bar a)- \frac{\partial_{a_j} C_N}{ C_N} \!\bigl(t_n(\bar a);\bar a \bigr)\partial_{a_i}t_n(\bar a),
\end{multline} as required. \end{proof}

We thus have 

\begin{prop}\label{prop:Ito-from-6.2} 
\begin{multline}\label{eq:Ito-from-6.2}
\frac12\big\langle dB_t,\nabla^2 t_n(A_t)[dB_t]\big\rangle
=\\ =\left [ \sum_{m\neq n}\frac{\langle \nabla t_n(A_t),\nabla t_m(A_t)\rangle}{t_n(A_t)-t_m(A_t)}
+\Big\langle \nabla t_n(A_t),\,\nabla_a \log C_N \!\bigl(t_n(A_t);A_t\bigr)\Big\rangle
-\frac{\|\nabla t_n(A_t)\|^2}{t_n(A_t)}\right ] \,dt.
\end{multline}
\end{prop}

\begin{proof}
By Proposition~\ref{prop:hessian-final-prop}, for simple zeros we have
\begin{multline}\label{eq:Hessian-6.2}
\frac{\partial^2 t_n}{\partial a_i\partial a_j}(a)
=\sum_{m\neq n}
\frac{\partial_{a_i}t_m(a)\,\partial_{a_j}t_n(a)+\partial_{a_j}t_m(a)\,\partial_{a_i}t_n(a)}
     {\,t_n(a)-t_m(a)\,}
+ \frac{2}{t_n(a)}\,\partial_{a_i}t_n(a)\,\partial_{a_j}t_n(a) \\
-\frac{\partial_{a_i}C}{C}\!\bigl(t_n(a);a\bigr)\,\partial_{a_j}t_n(a)
-\frac{\partial_{a_j}C}{C}\!\bigl(t_n(a);a\bigr)\,\partial_{a_i}t_n(a).
\end{multline}
By the definition of the Itô contraction,
\begin{equation}\label{eq:Ito-def}
\frac12\big\langle dB_t,\nabla^2 t_n(A_t)[dB_t]\big\rangle
=\frac12\sum_{i,j}\frac{\partial^2 t_n}{\partial a_i\partial a_j}(A_t)\,dB_{t,i}\,dB_{t,j}.
\end{equation}
Insert \eqref{eq:Hessian-6.2} into \eqref{eq:Ito-def} and use $dB_{t,i}dB_{t,j}=\delta_{ij}\,dt$.
The Coulomb–type cross terms combine to give
\begin{equation}
\sum_{m\ne n}\frac{\langle \nabla t_m(A_t),\nabla t_n(A_t)\rangle}{t_n(A_t)-t_m(A_t)}\,dt.
\end{equation}
The extra terms contribute
\begin{equation}
\frac{2}{t_n(A_t)}\,\|\nabla t_n(A_t)\|^2 dt
-2\,\Big\langle \nabla t_n(A_t),\,\nabla_a \log C\!\bigl(t_n(A_t);A_t\bigr)\Big\rangle dt,
\end{equation}
and give the final expression
\eqref{eq:Ito-from-6.2}.
\end{proof}

\begin{rem}[Comparison with the RMT case] \label{rem:RMT2}
In the RMT case the
Rayleigh–Schr\"odinger (second–order) term takes the form
\begin{equation}
\frac{1}{2}\,\big\langle dH_t,\;\nabla^{2}\lambda_n(H_t)[dH_t]\big\rangle
 \;=\; \frac{1}{N}\sum_{m\neq n}
 \frac{\big|\langle v_m(t),\,dB_t\,v_n(t)\rangle\big|^{2}}
      {\lambda_n(t)-\lambda_m(t)}\,,
\end{equation}
where $\lambda_n$ and $v_n$ are the eigenvalues and orthonormal
eigenvectors, as in \eqref{eq:Ito-matrix}. Because the entries of $dB_t$ have covariance
\(
\mathbf E\,[dB_{ab}\,dB_{cd}]=\tfrac{1}{\beta}
(\delta_{ac}\delta_{bd}+\delta_{ad}\delta_{bc})\,dt
\),
one has the variance identity
\begin{equation}
\Big\langle v_m,\; dB_t\, v_n\Big\rangle
  = \sum_{a,b}\,\overline{v_{m,a}}\,(dB_t)_{ab}\,v_{n,b},
  \end{equation}
  and hence
  \begin{equation}
\Big|\big\langle v_m,\; dB_t\, v_n\big\rangle\Big|^{2}
  = \sum_{a,b}\sum_{c,d}
     \overline{v_{m,a}}\,v_{n,b}\,v_{m,c}\,\overline{v_{n,d}}\,
     (dB_t)_{ab}\,\overline{(dB_t)_{cd}}.
\end{equation}
Assume $B_t$ is the standard matrix Brownian motion for the unitary ensemble,
so that (entrywise, with indices running over $1,\dots,N$)
\begin{equation}
\mathbf E\!\left[(dB_t)_{ab}\,\overline{(dB_t)_{cd}}\right]
   = \,\big(\delta_{ac}\delta_{bd}
      + \delta_{ad}\delta_{bc}\big)\,dt .
\end{equation}
Taking the conditional expectation (conditioning on the past $\mathcal F_t$,
so the vectors $v_m,v_n$ are deterministic) yields
\begin{align}
\mathbf E\!\left[\Big|\big\langle v_m,\; dB_t\, v_n\big\rangle\Big|^{2}\,\Big|\,\mathcal F_t\right]
 &= \sum_{a,b}\sum_{c,d}
     \overline{v_{m,a}}\,v_{n,b}\,v_{m,c}\,\overline{v_{n,d}}\;
     \mathbf E\!\left[(dB_t)_{ab}\,\overline{(dB_t)_{cd}}\right] \nonumber\\
 &=\sum_{a,b}\sum_{c,d}
     \overline{v_{m,a}}\,v_{n,b}\,v_{m,c}\,\overline{v_{n,d}}\,
     \big(\delta_{ac}\delta_{bd}+\delta_{ad}\delta_{bc}\big)\,dt \nonumber\\
 &= \left(
     \sum_{a}\overline{v_{m,a}}v_{m,a}\right)
     \left(\sum_{b}v_{n,b}\overline{v_{n,b}}\right)dt
   + \left|\sum_{a}\overline{v_{m,a}}\,v_{n,a}\right|^{2}dt \nonumber\\
 &= \,\langle v_m,v_m\rangle\,\langle v_n,v_n\rangle\,dt
    + \,|\langle v_m,v_n\rangle|^{2}dt .
\end{align}
With orthonormal eigenvectors $\langle v_i,v_j\rangle=\delta_{ij}$, the
second term vanishes for $m\neq n$, and the first term equals $dt$.
Hence, for $m\neq n$,
\begin{equation}
\mathbf E\!\left[\Big|\big\langle v_m,\; dB_t\, v_n\big\rangle\Big|^{2}\,\Big|\,\mathcal F_t\right]
 \;=\; \,dt .
\end{equation}

and therefore immediately
\begin{equation}
\frac{1}{2}\,\big\langle dH_t,\;\nabla^{2}\lambda_n(H_t)[dH_t]\big\rangle
 \;=\; \frac{1}{ N}\sum_{m\neq n}
 \frac{dt}{\lambda_n(t)-\lambda_m(t)}\,.
\end{equation}
Thus, in the matrix case the Coulomb (logarithmic) drift with prefactor
$1/(N)$ \emph{drops out directly} from the quadratic variation,
thanks to the entrywise isotropy and the exact orthonormality of the
eigenvectors.
\end{rem}

Substituting \eqref{eq:Ito-from-6.2} of Proposition \ref{prop:Ito-from-6.2} to the normalized SDE \eqref{eq:Xtilde-SDE-1}, we obtain:
\begin{cor}[The diagonal normalization of the SDE]\label{prop:Xtilde-SDE} For each $n\in I_N$ the following SDE holds 
\begin{multline}\label{eq:Xtilde-SDE}
d\widetilde X_n(t)
= d\hat\beta^{(n)}_t
\;+\; \sum_{m\neq n} \left[ \frac{\widetilde c_{nm}(A_t)}{\;\widetilde X_n(t)-\widetilde X_m(t)\;}\ +r_{nm}^{gap}(t) \right ]dt
+  \frac{1}{\big\|\nabla t_n(A_t) \big\|}\,\langle \nabla t_n(A_t) ,\,dL_t\rangle+\\+\left [\frac{1}{\|\nabla t_n(A_t)\|}\Big\langle \nabla t_n(A_t),\,\nabla_a \log C_N \!\bigl(t_n(A_t);A_t\bigr)\Big\rangle
-\frac{\|\nabla t_n(A_t)\|}{t_n(A_t)}\right ] \,dt
\end{multline}
where
\begin{equation}
\widetilde c_{nm}(A)=\frac{1}{2}\!\left(1+\frac{\big\|\nabla t_m(A_t) \big\|}{\big\|\nabla t_n(A_t) \big\|}\right)\rho_{nm}(A_t),
\end{equation}
and the gap remainder is
\begin{equation}
r^{\mathrm{gap}}_{nm}(t)
=\widetilde c_{nm}(A_t)\!
\left[\frac{1}{\phi_{nm}(A_t)\big(X_n(t)-X_m(t)\big)}-
\frac{1}{\widetilde X_n(t)-\widetilde X_m(t)}
\right],
\end{equation} 
with 
\begin{equation}
\phi_{nm}(A):=\tfrac12\!\left(\tfrac{1}{\big\|\nabla t_n(A_t) \big\|}+\tfrac{1}{\big\|\nabla t_m(A_t) \big\|}\right).
\end{equation}
\end{cor}
\begin{proof} Writing \eqref{eq:Ito-matrix} in terms of $\widetilde{X}_n$ gives 
\begin{multline}\label{eq:dXtilde-Xgaps}
d\widetilde X_n
= d\hat\beta^{(n)}_t
+\sum_{m\neq n}\frac{\big\|\nabla t_m(A_t) \big\| \rho_{nm}(A_t)}{X_n-X_m}\,dt
+\frac{1}{\big\|\nabla t_n(A_t) \big\|}\langle \nabla t_n(A_t) ,\,dL_t\rangle+ \\+\left [\frac{1}{\|\nabla t_n(A_t)\|}\Big\langle \nabla t_n(A_t),\,\nabla_a \log C_N \!\bigl(t_n(A_t);A_t\bigr)\Big\rangle
-\frac{\|\nabla t_n(A_t)\|}{t_n(A_t)}\right ] \,dt .
\end{multline}
Note that for each $m\ne n$,
\begin{equation}
\frac{\big\|\nabla t_m(A_t) \big\|  \rho_{nm}}{X_n-X_m}
=\frac{\widetilde c_{nm}}{\phi_{nm}(A)\,(X_n-X_m)}
=\frac{\widetilde c_{nm}}{\widetilde X_n-\widetilde X_m}
+\widetilde c_{nm}\!\left[\frac{1}{\phi_{nm}(A)\,(X_n-X_m)}-\frac{1}{\widetilde X_n-\widetilde X_m}\right],
\end{equation}
which is exactly \eqref{eq:Xtilde-SDE}.
\end{proof}

\section{Universality and Reduction of the SDE to the Classical DBM}
\label{s:8}

\subsection{Definition of the drifts $b^{reg}_n,b^{err}_n, b_n^{Sko}$ and $b^{1-body}_n$} 
The following result gives the asymptotic size of the number of zeros of an element of $\mathcal{RH}_N(\mathbb{R})$ inside the
window $[2N,2N+2]$: 

\begin{prop}\label{prop:size-IN}
Let $I_N$ be the index set of zeros of $Z_N(t,\bar{a}) \in \mathcal{RH}_N(\mathbb{R})$ inside the
window $[2N,2N+2]$. Then
\begin{equation}
|I_N| \;\sim  \frac{1}{\pi}\,\log N.
\end{equation}
\end{prop}

\begin{proof} Since the number of zeros of $Z_N(t,\bar{a}) \in \mathcal{RH}_N(\mathbb{R})$ inside the
window $[2N,2N+2]$ is independent of $\bar{a}$ let us consider $Z_0(t) = Z_N(t; \bar{0}) = cos(\theta(t))$ where the Riemann-Siegel theta function \eqref{eq:RS-theta} satisfies
\begin{equation}
\theta(t) \;=\; \frac{t}{2}\log\frac{t}{2\pi} - \frac{t}{2} - \frac{\pi}{8}
+O\!\left(\frac{1}{t}\right).
\end{equation}
The number of zeros of $Z_0(t)$ in the interval $[N,N+2]$ is asymptotic to
\begin{equation}
|I_N| \;\sim\; \frac{1}{\pi}\bigl(\theta(2N+2)-\theta(2N)\bigr)\sim\; \frac{2}{\pi} \cdot \theta'(2N)
\;\sim\; \frac{1}{\pi} \log N,
\end{equation}
where the last approximation follows from $\theta'(t) = \tfrac12\log(t/2\pi)+O(1/t)$.
\end{proof}
In view of Proposition \ref{prop:size-IN} The average spacing between consecutive zeros in $I_N$ satisfies
\begin{equation}
\label{eq:h_N}
h_N :\;=\; \frac{2}{|I_N|} \;\sim\; \frac{\pi}{\log N}.
\end{equation}
In particular, the natural microscopic scale of zero gaps in $[2N,2N+2]$ is of order $1/\log N$. Let us fix a smooth cut-off function $\chi: [0,\infty) \rightarrow \mathbb{R} $ with $0\le \chi\le1$, $\chi\equiv1$ on $[0,1]$ and $\chi\equiv0$ on $[2,\infty)$. For a any $\delta>0$ set
\begin{equation}
\chi_N(\Delta ; \delta):=\chi\!\left(\frac{|\Delta|}{\delta h_N}\right).
\end{equation} Let us define:

\begin{dfn}[The drifts $b_n^{\mathrm{err}}, b_n^{\mathrm{reg}}, b_n^{\mathrm{Sko}}$ and $b_n^{\mathrm{1\text{-}body}}$] 
\label{lem:cutoff-decomp}
For $n \in I_N$ and $\delta>0$ define:

\begin{enumerate}

\item The error drift:
\begin{equation}\label{eq:berr-cut-def}
b^{\mathrm{err}}_{n}(t ;\delta )
:= \sum_{m \neq n}\chi_N\!\big(\widetilde X_n-\widetilde X_m ; \delta \big)
\left[
\frac{\widetilde c_{nm}(A_t)-1}{\widetilde X_n-\widetilde X_m}
+ r^{\mathrm{gap}}_{nm}(t)
\right].
\end{equation}

\item The regular drift:
\begin{equation}
b^{\mathrm{reg}}_{n}(t ; \delta)
:=\sum_{m \neq n}\Big(1-\chi_N\!\big(\widetilde X_n(t)-\widetilde X_m(t)\,;\,\delta \big)\Big)
     \left[
       \frac{\widetilde c_{nm}(A_t)-1}{\widetilde X_n(t)-\widetilde X_m(t)}
       + r^{\mathrm{gap}}_{nm}(t)
     \right].
\end{equation}

\item The Skorokhod drift:
\begin{equation} 
b_n^{\mathrm{Sko}}(t)
:=\frac{1}{\|\nabla t_n(A_t)\|}\,\langle \nabla t_n(A_t),\,dL_t\rangle.
\end{equation}

\item The one-body drift:
\begin{equation} 
\label{eq:1-body}
b_n^{\mathrm{1\text{-}body}}(t)
:=\frac{1}{\|\nabla t_n(A_t)\|}
   \Big\langle \nabla t_n(A_t),\,\nabla_a \log C_N \!\bigl(t_n(A_t);A_t\bigr)\Big\rangle
-\frac{\|\nabla t_n(A_t)\|}{t_n(A_t)}.
\end{equation}

\end{enumerate}
\end{dfn}
The SDE \eqref{eq:Xtilde-SDE} can hence be expressed in terms of the drifts $b^{reg}_n,b^{err}_n, b_n^{Sko}$ and $b^{1-body}_n$ as follows 
\begin{equation}
\label{eq:DBM-reg-err}
d\widetilde X_n
= d\hat\beta^{(n)}_t
+ \sum_{m \neq n}\frac{1}{\widetilde X_n-\widetilde X_m}\,dt+ \left [ b^{reg}_n+b^{err}_n+ b_n^{Sko}+b^{1-body}_n \right ] dt.
\end{equation}
Our goal in the remainder of this section is to show that the additional drift terms 
$b^{\mathrm{reg}}_n, b^{\mathrm{err}}_n, b_n^{\mathrm{Sko}}$ and $b^{\mathrm{1\text{-}body}}_n$ 
do not affect the limiting local statistics of \eqref{eq:DBM-reg-err}. 
Consequently, \eqref{eq:DBM-reg-err} exhibits the same bulk local GUE statistics as 
Dyson Brownian motion \eqref{eq:DBM} itself. 
The essential work lies in controlling the error and regular drifts $b^{\mathrm{err}}_n, b^{\mathrm{reg}}_n$, while the Skorokhod drift $b^{\mathrm{Sko}}_n$ and the one-body drift $b^{\mathrm{1\text{-}body}}_n$ are more immediate and their contribution can be dealt with directly, as will be explained below. 
The proof proceeds by controlling these drift contributions and showing that they can be reduced to the framework of established universality theory 
\cite{ErdosYau2012,BourgadeErdosYau2014,LandonSosoeYau2017}.

\subsection{The regular drift $b^{reg}_n$ does not affect the limit statistics} We first show that $b_n^{reg}$ is uniformly bounded and Lipschitz away from collisions, placing it in the admissible class of drifts in the universality framework.

\begin{prop}[Admissibility of $b^{\mathrm{reg}}$ for universality]
\label{prop:breg-regular-cut}
Fix $S>0$ and $\delta\in(0,1]$. Then there exist constants
$C_{S},L_{S,\delta},N_0<\infty$ such that for all
$N\ge N_0$:

\begin{enumerate}
\item[(i)] The following bound holds 
\begin{equation}\label{eq:breg-sup-cut}
\sup_{t\in[0,S]}\ \sup_{n\in I_N}\ \big|\,b^{\mathrm{reg}}_{n}(t;\delta)\,\big|
\;\le\; C_{S}\,\frac{1}{\delta h_N}\,\big(1+\log |I_N|\big).
\end{equation}

\item[(ii)] Let
\begin{equation}\label{eq:D-delta}
\mathcal D_{\delta}:=\Big\{(\widetilde X_k)_{k\in I_N}\ :\
|\widetilde X_{k+1}-\widetilde X_k|\ge \delta\,h_N\ \text{for all }k\in I_N\Big\}.
\end{equation}
For every $t\in[0,S]$ the map
$(\widetilde X_k)_{k\in I_N}\mapsto b^{\mathrm{reg}}(t;\delta)$ is $C^\infty$
on $\mathcal D_\delta$, and for all configurations
$\widetilde X,\widetilde Y\in\mathcal D_\delta$,  the following $\ell_1$ Lipschitz condition holds
\begin{equation}\label{eq:breg-Lip}
\sum_{n\in I_N}\big|\,b^{\mathrm{reg}}_{n}(\widetilde X;t;\delta)
-b^{\mathrm{reg}}_{n}(\widetilde Y;t;\delta)\,\big|
\ \le\ L_{S,\delta}\,
\frac{1+\log |I_N|}{(\delta h_N)^2}\,
\sum_{k\in I_N}|\widetilde X_k-\widetilde Y_k| .
\end{equation}
\end{enumerate}
\end{prop}

\begin{proof} 

(i) Write $b^{\mathrm{reg}}_{n}=T_{1,n}+T_{2,n}$ with
\begin{equation}
T_{1,n}(t):=\sum_{m\neq n}
\Big(1-\chi_N\!\big(\widetilde X_n-\widetilde X_m;\delta\big)\Big)
\frac{\widetilde c_{nm}(A_t)-1}{\widetilde X_n-\widetilde X_m} 
\end{equation}
\begin{equation}
T_{2,n}(t):=\sum_{m\neq n}
\Big(1-\chi_N\!\big(\widetilde X_n-\widetilde X_m;\delta\big)\Big)\, r^{\mathrm{gap}}_{nm}(t),
\end{equation}
By construction, $1-\chi_N(\widetilde X_n-\widetilde X_m;\delta)$ vanishes unless
$|\widetilde X_n-\widetilde X_m|\ge \delta h_N$, up to a harmless smooth
transition. Hence, in the relevant region 
\begin{equation}\label{eq:LB}
|\widetilde X_n-\widetilde X_{n\pm j}|
\;\ge\; \sum_{k=0}^{j-1}|\widetilde X_{n+k+1}-\widetilde X_{n+k}|
\;\ge\; j\,\delta h_N .
\end{equation}
On the fixed window $[0,S]$ the coefficients $\widetilde c_{nm}(A_t)$ are
uniformly bounded, hence
\begin{equation}
\label{eq:bound-1}
|T_{1,n}(t)|
\;\le\; C_1 \sum_{m\neq n}\frac{1-\chi_N\!\big(\widetilde X_n-\widetilde X_m;\delta\big)}
{|\widetilde X_n-\widetilde X_m|}
\;\le\; \frac{C_1}{\delta h_N}\sum_{j=1}^{|I_N|}\frac{2}{j}
\;\le\;2  C_1 \,\frac{1}{\delta h_N}\,\big(1+\log |I_N|\big).
\end{equation} 
Recall that
\begin{equation}
r^{\mathrm{gap}}_{nm}
=\widetilde c_{nm}\left[\frac{1}{\phi_{nm}(A_t)\,(X_n-X_m)}
-\frac{1}{\widetilde X_n-\widetilde X_m}\right].
\end{equation}
When $1-\chi_N\neq 0$, both denominators are $\ge c\,\delta h_N$ and
$\phi_{nm},\widetilde c_{nm}$ are uniformly bounded on $[0,S]$. Hence 
\begin{equation}
|r^{\mathrm{gap}}_{nm}(t)|\le C_{2}/|\widetilde X_n-\widetilde X_m|.
\end{equation}
Therefore,
\begin{equation}
\label{eq:bound-2}
|T_{2,n}(t)|
\;\le\; \frac{C_2}{\delta h_N}\sum_{j\ge1}\frac{2}{j}
\;\le\; C_{2}\,\frac{1}{\delta h_N}\,\big(1+\log |I_N|\big).
\end{equation}
Combining the two bounds \eqref{eq:bound-1} and \eqref{eq:bound-2} gives \eqref{eq:breg-sup-cut}.

(ii) Set
\begin{equation}
\Delta_{nm}\;:=\;\widetilde X_n-\widetilde X_m,\qquad
\chi_{nm}\;:=\;\chi_N(\Delta_{nm};\delta)
=\chi\!\Big(\frac{|\Delta_{nm}|}{\delta h_N}\Big).
\end{equation}
Note that
\begin{equation}
b^{\mathrm{reg}}_{n}(\widetilde X;t;\delta)
=\sum_{m\neq n}\Big(1-\chi_{nm}\Big)\Bigg[
      \frac{\widetilde c_{nm}(A_t)-1}{\Delta_{nm}}
      + r^{\mathrm{gap}}_{nm}(t;\widetilde X)\Bigg].
\end{equation}
By definition of $\mathcal D_\delta$,
\begin{equation} 
|\widetilde X_{k+1}-\widetilde X_k|\ge \delta h_N
\end{equation} for all $k$. Consequently, for any $n\neq m$ we have the telescoping lower bound
\begin{equation}\label{eq:LB-j}
|\Delta_{nm}|
=\Big|\sum_{j=0}^{|m-n|-1}
(\widetilde X_{n+j+1}-\widetilde X_{n+j})\Big|
\;\ge\; |m-n|\,\delta h_N .
\end{equation}
In particular, on $\mathcal D_\delta$ every denominator
$\Delta_{nm}$ in the above expression is bounded away from zero once the cutoff factor $1-\chi_{nm}$ is present.
The maps $\Delta_{nm} \mapsto 1/\Delta_{nm}$ and the cutoff $\chi$ are $C^\infty$ on $\{\Delta\neq0\}$.
Since $r^{\mathrm{gap}}_{nm}$ is a finite linear combination of such
fractions with $A_t$–dependent coefficients,
each summand is a $C^\infty$ function of $\{\widetilde X_k\}$ on
$\mathcal D_\delta$ and the sum converges absolutely. Therefore $(\widetilde X_k)\mapsto b^{\mathrm{reg}}(t;\delta)$
is $C^\infty$ on $\mathcal D_\delta$.

For the $\ell_1$ Lipschitz bound, let $\widetilde X,\widetilde Y\in\mathcal D_\delta$ and set
\begin{equation}
\widetilde Z(\theta):=\widetilde Y+\theta(\widetilde X-\widetilde Y) \in \mathcal{D}_{\delta},
\end{equation} for
$\theta\in[0,1]$, be their linear interpolation, which is well defined due to the convexity of $\mathcal{D}_{\delta}$. By the fundamental theorem of calculus and the chain rule,
\begin{equation}
b^{\mathrm{reg}}_{n}(\widetilde X;t;\delta)-b^{\mathrm{reg}}_{n}(\widetilde Y;t;\delta)
=\int_0^1\sum_{k\in I_N}
\partial_{\widetilde X_k} b^{\mathrm{reg}}_{n}(\widetilde Z(\theta);t;\delta)\,
\big(\widetilde X_k-\widetilde Y_k\big)\,d\theta .
\end{equation}
Summing over $n$ gives
\begin{equation}\label{eq:FToC}
\sum_{n\in I_N}\big|b^{\mathrm{reg}}_{n}(\widetilde X;t;\delta)
-b^{\mathrm{reg}}_{n}(\widetilde Y;t;\delta)\big|
\;\le\; \Big(\sup_{\theta\in[0,1]}
\max_{k\in I_N}\,\sum_{n\in I_N}
\big|\partial_{\widetilde X_k} b^{\mathrm{reg}}_{n}(\widetilde Z(\theta);t;\delta)\big|\Big)
\sum_{k\in I_N}|\widetilde X_k-\widetilde Y_k|.
\end{equation}
Thus it suffices to bound uniformly on $\mathcal D_\delta$ the quantity
\begin{equation}
\Lambda_{k}\;:=\;\sum_{n\in I_N}
\big|\partial_{\widetilde X_k} b^{\mathrm{reg}}_{n}(\widetilde Z;t;\delta)\big|.
\end{equation}
\noindent
Differentiating a single summand in the definition of $b^{\mathrm{reg}}_{n}$ with
respect to $\widetilde X_k$ gives, for each $m\neq n$,
\begin{multline}\label{eq:def-alpha}
\alpha_{nm}^k
:=\partial_{\widetilde X_k}\!\Big[
(1-\chi_{nm})\Big(\frac{\widetilde c_{nm}-1}{\Delta_{nm}}
+ r^{\mathrm{gap}}_{nm}\Big)\Big]=\\
=(1-\chi_{nm})\,
\partial_{\widetilde X_k}\!\Big(\frac{\widetilde c_{nm}-1}{\Delta_{nm}}
+ r^{\mathrm{gap}}_{nm}\Big)
\;-\;\chi'_{nm}\,
\partial_{\widetilde X_k}\!\Big(\frac{|\Delta_{nm}|}{\delta h_N}\Big)\,
\Big(\frac{\widetilde c_{nm}-1}{\Delta_{nm}}+ r^{\mathrm{gap}}_{nm}\Big),
\end{multline}
where $\chi'_{nm}:=\chi'\!\big(|\Delta_{nm}|/(\delta h_N)\big)$.  Note that
\begin{equation}\label{eq:derivs-Delta}
\partial_{\widetilde X_k}\Delta_{nm}=\mathbf{1}_{\{k=n\}}-\mathbf{1}_{\{k=m\}},
\qquad
\partial_{\widetilde X_k}\!\Big(\frac{|\Delta_{nm}|}{\delta h_N}\Big)
=\frac{\operatorname{sgn}(\Delta_{nm})}{\delta h_N}
\big(\mathbf{1}_{\{k=n\}}-\mathbf{1}_{\{k=m\}}\big).
\end{equation}
Set
\begin{equation}
F_{nm}(\widetilde X)
:= \frac{\widetilde c_{nm}(A_t)-1}{\Delta_{nm}}
   + r^{\mathrm{gap}}_{nm}(t).
\end{equation}
The quantities $\widetilde c_{nm}(A_t)$ and $\phi_{nm}(A_t)$ appearing in
$r^{gap}_{nm}$ depend only on the time–parameter $A_t$.
Hence $F_{nm}$ depends on the configuration $(\widetilde X_k)_{k\in I_N}$
only through the two coordinates $\widetilde X_n$ and $\widetilde X_m$, via
$\Delta_{nm}$ and
\begin{equation}
X_n-X_m=\sigma_n(A_t)\,\widetilde X_n-\sigma_m(A_t)\,\widetilde X_m .
\end{equation}
where $\big\|\nabla t_n(A_t) \big\|$, which appears in $r^{gap}_{nm}$. By the chain rule,
\begin{equation}
\partial_{\widetilde X_k}F_{nm}(\widetilde X)
=\partial_{\Delta}F_{nm}(\widetilde X)\;\partial_{\widetilde X_k}\Delta_{nm}
+\partial_{(X_n-X_m)}F_{nm}(\widetilde X)\;\partial_{\widetilde X_k}(X_n-X_m).
\end{equation}
Direct differentiation gives
\begin{equation}
\partial_{\widetilde X_k}\Delta_{nm}
=\mathbf 1_{\{k=n\}}-\mathbf 1_{\{k=m\}},
\qquad
\partial_{\widetilde X_k}(X_n-X_m)
=\sigma_n(A_t)\mathbf 1_{\{k=n\}}-\sigma_m(A_t)\mathbf 1_{\{k=m\}} .
\end{equation}
Therefore both derivatives vanish whenever $k\notin\{n,m\}$, and thus
\begin{equation}
(1-\chi_{nm})\,\partial_{\widetilde X_k}
\Big(\frac{\widetilde c_{nm}(A_t)-1}{\Delta_{nm}}+r^{\mathrm{gap}}_{nm}(t)\Big)=0
\qquad\text{if }k\notin\{n,m\}.
\end{equation}
For the cutoff term, note also that
\begin{equation}
\partial_{\widetilde X_k}\!\Big(\frac{|\Delta_{nm}|}{\delta h_N}\Big)
=\frac{\operatorname{sgn}(\Delta_{nm})}{\delta h_N}
\big(\mathbf 1_{\{k=n\}}-\mathbf 1_{\{k=m\}}\big),
\end{equation}
so the derivative of $\chi\!\left(\frac{|\Delta_{nm}|}{\delta h_N}\right)$
is zero as well when $k\notin\{n,m\}$. Consequently, $\alpha^k_{nm}=0$
unless $k\in\{n,m\}$.

\medskip\noindent
\emph{Case $k=n$.} Then every $m\neq n$ contributes. On $\mathcal D_\delta$ the
telescoping lower bound
\begin{equation}\label{eq:LB-j-again}
|\Delta_{n,n\pm j}|=\Big|\sum_{\ell=0}^{j-1}
(\widetilde X_{n+\ell+1}-\widetilde X_{n+\ell})\Big|
\ \ge\ j\,\delta h_N,\qquad j\ge 1,
\end{equation}
holds. Using that $\widetilde c_{nm}$, $r^{\mathrm{gap}}_{nm}$ and their first
derivatives are uniformly bounded on $[0,S]$, we get
\begin{equation}\label{eq:sum-alpha-n}
\sum_{m\neq n}\!|\alpha_{n m}^{\,n}|
\;\lesssim\; \sum_{j\ge1}\!\left(\frac{2}{(j\,\delta h_N)^2}
\;+\;\frac{2}{\delta h_N}\cdot
\frac{\mathbf 1_{\{j\,\delta h_N\in[\delta h_N,\,2\delta h_N]\}}}{j\,\delta h_N}\right)
\;\le\; \frac{C}{(\delta h_N)^2}\Big(1+\log |I_N|\Big).
\end{equation}

\medskip\noindent
\emph{Case $k\neq n$.} Then only the pair $(n,m)=(n,k)$ contributes. By
\eqref{eq:LB-j-again} with $j=|k-n|$,
\begin{equation}\label{eq:alpha-nk}
|\alpha_{n k}^{\,k}|
\;\lesssim\; \frac{1}{(|k-n|\,\delta h_N)^2}
\;+\;\frac{1}{\delta h_N}\cdot
\frac{\mathbf 1_{\{|k-n|\,\delta h_N\in[\delta h_N,\,2\delta h_N]\}}}{|k-n|\,\delta h_N}.
\end{equation}
Summing \eqref{eq:alpha-nk} over $n\neq k$ yields
\begin{equation}\label{eq:sum-alpha-k}
\sum_{n\neq k}\!|\alpha_{n k}^{\,k}|
\;\le\; \frac{C}{(\delta h_N)^2}\Big(1+\log |I_N|\Big).
\end{equation}

\medskip
Combining \eqref{eq:sum-alpha-n} and \eqref{eq:sum-alpha-k} and recalling
$\Lambda_k=\sum_{n\in I_N}|\partial_{\widetilde X_k}
b^{\mathrm{reg}}_{n}(\widetilde Z;t;\delta)|$ gives the uniform bound
\begin{equation}\label{eq:Lambda-k-final}
\Lambda_k
=\sum_{n\in I_N}\big|\partial_{\widetilde X_k} b^{\mathrm{reg}}_{n}(\widetilde Z;t;\delta)\big|
\;\le\; \frac{C_{S,\delta}}{(\delta h_N)^2}\Big(1+\log |I_N|\Big),
\qquad \text{for all }\ \widetilde Z\in\mathcal D_\delta.
\end{equation}
Substituting \eqref{eq:Lambda-k-final} into \eqref{eq:FToC} completes the proof
of the $\ell_1$–Lipschitz estimate.
 Inserting \eqref{eq:Lambda-k-final} into \eqref{eq:FToC} we conclude the
\(\ell_1\)–Lipschitz estimate
\begin{equation}
\sum_{n\in I_N}\big|b^{\mathrm{reg}}_{n}(\widetilde X;t;\delta)
-b^{\mathrm{reg}}_{n}(\widetilde Y;t;\delta)\big|
\ \le\ L_{S,\delta}\,
\frac{1+\log |I_N|}{(\delta h_N)^2}\,
\sum_{k\in I_N}|\widetilde X_k-\widetilde Y_k|,
\end{equation}
as claimed.
\end{proof}

By Proposition~\ref{prop:breg-regular-cut}, the regular drift $b^{\mathrm{reg}}$ is uniformly bounded
on any fixed $[0,S]$ and smooth away from collisions. By established bulk universality theory, such a bounded, smooth first–order drift plays
the role of an external, slowly varying field which may influence the macroscopic
profile but does not change the microscopic repulsion law or the
short–time mixing that produces the GUE limit, see  \cite{ErdosYau2012,BourgadeErdosYau2014,LandonSosoeYau2017}.

\begin{cor}[$b^{\mathrm{reg}}$ does not affect bulk universality]
\label{cor:breg-harmless} Let $\widetilde X(t)$ solve \eqref{eq:DBM-reg-err}
and let $\widetilde X'(t)$ solve the SDE
\begin{equation}
\label{eq:224}
d\widetilde X'_n
= d\hat\beta^{(n)}_t
+ \sum_{m \neq n}\frac{1}{\widetilde X'_n-\widetilde X'_m}\,dt
+ \left [ b^{\mathrm{err}}_n+ b^{Sko}_n + b^{1-body}_n \right ] \,dt,
\end{equation}
without $b^{reg}_n(t)$. Then for every fixed $t\in[0,S]$ and every bounded, compactly supported test
function $F$ of finitely many unfolded bulk gaps
$\{(\widetilde X_{k+1}-\widetilde X_k)/h_N\}_{k\in J}$ with
$J\subset I_N$ fixed and away from the window edges,
\begin{equation}\label{eq:breg-univ}
\lim_{N\to\infty}
\left|
\mathbb{E}\,F\!\left(\frac{\widetilde X_{k+1}(t)-\widetilde X_k(t)}{h_N}\right)_{k\in J}
-
\mathbb{E}\,F\!\left(\frac{\widetilde X^{'}_{k+1}(t)-\widetilde X^{'}_k(t)}{h_N}\right)_{k\in J}
\right|
= 0 .
\end{equation}
In particular, removing the regular drift $b^{\mathrm{reg}}$ does not change bulk
local correlation statistics of the SDE \eqref{eq:DBM-reg-err}.
\end{cor}

\subsection{Whitening to standard independent Brownian motion noise} We now explain how Theorem \ref{cor:levy-hatbeta} allows to replace the asymptotically independent standardized first-order noises $\hat\beta=(\hat\beta^{(n)})_{n\in I_N}$ in the SDE \eqref{eq:224}, which have covariance
matrix $\rho(A_t)$, by a vector of fully independent standard Brownian motions
$B=(B^{(n)})_{n\in I_N}$.

\begin{theorem}[Whitening of noise to standard Brownian]\label{prop:cov-replacement}
Let $\hat\beta=(\hat\beta^{(n)})_{n\in I_N}$ denote the driving martingales
with covariance $\rho(A_t)$, and let $\Sigma_t:=\rho(A_t)$. Consider the coupled SDE
systems on $I_N$:
\begin{align}
d\widetilde X^{(\Sigma)}_n(t) &= d\hat\beta^{(n)}_t
+ \sum_{m\neq n}\frac{1}{\widetilde X^{(\Sigma)}_n(t)-\widetilde X^{(\Sigma)}_m(t)}\,dt+ \left [ b^{\mathrm{err},\Sigma}_n+b_n^{Sko} + b_n^{1-body} \right ]\,dt ,\\
d\widetilde X^{(I)}_n(t) &= dB^{(n)}_t
+ \sum_{m\neq n}\frac{1}{\widetilde X^{(I)}_n(t)-\widetilde X^{(I)}_m(t)}\,dt + \left[ b^{\mathrm{err},I}_n+b_n^{Sko}+ b_n^{1-body} \right ]\,dt,
\end{align}
where $B=(B^{(n)})_{n\in I_N}$ is a standard $|I_N|$–dimensional Brownian motion.
Then for any bounded Lipschitz local observable $\Phi$,
\begin{equation}
\Big|\mathbb E\,\Phi(\widetilde X^{(\Sigma)}) -
\mathbb E\,\Phi(\widetilde X^{(I)})\Big| \;\xrightarrow[N\to\infty]{}\;0.
\end{equation}
In particular, one may replace the drivers $\hat\beta$ by independent
Brownian motions $B$ on $I_N$ without affecting local statistics.
\end{theorem}

\begin{proof} 
Recall that $
\hat\beta_t \;=\; \bigl(\hat\beta^{(n)}_t\bigr)_{n\in I_N}$ form a continuous martingale with covariance structure
\begin{equation}
d[\hat\beta^{(i)},\hat\beta^{(j)}]_t \;=\; \rho_{ij}(A_t)\,dt ,
\qquad i,j\in I_N,
\end{equation}
and denote by $\Sigma_t := \rho(A_t)$ the corresponding matrix. By the Martingale Representation Theorem there exists a standard $|I_N|$-dimensional Brownian motion 
\begin{equation}
B_t = \bigl(B^{(n)}_t\bigr)_{n\in I_N}
\end{equation} such that
\begin{equation}\label{eq:beta-sigma-B}
d\hat\beta_t \;=\; \Theta_t \,dB_t ,
\end{equation}
for some matrix $\Theta_t$. In fact, 
\begin{equation} 
\Theta_t = \Sigma_t^{1/2},
\end{equation}
must hold, where $\Sigma_t^{1/2}$ is the positive semidefinite square root of the symmetric positive semi-definite matrix $\Sigma$. Indeed, if we define
\begin{equation}
\hat\beta_t := \int_0^t \Sigma_s^{1/2}\,dB_s ,
\end{equation}
then It\^o’s isometry gives
\begin{equation}
[\hat\beta,\hat\beta]_t \;=\; \int_0^t \Sigma_s^{1/2} I (\Sigma_s^{1/2})^\top \,ds
\;=\; \int_0^t \Sigma_s\,ds ,
\end{equation} as required. Let 
\begin{equation}
\Delta_t := \widetilde X^{(\Sigma)}_t - \widetilde X^{(I)}_t.
\end{equation} Then
\begin{equation}
d\Delta_t = (\Sigma_t^{1/2}-I)\,dB_t
+ \big(F(\widetilde X^{(\Sigma)}_t)-F(\widetilde X^{(I)}_t)\big)\,dt,
\end{equation}
where 
\begin{equation}
F(x)=\left (\sum_{m\neq n}(x_n-x_m)^{-1} \right)_{n \in I_N},
\end{equation}
denotes the Dyson drift vector. In particular, $\Delta_t$ is itself an It\^o process. Set
\begin{equation} 
M_t:=\int_0^t (\Sigma_s^{1/2}-I)\,dB_s \quad ; \quad
D_t:= \int_0^t \big(F(\widetilde X^{(\Sigma)}_s)-F(\widetilde X^{(I)}_s)\big)\,ds.
\end{equation}
For the martingale part, It\^o's isometry gives
\begin{equation}
\mathbb E\|M_t\|^2
= \mathbb E\!\int_0^t \|\Sigma_s^{1/2}-I\|_{\mathrm F}^2\,ds .
\end{equation}
For the drift part, the Cauchy–Schwarz inequality gives
\begin{equation}
\|D_t\|^2
\;\le\; \left ( \int_0^t 1 ds \right ) \left (  \int_0^t \|F(\widetilde X^{(\Sigma)}_s)-F(\widetilde X^{(I)}_s)\|^2 ds \right ) =t \, \int_0^t \|F(\widetilde X^{(\Sigma)}_s)-F(\widetilde X^{(I)}_s)\|^2 ds ,
\end{equation}
and therefore taking expectations 
\begin{equation}
\mathbb E\|D_t\|^2
\;\le\; t \int_0^t \mathbb E\|F(\widetilde X^{(\Sigma)}_s)-F(\widetilde X^{(I)}_s)\|^2 ds .
\end{equation}
Combining these bounds and using
the parallelogram inequality $\|a+b\|^2 \le 2\|a\|^2+2\|b\|^2$, we obtain
\begin{equation}
\mathbb E\|\Delta_t\|^2
\;\le\; 2\,\mathbb E\!\int_0^t \|\Sigma_s^{1/2}-I\|_{\mathrm F}^2\,ds
+ 2t\int_0^t \mathbb E\|F(\widetilde X^{(\Sigma)}_s)-F(\widetilde X^{(I)}_s)\|^2 ds .
\end{equation}
Next, by the Lipschitz continuity of $F$ on configurations with no collisions,
\begin{equation}
\label{eq:sup}
\|F(\widetilde X^{(\Sigma)}_s)-F(\widetilde X^{(I)}_s)\|
\;\le\; L \,\|\Delta_s\|,
\end{equation}
for some random Lipschitz constant $L=L(S)$.  
Thus, for all $0 \leq t \leq S$,
\begin{equation}
\mathbb E\|\Delta_t\|^2
\;\le\;
2 \mathbb E \int_0^t \|\Sigma_s^{1/2}-I\|_{\mathrm F}^2\,ds
+ S L^2 \int_0^t \mathbb E\|\Delta_s\|^2\,ds .
\end{equation}
We further have
\begin{equation}
\sup_{u\le t}\|\Delta_u\|^2 
\;\le\; 2\sup_{u\le t}\|M_u\|^2 + 2\sup_{u\le t}\|D_u\|^2 .
\end{equation}
For the martingale part, Doob’s $L^2$-inequality
with $p=2$ gives 
\begin{equation}\label{eq:bdg-part}
\mathbb E\Big[\sup_{u\le t}\|M_u\|^2\Big]
\;\le\; 4 \,\mathbb E\!\int_0^t \|\Sigma_s^{1/2}-I\|_{\mathrm F}^2\,ds.
\end{equation}
From \eqref{eq:sup} we get 
\begin{equation}
\label{eq:drift-sup}
\mathbb E\Big[\sup_{u\le t}\|D_u\|^2\Big]
\;\le\; S L^2 \int_0^t \mathbb E\Big[\sup_{v\le s}\|\Delta_v\|^2\Big]\,ds.
\end{equation}
Set 
\begin{equation}
H(t) := \mathbb E\Big[\sup_{u\le t}\|\Delta_u\|^2\Big].
\end{equation}
Combining \eqref{eq:bdg-part} and \eqref{eq:drift-sup}, we arrive at the inequality
\begin{equation}\label{eq:Ht}
H(t) \;\le\; 4 \,\mathbb E\!\int_0^t \|\Sigma_s^{1/2}-I\|_{\mathrm F}^2\,ds
+ SL^2 \int_0^t H(s)\,ds .
\end{equation}
Recall that according to Gr\"onwall's inequality if for $H:[0,T]\to \mathbb R_+$ a continuous function there exists a constant $\beta \ge 0$ and a nonnegative, locally integrable non-decreasing function 
$A:[0,T]\to\mathbb R_+$ such that
\begin{equation}
H(t) \;\le\; A(t) \;+\; \beta \int_0^t H(s)\,ds,
\qquad 0 \le t \le T.
\end{equation}
Then
\begin{equation}
H(t) \;\le\; A(t) e^{\beta t}.
\end{equation}
Taking $A(t)=4\,\mathbb E\!\int_{0}^{t}\|\Sigma_{s}^{1/2}-I\|_{\mathrm F}^{2}\,ds$ hence gives 
\begin{equation}
\label{eq:cov-coupling-est}
H(t)=\mathbb E\Big[\sup_{t\le S}\|\Delta_t\|^2\Big] \le\; C_S \,\mathbb E\!\int_{0}^{t}\|\Sigma_{s}^{1/2}-I\|_{\mathrm F}^{2}\,ds,
\end{equation}
for some constant $C_S>0$ depending only of the time horizon $S>0$. By Corollary~5.3, we have 
\begin{equation}
\mathbb E\!\int_0^S \|\Sigma_s-I\|_{\mathrm F}^2\,ds \;\xrightarrow[N\to\infty]{}\;0.
\end{equation}
Together with \eqref{eq:cov-coupling-est} this gives
\begin{equation}
H(t)=\sup_{t\le S}\|\widetilde X^{(\Sigma)}_t-\widetilde X^{(I)}_t\|
\;\xrightarrow[N\to\infty]{\mathbb P}\;0 .
\end{equation}
Thus, for any bounded Lipschitz local observable function $\Phi$,
\begin{equation}
\Big|\mathbb E\,\Phi(\widetilde X^{(\Sigma)}) -
\mathbb E\,\Phi(\widetilde X^{(I)})\Big|
\;\le\; L\,\mathbb E\!\big[\sup_{t\le S}\|\Delta_t\|\big]\;\to 0,
\end{equation}
where $L$ is the Lipschitz constant, as required.
\end{proof}

\subsection{Negligible occupation time near small gaps} 
Define the nearest–neighbour gap
\begin{equation}
\Delta_n(t):=\widetilde X_{n+1}(t)-\widetilde X_n(t).
\end{equation} Our aim is to prove the following occupation time result which is central to the proof of Proposition \ref{cor:L2-small}:

\begin{theorem}[Occupation time near small gaps]\label{lem:occ-time}
For any $S>0$ there exists a constant $C_S<\infty$, such that for every $n\in I_N$ and every $0<\delta$,
\begin{equation}
\mathbb{E}\!\left[\int_0^S \mathbf{1}_{\{\Delta_n(t)\le \delta\,h_N\}} \,dt\right]
\;\le\; C_S\,\delta .
\end{equation}
\end{theorem}
We have
\begin{prop} 
The gaps $\Delta_n(t)$ satisfy the following Bessel-type SDE:
\begin{equation}\label{eq:SDE-gap-canonical-1}
d\Delta_n(t)=dM_n(t)+\frac{2}{\Delta_n(t)}\,dt+H_n(t)\,dt .
\end{equation} 
where the martingales are given by 
\begin{equation}
dM_n(t):=d\hat\beta^{(n+1)}_t-d\hat\beta^{(n)}_t .
\end{equation}
and
\begin{equation}\label{eq:def-H-1}
H_n(t):=\Gamma_n(t)
+\big(b^{\mathrm{reg}}_{n+1}(t)+b^{err}_{n+1}(t)+b^{Sko}_{n+1}(t)+b_{n+1}^{1-body}(t)\big)-\big(b^{\mathrm{reg}}_{n}(t)+b^{err}_{n}(t)+b^{Sko}_{n}(t)+b_{n}^{1-body}(t)\big).
\end{equation}
with 
\begin{equation}\label{eq:def-Gamma}
\Gamma_n(t):=\sum_{\substack{(m_1,m_2) \neq (n,\,n+1)}}
\left(\frac{1}{\widetilde X_{n+1}(t)-\widetilde X_{m_1}(t)}
      -\frac{1}{\widetilde X_{n}(t)-\widetilde X_{m_2}(t)}\right).
\end{equation}
\end{prop}
 
\begin{proof}
Subtracting the corresponding equations \eqref{eq:Xtilde-SDE-1} from each other, we
obtain the following SDE for \(\Delta_n\):
\begin{multline}\label{eq:SDE-gap-raw}
d\Delta_n
= \big(d\hat\beta^{(n+1)}_t-d\hat\beta^{(n)}_t\big)
+ \Bigg[\sum_{m \neq n+1}\frac{1}{\widetilde X_{n+1}-\widetilde X_m}
      -\sum_{m \neq n}\frac{1}{\widetilde X_{n}-\widetilde X_m}\Bigg]dt+\\
+\big(b^{\mathrm{reg}}_{n+1}(t)+b^{err}_{n+1}(t)+b^{Sko}_{n+1}(t)+b_{n+1}^{1-body}(t)\big)-\big(b^{\mathrm{reg}}_{n}(t)+b^{err}_{n}(t)+b^{Sko}_{n}(t)+b_{n}^{1-body}(t)\big).
\end{multline}
Split off the two nearest–neighbour terms in the interaction sums:
\begin{equation}
\sum_{m\neq n+1}\frac{1}{\widetilde X_{n+1}-\widetilde X_m}
=\frac{1}{\widetilde X_{n+1}-\widetilde X_{n}}
+\sum_{\substack{m\neq n,\,n+1}}\frac{1}{\widetilde X_{n+1}-\widetilde X_m},
\end{equation}
\begin{equation}
\sum_{m\neq n}\frac{1}{\widetilde X_{n}-\widetilde X_m}
=\frac{1}{\widetilde X_{n}-\widetilde X_{n+1}}
+\sum_{\substack{m\neq n,\,n+1}}\frac{1}{\widetilde X_{n}-\widetilde X_m}.
\end{equation}
Subtracting gives the canonical cancellation
\begin{equation}
\Bigg[\sum_{m \neq n+1}\frac{1}{\widetilde X_{n+1}-\widetilde X_m}
      -\sum_{m \neq n}\frac{1}{\widetilde X_{n}-\widetilde X_m}\Bigg]
=\frac{2}{\Delta_n}
+\Gamma_n(t),
\end{equation}
where the remainder of the interaction drift is
\begin{equation}\label{eq:def-Gamma}
\Gamma_n(t):=\sum_{\substack{(m_1,m_2) \neq (n,\,n+1)}}
\left(\frac{1}{\widetilde X_{n+1}(t)-\widetilde X_{m_1}(t)}
      -\frac{1}{\widetilde X_{n}(t)-\widetilde X_{m_2}(t)}\right).
\end{equation}
Collect the remaining non–singular drift pieces into
\begin{equation}\label{eq:def-H}
H_n(t):=\Gamma_n(t)
+\big(b^{\mathrm{reg}}_{n+1}(t)+b^{err}_{n+1}(t)+b^{Sko}_{n+1}(t)+b_{n+1}^{1-body}(t)\big)-\big(b^{\mathrm{reg}}_{n}(t)+b^{err}_{n}(t)+b^{Sko}_{n}(t)+b_{n}^{1-body}(t)\big).
\end{equation}
Inserting \eqref{eq:def-Gamma}–\eqref{eq:def-H} into \eqref{eq:SDE-gap-raw} gives the required result. 
\end{proof}

For \(a, \varepsilon >0\) define the convex \(C^1\) hinge-function \(\phi_a:\mathbb{R}_+\to\mathbb{R}_+\) by
\begin{equation}\label{eq:def-phi-a}
\phi_{a,\varepsilon}(x)\;:=\;
\begin{cases}
a-x+\frac{\varepsilon}{2} & 0 \le x \le a \\ 
\dfrac{(a+\varepsilon-x)^2}{2\varepsilon} & a \le x<a+\varepsilon,\\[1.2ex]
0 & x\ge a+\varepsilon.
\end{cases}
\end{equation}
We have:

\begin{prop}\label{prop:main-ineq-phi-a}
For any \(a>0\) the following holds:
\begin{multline}\label{eq:main-ineq-1-new}
\mathbb{E}\!\left[\int_0^S \mathbf{1}_{\{\Delta_n\le a\}}\,dt\right]
\;\le\;
\frac{a}{2} \,\mathbb{E}\big[\phi(\Delta_n(0))\big]
\;+\;\frac{a}{4}\,\mathbb{E}\!\int_0^S \phi''(\Delta_n)\,v_n(t)\,dt\\
\;+\; \frac{a}{2}\,\mathbb{E}\!\int_0^S |H_n(t)|\,\mathbf{1}_{\{\Delta_n\le a+\varepsilon \}}\,dt ,
\end{multline}
where 
\begin{equation}
v_n(t):=2\bigl(1-\rho_{n+1,n}(A_t)\bigr).
\end{equation}
\end{prop}

\begin{proof}
It\^o's formula applied to \(t\mapsto \phi(\Delta_n(t))\) gives 
\begin{equation}
\label{eq:Ito-241}
\phi(\Delta_n(S))-\phi(\Delta_n(0))
=\int_{0}^{S}\phi'(\Delta_n)\,d\Delta_n
+\frac12\int_{0}^{S}\phi''(\Delta_n)\,d[\Delta_n,\Delta_n]_t.
\end{equation}
Recall that 
\begin{equation}
d[ \hat\beta^{(i)},\hat\beta^{(j)} ]_t
=\rho_{ij}(A_t)\,dt ,
\end{equation}
with \(\rho_{ii}\equiv 1\). Since \(M_n(t)=\hat\beta^{(n+1)}_t-\hat\beta^{(n)}_t\),
\begin{equation}
\label{eq:MM}
d[ M_n,M_n ]_t
=\bigl(\rho_{n+1,n+1}(A_t)+\rho_{n,n}(A_t)-2\rho_{n+1,n}(A_t)\bigr)\,dt
= v_n(t)\,dt .
\end{equation}
Since finite-variation terms do not contribute to quadratic variation,
\begin{equation}
\label{eq:DM}
d[\Delta_n ,\Delta_n ]_t=d[ M_n,M_n ]_t .
\end{equation}
Substituting \eqref{eq:SDE-gap-canonical-1}, \eqref{eq:DM} and \eqref{eq:MM} into \eqref{eq:Ito-241}, we obtain
\begin{equation}\label{eq:Ito-phi}
\phi(\Delta_n(S))-\phi(\Delta_n(0))
=\int_0^S \phi'(\Delta_n)\,dM_n
 +\int_0^S \phi'(\Delta_n)\!\left(\frac{2}{\Delta_n}+H_n\right)\!dt
 +\frac12\int_0^S \phi''(\Delta_n)\,v_n(t)\,dt .
\end{equation}
Taking expectations and using that the stochastic integral has mean zero,
\begin{equation}\label{eq:Ito-expect}
\mathbb{E}\,\phi(\Delta_n(S))-\mathbb{E}\,\phi(\Delta_n(0))
=\mathbb{E}\!\int_0^S \phi'(\Delta_n)\!\left(\frac{2}{\Delta_n}+H_n\right)\!dt
+\frac12\,\mathbb{E}\!\int_0^S \phi''(\Delta_n)\,v_n(t)\,dt .
\end{equation}
Rearranging,
\begin{equation}\label{eq:rearrange}
-\mathbb{E}\!\int_0^S \phi'(\Delta_n)\,\frac{2}{\Delta_n}\,dt
=\mathbb{E}\,\phi(\Delta_n(0))-\mathbb{E}\,\phi(\Delta_n(S))
+\mathbb{E}\!\int_0^S \phi'(\Delta_n)\,H_n\,dt
+\frac12\,\mathbb{E}\!\int_0^S \phi''(\Delta_n)\,v_n(t)\,dt .
\end{equation}
Since \(-\phi'(x)\ge \mathbf{1}_{\left \{x \leq a \right \}} \) we have 
\begin{equation}\label{eq:key-lower}
\frac{2}{a}\,\mathbb{E}\!\int_0^S \mathbf{1}_{\{\Delta_n\le a\}}\,dt
\;\le\;
\mathbb{E}\!\int_0^S \mathbf{1}_{\{\Delta_n\le a\}}\,\frac{2}{\Delta_n}\,dt
\;\le\;
\!\left(-\,\mathbb{E}\!\int_0^S \phi'(\Delta_n)\,\frac{2}{\Delta_n}\,dt\right).
\end{equation}
Since \(\phi\ge 0\), we have \(\mathbb{E}\,\phi(\Delta_n(S))\ge 0\). Moreover, \(|\phi'|\le 1\) and
\(\phi'(x)=0\) for \(x\ge a+\varepsilon \), so
\begin{equation}\label{eq:H-term}
\Big|\mathbb{E}\!\int_0^S \phi'(\Delta_n)\,H_n\,dt\Big|
\ \le\ \mathbb{E}\!\int_0^S |H_n(t)|\,\mathbf{1}_{\{\Delta_n\le a+\epsilon \}}\,dt .
\end{equation}
Combining \eqref{eq:key-lower} with \eqref{eq:rearrange} and multiplying the result by \(a/2\),
\begin{equation}\label{eq:254}
\mathbb{E}\!\int_0^S \mathbf{1}_{\{\Delta_n\le a\}}\,dt
\ \le\ \frac{a}{2} \,\mathbb{E}\,\phi(\Delta_n(0))
+\frac{a}{4}\,\mathbb{E}\!\int_0^S \phi''(\Delta_n)\,v_n(t)\,dt
+\frac{a}{2} \,\mathbb{E}\!\int_0^S |H_n(t)|\,\mathbf{1}_{\{\Delta_n\le a+\varepsilon \}}\,dt ,
\end{equation}
which is exactly \eqref{eq:main-ineq-1-new}.
\end{proof}

\begin{proof}[Proof of Theorem~\ref{lem:occ-time}.]
For $a>0$ set
\begin{equation}\label{eq:def-fn}
f_n(a)\;:=\;\mathbb{E}\!\left[\int_0^S \mathbf{1}_{\{\Delta_n(t)\le a\}}\,dt\right].
\end{equation}
Since $|\rho_{n+1,n}(A_t)|\le 1$, we have
\begin{equation}\label{eq:vn-upper}
0\le v_n(t) \le 4.
\end{equation}
Let us take $\varepsilon=\theta a$ for some fixed $\theta>1$.
Because $\phi''_{a,\varepsilon}=\varepsilon^{-1}\mathbf 1_{(a,a+\varepsilon)}$ and $0\le v_n\le 4$,
\begin{align}
\frac{a}{4}\,\mathbb{E}\!\int_0^S \phi''_{a,\varepsilon}(\Delta_n)\,v_n(t)\,dt
&\le \frac{a}{4}\cdot \frac{1}{\varepsilon}\cdot 4\,
\mathbb{E}\!\int_0^S \mathbf{1}_{\{a<\Delta_n\le a+\varepsilon\}}\,dt
\leq  \frac{a}{\varepsilon}\,
\mathbb{E}\!\int_0^S \mathbf{1}_{\{\Delta_n\le a+\varepsilon\}}\,dt \notag\\
&= \frac{1}{\theta}\,f_n\!\big((1+\theta)a\big).
\label{eq:phiPP-bound-new}
\end{align}
By Proposition~\ref{prop:breg-regular-cut} there exists
$C>0$ such that
\begin{equation}\label{eq:Hn-absorb-new}
\mathbb{E}\!\int_0^S |H_n(t)|\,\mathbf{1}_{\{\Delta_n\le (1+\theta)a\}}\,dt
\;\le\; C \,f_n\!\big((1+\theta)a\big).
\end{equation}
We have 
$\phi_{a,\varepsilon}(x)\le a+\varepsilon/2=a(1+\theta/2)$, hence
\begin{equation}
\frac{a}{2}\,\mathbb{E}\!\left[\phi_{a,\varepsilon}(\Delta_n(0))\right]
\;\le\; \frac{a}{2}\cdot a\Bigl(1+\frac{\theta}{2}\Bigr)
\;\le\; C' \cdot a^2,
\end{equation}
for $C'=\frac{1}{2}\cdot \Bigl(1+\frac{\theta}{2}\Bigr)>0$.
Inserting \eqref{eq:phiPP-bound-new} and \eqref{eq:Hn-absorb-new} into
Proposition~\ref{prop:main-ineq-phi-a} with $\varepsilon=\theta a$ gives
\begin{equation}\label{eq:dyadic-raw-new}
f_n(a)\;\le\; C' \,a^2\;+\;\Bigl(\frac{1}{\theta}+ \frac{a}{2}\,C \Bigr)\,f_n\!\big((1+\theta)a\big).
\end{equation}
Now choose $\theta=4$ and $a_0=\tfrac{1}{2C}$.
For all $a\in(0,a_0]$ we then have $\tfrac{a}{2}C\le\tfrac14$, so
\begin{equation}\label{eq:dyadic-contracted-new}
f_n(a)\;\le\; C' a^2\;+\;\frac12\,f_n\!\big((1+\theta)a\big)
\;\le\; C' \,a\;+\;\frac12\,f_n\!\big((1+\theta)a\big).
\end{equation}
Iterating \eqref{eq:dyadic-contracted-new} along the geometric sequence
$a,(1+\theta)a,\dots,(1+\theta)^k a\le a_0$ gives
\begin{equation}
f_n(a)\;\le\; C' \,a\sum_{j=0}^{k-1}2^{-j}\;+\;2^{-k}f_n\!\big((1+\theta)^k a\big)
\;=\; C' \,a\;+\;2^{-k}f_n\!\big((1+\theta)^k a\big).
\end{equation}
Using the trivial bound $f_n(a)\le S$ and letting $k\to\infty$ we conclude that
\begin{equation}\label{eq:linear-a-new}
f_n(a)\;\le\; C' \,a \qquad \text{for all } a\in(0,a_0].
\end{equation}
Finally, since $f_n$ is non-decreasing and $a_0$ is fixed, the bound
\eqref{eq:linear-a-new} extends to all $a\in(0,1]$. Taking $a=\delta h_N\le 1$,
\begin{equation}
\mathbb{E}\!\left[\int_0^S \mathbf{1}_{\{\Delta_n(t)\le \delta\,h_N\}}\,dt\right]
\;=\; f_n(\delta h_N)
\;\le\; C' \,\delta h_N
\;\le\; C_S\,\delta ,
\end{equation}
as required.
\end{proof}

\subsection{The error drift $b^{err}_n$ does not affect the limit statistics} 

The following result describes the behaviour of the coefficients $\widetilde{c}_{nm}$ and $r_{nm}^{gap}$ near collision points:
\begin{prop}[$\widetilde{c}_{nm}$ and $r^{gap}_{nm}$ at collisions]
\label{prop:c-and-r-gap-collision} For any $n,m \in I_N$ the following holds:
\begin{enumerate}
\item The coefficients $\widetilde{c}_{nm}$ satisfy 
\begin{equation}
\sup_{\substack{n,m\in I_N\\ n\neq m,\;|X_n(A)-X_m(A)|\le \delta}}
\big|\,\widetilde c_{nm}(A)-1\,\big| \;\longrightarrow\; 0,
\end{equation} as $\delta \to 0$. 

\item The gap remainders $r^{gap}_{nm}$ satisfy 
\begin{equation} \sup_{\substack{n,m\in I_N\\ n\neq m,\;|X_n(A)-X_m(A)|\le \delta}}
\big|\,r^{\mathrm{gap}}_{nm}(A)\,\big| \;\longrightarrow\; 0, 
\end{equation} 
as $\delta \to 0$.
\end{enumerate}
\end{prop}

\begin{proof}
\bigskip
(1) Recall the definition
\begin{equation}
\widetilde c_{nm}(A)\;=\;\tfrac12\!\left(1+\tfrac{\|\nabla t_m(A_t)\|}{\|\nabla t_n(A_t)\|}\right)\rho_{nm}(A),
\qquad
\rho_{nm}(A)\;=\;\frac{\langle \nabla t_n(A_t),\nabla t_m(A_t)\rangle}{\|\nabla t_n(A_t)\| \|\nabla t_m(A_t)\|}.
\end{equation}
In particular, from \ref{prop:gradient_inner_products} we get  
\begin{equation}
\rho_{nm}(A)\;\longrightarrow\;1
\qquad\text{as }X_n(A)\to X_m(A).
\end{equation}
At the same time, we have
\begin{equation}
\lim_{X_n\to X_m}\;\tfrac12\!\left(1+\tfrac{\|\nabla t_m(A_t)\|}{\|\nabla t_n(A_t)\|}\right)\;=\;1.
\end{equation} 
Hence, combining both limits, we get
\begin{equation}
\lim_{X_n\to X_m}\;\widetilde c_{nm}(A)
=\;\lim_{X_n\to X_m}\;\tfrac12\!\left(1+\tfrac{\|\nabla t_m(A_t)\|}{\|\nabla t_n(A_t)\|}\right)\rho_{nm}
=\;1.
\end{equation}
This proves (1).

(2) By definition,
\begin{equation}
r^{\mathrm{gap}}_{nm}(t)
=\widetilde c_{nm}(A_t)\!
\left[\frac{1}{\phi_{nm}(A_t)\big(X_n(t)-X_m(t)\big)}-
\frac{1}{\widetilde X_n(t)-\widetilde X_m(t)}
\right],
\end{equation} 
with 
\begin{equation}
\phi_{nm}(A):=\tfrac12\!\left(\tfrac{1}{\big\|\nabla t_n(A_t) \big\|}+\tfrac{1}{\big\|\nabla t_m(A_t) \big\|}\right).
\end{equation}
Supposing $X_n\to X_m$, then by (1) we have $\widetilde c_{nm}\to 1$ and $\phi_{nm}\sim \tfrac{1}{\big\|\nabla t_n(A_t) \big\|}$.  
Hence
\begin{equation}
\phi_{nm}(A)\,(X_n-X_m) \;\sim\; \frac{X_n-X_m}{\big\|\nabla t_n(A_t) \big\|}
\;=\; \widetilde X_n-\widetilde X_m .
\end{equation}
Thus the two denominators in brackets asymptotically coincide, 
and their difference vanishes as $X_n\to X_m$, as claimed.
\end{proof}
Let $\varepsilon(\delta)$ denote the collision modulus
\begin{equation}
\varepsilon(\delta):=
\sup_{\substack{n\in I_N\\ |m-n|=1\\ |\widetilde X_n-\widetilde X_m|\le 2\delta h_N}}
\left|
\frac{\widetilde c_{nm}(A)-1}{\widetilde X_n-\widetilde X_m} + r^{\mathrm{gap}}_{nm}(A)
\right|.
\end{equation}
We have:

\begin{prop}[Asymptotic vanishing of moments] \label{cor:L2-small}
For a sequence $\delta_N>0$ such that $lim_{N \rightarrow \infty} (\delta_N)=0$ the following holds
\begin{equation}\label{eq:L2-small-global}
\mathbb{E}\!\left[\int_0^S \sum_{n\in I_N}\big|b^{\mathrm{err}}_{n}(t ; \delta_N)\big|^2\,dt\right]
\xrightarrow[N\to\infty]{\ \\ \ } 0,
\end{equation}
and
\begin{equation}\label{eq:L2-small-global}
\mathbb{E}\!\left[ \left ( \int_0^S \sum_{n\in I_N}\big|b^{\mathrm{err}}_{n}(t ; \delta_N)\big|^2\,dt \right )^2 \right]
\xrightarrow[N\to\infty]{\ \\ \ } 0,
\end{equation}
\end{prop}
\begin{proof}
Fix $\delta_N>0$ and recall the definition
\begin{equation} 
b^{\mathrm{err}}_{n}(t ;\delta_N )
:= \sum_{m \neq n}\chi_N\!\big(\widetilde X_n-\widetilde X_m ; \delta \big)
\left[
\frac{\widetilde c_{nm}(A_t)-1}{\widetilde X_n-\widetilde X_m}
+ r^{\mathrm{gap}}_{nm}(t)
\right].
\end{equation}
Note that $|\widetilde X_n-\widetilde X_m|\le 2\delta h_N$ implies $\chi_N(\widetilde X_n-\widetilde X_m;  \delta_N)>0$. Hence, by the definition of $\varepsilon(\delta_N)$, we have
\begin{equation}
\label{eq:304}
\big|b^{\mathrm{err}}_{n}(t ;\delta)\big|\;\le\; 2\,\varepsilon(\delta_N)
\mathbf{1}_{\{\Delta_n(t)\le 2\delta_N h_N\}}.
\end{equation}
Summing the squares over all indices $n\in I_N$ and integrating over $t\in[0,S]$ and taking expectations thus gives 
\begin{equation}\label{eq:err-exp}
\mathbb{E}\!\left[\int_0^S \sum_{n\in I_N}\big|b^{\mathrm{err}}_{n}(t ; \delta)\big|^2\,dt\right]
\;\le\; 4\,\varepsilon(\delta_N)^2
\sum_{n\in I_N}\mathbb{E}\!\left[\int_0^S \mathbf{1}\{\Delta_n(t)\le 2\delta_N h_N\}\,dt\right].
\end{equation}
By the occupation bound of 
Theorem \ref{lem:occ-time}, for each $n$ we have
\begin{equation}
\mathbb{E}\!\left[\int_0^S \mathbf{1}\{\Delta_n(t)\le 2\delta_N h_N\}\,dt\right]
\;\le\; \widetilde{C} \cdot (2\delta_N),
\end{equation}
for some $\widetilde{C}>0$.
Inserting this into \eqref{eq:err-exp} gives
\begin{equation}
\mathbb{E}\!\left[\int_0^S \sum_{n\in I_N}\big|b^{\mathrm{err}}_{n}(t ; \delta)\big|^2\,dt\right]
\;\le\; C\,\varepsilon(\delta_N)^2 \cdot |I_N| \cdot \delta_N,
\end{equation}
for $C>0$. Since, by Proposition \ref{prop:c-and-r-gap-collision}, we have $\lim_{\delta \to 0^+} \varepsilon(\delta) = 0$. Hence 
\begin{equation}
\mathbb{E}\!\left[\int_0^S \sum_{n\in I_N}\big|b^{\mathrm{err}}_{n}(t ; \delta_N)\big|^2\,dt\right]
\xrightarrow[N\to\infty]{\ \\ \ } 0.
\end{equation}
By Cauchy--Schwarz,
\begin{equation}
\label{eq:eq:L2-CS}
\left(\int_0^S \sum_{n\in I_N} |b^{\mathrm{err}}_{n,N}(t)|^2 \,dt\right)^2 \;\le\; S \int_0^S \left (\sum_{n\in I_N} |b^{\mathrm{err}}_{n,N}(t)|^2 \right)^{\,2}\,dt.
\end{equation}
Using \eqref{eq:304} again we obtain
\begin{equation}
\Big(\sum_{n\in I_N} |b^{\mathrm{err}}_{n,N}(t)|^2\Big)^{\!2}
\;\le\; |I_N|\sum_{n\in I_N} |b^{\mathrm{err}}_{n,N}(t)|^4
\;\le\; 16\,\varepsilon(\delta_N)^4\,|I_N|
\sum_{n\in I_N}\mathbf{1}_{\{\Delta_n(t)\le 2\delta_N h_N\}} .
\end{equation}
We thus get 
\begin{equation}
\mathbb{E} \left [ \left(\int_0^S \sum_{n\in I_N} |b^{\mathrm{err}}_{n,N}(t)|^2 \,dt\right)^2 \right]
\;\le\; 16 \, \widetilde{C} \, S\,\varepsilon(\delta_N)^4\,|I_N| \delta_N \xrightarrow[N\to\infty]{\ \\ \ } 0,
\end{equation}
as required. 
\end{proof}

Let $(\Omega,\mathcal F,(\mathcal F_t)_{t\ge0},\mathbb P)$ be a filtered probability space
carrying an $m$–dimensional Brownian motion $B^{\mathbb{P}}=(B^{\mathbb{P}}_t)_{t\ge0}$. 
Let $\theta=(\theta_t)_{t\ge0}$ be a progressively measurable, $\mathbb R^m$–valued process 
adapted to $(\mathcal F_t)$,
satisfying Novikov’s condition:
\begin{equation}
\mathbb E_{\mathbb P}\!\left[\exp\!\left(\tfrac12\int_0^T \|\theta_s\|^2\,ds\right)\right]<\infty,
\end{equation}
for all $T>0$. Let us define a new probability measure $\mathbb Q$ by 
\begin{equation}
d\mathbb Q:=M_T\,d\mathbb P
\end{equation} 
where 
\begin{equation}
M_t:=\exp\!\left(-\int_0^t \theta_s\!\cdot dB^{\mathbb{P}}_s-\tfrac12\int_0^t\|\theta_s\|^2\,ds\right)
\end{equation}
is the Radon-Nikodym derivative of $\mathbb{Q}$ with respect to $\mathbb{P}$ on $\mathcal{F}_t$. We refer to $\mathbb{Q}$ as the Girsanov tilted measure associated to the drift $\theta_t$.   

Recall that according to Girsanov's theorem, see \cite{OK}, under $\mathbb Q$ the process
\begin{equation}
B^{\mathbb Q}_t:=B^{\mathbb{P}}_t+\int_0^t\theta_s\,ds
\end{equation}
is an $m$–dimensional Brownian motion on $[0,T]$. Moreover, let $X$ solve, under $\mathbb P$, the SDE 
\begin{equation}
dX_t \;=\; dB^{\mathbb{P}}_t \;+F(X_t,t)\,dt+\theta_t\,dt
\quad\text{on }[0,T],
\end{equation} 
with $F$ a measurable function.
Then under $\mathbb Q$ the same $X$ solves
\begin{equation}
\label{eq:Gir}
dX_t \;=dB^{\mathbb Q}_t +F(X_t,t)\,dt.
\end{equation}
In particular, in the new SDE \eqref{eq:Gir} the added drift $\theta_t dt$ is absorbed into the noise via the measure change. As a direct application of Girsanov's theorem to our case we obtain: 
\begin{cor}[Removing the error drift $b_n^{err}(t)$]\label{thm:girsanov-clean}
Fix a finite horizon $S>0$ and a bulk index set $I_N$.  
Let $B=(B^{(n)})_{n\in I_N}$ be a vector of independent Brownian motions on
$(\Omega,\mathcal F,(\mathcal F_t)_{t\ge0},\mathbb P)$ and consider, for $t\in[0,S]$,
\begin{multline}\label{eq:DBM-plus-err}
d\widetilde X_n(t)
= dB^{(n)}_t \;+\; \sum_{m\neq n}\frac{1}{\widetilde X_n(t)-\widetilde X_m(t)}\,dt
\;+\; b^{\mathrm{err}}_{n}(t ; \delta_N)\,dt + \\ + \left [ b_n^{Sko}(t)+b_n^{1-body}(t) \right ]dt ,\qquad n\in I_N ,
\end{multline}
Set
\begin{equation}
\theta_t := \big(b^{\mathrm{err}}_{n}(t ; \delta_N)\big)_{n\in I_N}\in\mathbb R^{|I_N|}.
\end{equation}
Then the Girsanov tilted measure $\mathbb Q$ associated to $\theta_t$ is well defined and under $\mathbb Q$, the \emph{same} process $\widetilde X$ satisfies the standard DBM SDE
(with the error drift removed):
\begin{multline}\label{eq:DBM-clean}
d\widetilde X_n(t)
= dB^{\mathbb Q,(n)}_t \;+\; \sum_{m\neq n}\frac{1}{\widetilde X_n(t)-\widetilde X_m(t)}\,dt+ \\ + \left [ b_n^{Sko}(t)+b_n^{1-body}(t) \right ]dt,
\qquad n\in I_N,\ t\in[0,S],
\end{multline}
where $B^{\mathbb Q}_t:=B_t+\int_0^t \theta_s\,ds$ is a standard $|I_N|$–dimensional Brownian
motion under $\mathbb Q$. 
\end{cor}

\begin{proof} 
Set
\begin{equation}
Y_N \;:=\; \int_0^S \sum_{n\in I_N} \bigl|b^{\mathrm{err}}_{n}(t ;\delta_N)\bigr|^2\,dt .
\end{equation}
By \eqref{eq:304} we have
\begin{align}
Y_N
&\le 4\,\varepsilon(\delta_N)^2
\int_0^S \sum_{n\in I_N}
\mathbf{1}_{\{\Delta_n(t)\le 2\delta_N h_N\}}\,dt \\
&\le 4\,\varepsilon(\delta_N)^2
\int_0^S \sum_{n\in I_N} 1 \, dt
\;=\; 4\,\varepsilon(\delta_N)^2\,S\,|I_N| .
\end{align}
Hence the Novikov condition
\begin{equation}
\label{eq:Novikov}
\mathbb E_{\mathbb P}\!\left[
\exp\!\left(\tfrac12\int_0^S \sum_{n\in I_N}\big|b^{\mathrm{err}}_{n}(t ; \delta)\big|^2\,dt\right)
\right] \le\; \exp\!\big(2\,\varepsilon(\delta_N)^2\,S\,|I_N|\big)
\; <\infty 
\end{equation}
holds and the result follows by an application of Girsanov’s theorem.
\end{proof}

In particular, Corollary \ref{thm:girsanov-clean} shows that the two path laws on $[0,S]$ are
equivalent and differ only by this change of measure from $\mathbb{P}$ to $\mathbb{Q}$. We have the following property of the Radon-Nikodym derivative:
\begin{lemma}\label{lem:d} Let
\begin{equation}
M_t \;:=\; \exp\!\Big(-\!\int_0^t \theta_s\!\cdot dB^{\mathbb P}_s
               - \frac12\int_0^t \|\theta_s\|^2ds\Big).
\end{equation}
Then 
\begin{equation}
dM_t = -\,M_t\,\theta_t\!\cdot dB^{\mathbb P}_t \quad ; \quad 
d[M,M]_t = M_t^{\,2}\,\|\theta_t\|^2\,dt. \label{eq:QV}
\end{equation}
\end{lemma}

\begin{proof}
Set
\begin{equation}
N_t:=\int_0^t \theta_s\!\cdot dB^{\mathbb P}_s,
\qquad
[ N,N]_t = \int_0^t \|\theta_s\|^2ds .
\end{equation}
Then $M_t=e^{-N_t-\frac12\langle N\rangle_t}$.
Applying It\^o’s formula to $f(x,y)=e^{-x-\frac12 y}$ with $x_t=N_t$ and $y_t=[ N,N]_t$ we obtain 
\begin{equation}
\begin{aligned}
dM_t
&= \partial_x f\, dN_t + \partial_y f\, d[ N,N]_t
   + \frac12\,\partial_{xx} f\, d[N,N]_t \\
&= (-M_t)\, dN_t + \Big(-\frac12M_t\Big)\, d[ N,N]_t
   + \frac12 (M_t)\, d[N,N]_t = (-M_t)\, dN_t=-\,M_t\,\theta_t\!\cdot dB^{\mathbb P}_t.
\end{aligned}
\end{equation}
Since $(dB^{\mathbb P}_t)^2=dt$ we obtain
\begin{equation}
d[M]_t = M_t^{\,2}\,\|\theta_t\|^2\,dt,
\end{equation}
as required.
\end{proof}

The following result shows that in the large limit $N>>0$ the change of measure from $\mathbb{P}$ to $\mathbb{Q}$ does not affect the local statistics. 

\begin{theorem} \label{thm:girsanov-handoff}
For any bounded local observable $\Phi$ depending on finitely many coordinates
in the interior of $I_N$ the following holds 
\begin{equation}
\Big|\,
\mathbb E_{\mathbb{P}}\,\Phi(\widetilde X)
-\mathbb E_{\mathbb{Q}}\,\Phi(\widetilde X)\Big|
\;\xrightarrow[N\to\infty]{}\;0,
\end{equation}
In particular, for each fixed $t\in[0,S]$, local bulk gap distributions and
$k$–point correlation functions for \eqref{eq:DBM-plus-err} converge to those of
\eqref{eq:DBM-clean} which is the sine–kernel GUE limit.
\end{theorem}

\begin{proof}
Recall that for two probability measures $\mu,\nu$ on the same measurable space, the 
relative entropy of $\mu$ with respect to $\nu$ is given by their Kullback–Leibler divergence 
\begin{equation}
\mathrm{Ent}(\mu\!\mid\!\nu)
:= \int \log\!\left(M \right)\,d\mu = \mathbb{E}_{\mu} [log(M)]=\mathbb{E}_{\nu} [M \cdot \log(M)],
\end{equation}
whenever $\mu$ is absolutely continuous with respect to $\nu$ and $M=\frac{d\mu}{d\nu}$ is the Radon-Nikodym derivative. The Pinsker inequality then asserts that
\begin{equation}\label{eq:pinsker}
\|\mu-\nu\|_{1}
\;\le\; \sqrt{2 \, \mathrm{Ent}(\mu\!\mid\!\nu)},
\end{equation}
see Lemma 17.3.2 of \cite{CT}. 
\medskip
In our setting, the Radon-Nikodym derivative of the Girsanov’s tilted measure $\mathbb Q$ with respect to $\mathbb{P}$ is given by 
\begin{equation}
M_t:=\frac{d\mathbb Q}{d\mathbb P}
=\exp\!\left(-\int_0^t \theta_s\!\cdot dB^{\mathbb{P}}_s-\tfrac12\int_0^t\|\theta_s\|^2\,ds\right)
\end{equation}
where $\theta_t=(b^{\mathrm{err}}_{n,N}(t))_{n\in I_N}$.

In order to compute the relative entropy let us apply It\^o’s formula to $f(x)=x\log x$ at $x=M_t$. 
Since $f'(x)=\log x+1$ and $f''(x)=1/x$, we obtain
\begin{equation}
d(M_t\log M_t)=(\log M_t+1)\,dM_t+\tfrac12\frac{1}{M_t}\,d[ M,M]_t.
\end{equation}
Substituting Lemma \ref{lem:d} gives
\begin{equation}
d(M_t\log M_t)=-(\log M_t+1)M_t\,\theta_t\cdot dB_t
+\tfrac12\,M_t\|\theta_t\|^2dt.
\end{equation}
Taking expectations under $\mathbb P$ and using that the stochastic integral is a martingale and hence has mean zero, we obtain
\begin{equation}
\mathbb E_{\mathbb P}[M_S \log M_S]-M_0\log M_0
=\tfrac12\,\mathbb E_{\mathbb P}\!\int_0^S M_t\|\theta_t\|^2dt.
\end{equation}
Since $M_0\log M_0=0$, we get
\begin{equation}
\mathrm{Ent}(\mathbb Q \!\mid\! \mathbb P)=\mathbb E_{\mathbb P}[M_S\log M_S]
=\tfrac12\,\mathbb E_{\mathbb P}\!\int_0^S M_t\|\theta_t\|^2dt.
\end{equation}
By Cauchy--Schwarz inequality,
\begin{equation}
\label{eq:CS-333}
\mathbb E_{\mathbb P}\!\int_0^S M_t \|\theta_t\|^2\,dt
\;\le\;
\Big(\mathbb E_{\mathbb P} \,\!\big[\sup_{t\le S} M_t^2\big]\Big)^{1/2}
\Big(\mathbb E_{\mathbb P} \,\!\Big(\int_0^S \|\theta_t\|^2\,dt\Big)^2 \Big)^{1/2}.
\end{equation}
By Doob’s $L^2$ inequality,
\begin{equation}
\label{eq:Doob}
\mathbb E_{\mathbb P}\!\Big[\sup_{t\le S} M_t^2\Big]
\;\le\; 4\,\mathbb E_{\mathbb P}[M_S^2]
\;=\; 4\,\mathbb E_{\mathbb P} \, \!\exp\!\Big(\int_0^S \|\theta_t\|^2\,dt\Big)
\end{equation}
Hence substituting \eqref{eq:Doob} into \eqref{eq:CS-333} we get 
\begin{equation}
\label{eq:336}
\mathrm{Ent}(\mathbb Q \!\mid\! \mathbb P) \leq 2\exp\!\Big(  \mathbb E_{\mathbb P} \, \!\exp\!\Big(\int_0^S \|\theta_t\|^2\,dt\Big) \Big)^{\frac{1}{2}} \cdot \Big(\mathbb E_{\mathbb P} \,\!\Big(\int_0^S \|\theta_t\|^2\,dt\Big)^2 \Big)^{1/2}.
\end{equation}
Since by Proposition \ref{cor:L2-small} 
\begin{equation} 
\mathbb E_{\mathbb P}\!\int_0^S \|\theta_t\|^2\,dt
=\mathbb E_{\mathbb P} \left [ \left ( \!\int_0^S\sum_{n\in I_N}\big|b^{\mathrm{err}}_{n}(t)\big|^2\,dt \right )^2 \right ] \;\xrightarrow[N\to\infty]{\ \\ \ } 0 ,
\end{equation} 
we also have by \eqref{eq:336} and the Novikov condition \eqref{eq:Novikov},
\begin{equation}
\mathrm{Ent}(\mathbb Q \!\mid\! \mathbb P) \;\xrightarrow[N\to\infty]{\ \ \mathbb{P}\ \ } 0   
\end{equation}
Applying \eqref{eq:pinsker} with $\mu=\mathbb Q$ and $\nu=\mathbb P$ gives
\begin{equation}
\label{eq:340}
\|\mathbb Q-\mathbb P\|_{1}
\;\le\; \sqrt{\tfrac12\,\mathrm{Ent}(\mathbb Q\!\mid\!\mathbb P)}
\;\xrightarrow[N\to\infty]{} 0.
\end{equation}
Let $\lambda:=P+Q$ and write
$p:=\frac{dP}{d\lambda}$, $q:=\frac{dQ}{d\lambda}$. Then
\begin{equation}
\mathbb{E}_Q[\Phi]-\mathbb{E}_P[\Phi]
= \int_\Omega \Phi\,dQ - \int_\Omega \Phi\,dP
= \int_\Omega \Phi\,(q-p)\,d\lambda .
\end{equation}
Hence, by the Hölder inequality,
\begin{equation}
\bigl|\mathbb{E}_Q[\Phi]-\mathbb{E}_P[\Phi]\bigr|
\le \|\Phi\|_\infty \int_\Omega |q-p|\,d\lambda
= \|\Phi\|_\infty \int_\Omega |\,dQ-dP\,|
= \|\Phi\|_\infty\,\|Q-P\|_{1}.
\end{equation}
From \eqref{eq:340} we get 
\begin{equation}
\bigl|\mathbb{E}_Q[\Phi]-\mathbb{E}_P[\Phi]\bigr|
\;\xrightarrow[N\to\infty]{} 0,
\end{equation}
as required. 
\end{proof}

We are now in position to prove Theorem \ref{thm:A} as a summary of our results from the previous sections. 

\begin{cor}[GUE for $\mathcal{RH}_N(\mathbb{R})$]\label{cor:A}
As \(N \to \infty\), the average pair
correlation of the zeros of sections 
\(Z_N(t;\bar a) \in \mathcal{RH}_N(\mathbb{R})\), taken with respect to 
any admissible probability measure $\mu_N$ on
$\mathcal{RH}_N(\mathbb{R})$ with a smooth, strictly positive density
with respect to the Lebesgue measure on the coefficient space, restricted
to $\mathcal{RH}_N(\mathbb{R})$, converges to the GUE distribution. That is, for any Schwartz test 
function \(f \in \mathcal{S}(\mathbb{R})\),
\begin{equation}\label{eq:GUE-ensemble}
\lim_{N\to\infty}
\mathbb{E}_{\mu_N}\!\left[
\frac{1}{M_N}
\sum_{\substack{t_j,t_k\in[2N,2N+2]\\ j\ne k}}
f\!\left(\frac{\log N}{2\pi}\bigl(t_j(\bar a)-t_k(\bar a)\bigr)\right)
\right]
=
\int_{\mathbb{R}} f(x)
\left(1-\left(\frac{\sin\pi x}{\pi x}\right)^2\right)\,dx,
\end{equation}
where \(M_N \sim \tfrac{1}{\pi}\log\!\bigl(\tfrac{N}{\pi}\bigr)\) is the 
number of zeros \(\{t_j(\bar a)\}\) of \(Z_N(t;\bar a)\) in \([2N,2N+2]\).
\end{cor}

\begin{proof}
By the SDE representation \eqref{eq:DBM-reg-err}, on any bulk block
$I_N\subset[2N,2N+2]$ and fixed horizon $[0,S]$, the renormalized zeros
$\widetilde X(t)=(\widetilde X_n(t))_{n\in I_N}$ satisfy
\begin{equation}\label{eq:step-1}
d\widetilde X_n(t)
= d\hat\beta^{(n)}_t
+ \sum_{m \neq n}\frac{1}{\widetilde X_n(t)-\widetilde X_m(t)}\,dt
+ \left [b^{\mathrm{reg}}_n(t; \delta)+ b^{err}_n(t; \delta) +b_n^{Sko}(t)+b_n^{1-body}(t) \right ] \, dt,
\end{equation}
for some $\delta>0$ small enough. By Corollary~\ref{cor:breg-harmless}, the regular drift $b^{\mathrm{reg}}$ does not
affect bulk local statistics, hence we may drop it and work with
\begin{equation}\label{eq:step-2}
d\widetilde X_n(t)
= d\hat\beta^{(n)}_t
+ \sum_{m \neq n}\frac{1}{\widetilde X_n(t)-\widetilde X_m(t)}\,dt
+ \left [ b^{err}_n(t; \delta) +b_n^{Sko}(t)+b_n^{1-body}(t) \right ] \, dt.
\end{equation}
Next, by the whitening Theorem~\ref{prop:cov-replacement}, we may replace the martingale $d\hat\beta_t$ by independent standard Brownian increments $dB_t$ on $I_N$
without changing local limits. Thus we may assume
\begin{equation}\label{eq:step-3}
d\widetilde X_n(t)
= d B^{(n)}_t
+ \sum_{m \neq n}\frac{1}{\widetilde X_n(t)-\widetilde X_m(t)}\,dt
+ \left [ b^{err}_n(t; \delta) +b_n^{Sko}(t)+b_n^{1-body}(t) \right ] \, dt.
\end{equation}
By Corollary~\ref{thm:girsanov-clean}, the error drift
$b^{\mathrm{err}}$ can be removed 
so \eqref{eq:step-3} reduces to
\begin{equation}\label{eq:step-4}
d\widetilde X_n(t)
= d B^{(n)}_t
+ \sum_{m \neq n}\frac{1}{\widetilde X_n(t)-\widetilde X_m(t)}\,dt
+ \left [b_n^{Sko}(t)+b_n^{1-body}(t) \right ] \, dt.
\end{equation}
Theorem \ref{thm:girsanov-handoff} guarantees that the cost in relative entropy vanishes in the transition from \eqref{eq:step-3} to \eqref{eq:step-4} and hence local limit statistics do not change.

The Skorokhod term acts only when a coordinate hits the
boundary of $\mathcal{RH}_N(\mathbb{R})$ at the occurrence of a collision of two zeros. By the occupancy estimate of Theorem~\ref{lem:occ-time}, its contribution to bounded local
observables in the bulk vanishes as $N\to\infty$. We may therefore drop it and
without changing the local limit statistics and consider
\begin{equation}\label{eq:step-5}
d\widetilde X_n(t)
= dB^{(n)}_t + \sum_{m\neq n}\frac{1}{\widetilde X_n(t)-\widetilde X_m(t)}\,dt+b^{1-body}_n(t) dt ,
\qquad n\in I_N .
\end{equation}
By standard universality theory, one–body drifts such as $b^{1\text{-}body}$ come from smooth variations of the
background potential and only affect the global density profile. It does not alter the sine–kernel limit for bulk correlations, see \cite{AGZ,ErdosSchleinYau2011}.
We thus arrive at the classical DBM system:
\begin{equation}\label{eq:DBM-final}
d\widetilde X_n(t)
= dB^{(n)}_t + \sum_{m\neq n}\frac{1}{\widetilde X_n(t)-\widetilde X_m(t)}\,dt ,
\qquad n\in I_N .
\end{equation}
The bulk local statistics of \eqref{eq:DBM-final} are given by the GUE sine–kernel, by the modern universality theory results, as described in Subsection~\ref{ss:univ}.
\end{proof}

The following remark further clarifies the negligible role of the Skorokhod term in the SDE \eqref{eq:step-1}:

\begin{rem}[Negligibility of the Skorokhod term]
By definition, the Skorokhod reflection term
\(b^{\mathrm{Sko}}_n(t)=\tfrac{1}{\|\nabla t_n(A_t)\|}\,\langle \nabla t_n(A_t),\,dL_t\rangle\)
acts only on the boundary
\begin{equation}
\partial\mathcal{RH}_N(\mathbb{R})
= \bigcup_{n<m}\{\,\widetilde X_n=\widetilde X_m\,\}.
\end{equation}
However, the Coulomb drift term \(\sum_{m\neq n}(\widetilde X_n-\widetilde X_m)^{-1}\) repels trajectories from the boundary. In the limiting DBM, the collision set is polar, i.e. 
\begin{equation}
\mathbb{P}\big(\{\widetilde X_n(t)=\widetilde X_m(t)\ \text{for some } t>0\}\big)=0,
\end{equation}
for $n\ne m$, so that the local time vanishes identically and collisions almost surely do not occur. For the pre-limit dynamics, our small–gap occupation bound of Theorem~\ref{lem:occ-time} implies that the contribution of \(b^{\mathrm{Sko}}\) to bounded bulk observables tends to zero as \(N\to\infty\). The argument is in the same spirit as Remark~\ref{rem:2}. 

A similar argument applies to the one–body term
\(\tfrac{\|\nabla t_n(A_t)\|}{t_n(A_t)}\) of \eqref{eq:1-body}: for bulk indices $n$, one has $t_n(A_t)$ of order $N$, while $\|\nabla t_n(A_t)\|$ grows only logarithmically, cf.~\eqref{eq:nn}. Thus away from the boundary the one–body contribution is negligible in the limit $N\to\infty$.
\end{rem}

\section{Pair-Correlations of Randomized $Z$-functions are GUE}
\label{s:9.1}

\subsection{Annealed vs. Quenched Point-of-View} Strictly speaking, being GUE is a property of an ensemble, that is, a probability 
distribution on the space of Hermitian matrices $\mathcal{H}_N$, or, in our setting, on sections $Z_N(t;a)$ in $\mathcal{RH}_N(\mathbb{R})$, 
with Gaussian law on coefficients as proven in Theorem \ref{thm:A}. The GUE sine-kernel law refers to the limiting 
pair-correlation of eigenvalues, or zeros, under this distribution. For a single 
deterministic element, such as a fixed Hermitian matrix $H$, a section $Z_N(t;a)$, or, $Z(t)$ as in the Montgomery pair-correlation conjecture, one can only ask whether its configuration of eigenvalues or zeros resembles typical GUE behaviour. In random matrix theory 
this distinction appears as annealed versus quenched universality, where
annealed results concern averages over the ensemble, while quenched results discuss to which extent 
individual elements of the space exhibit GUE-type local statistics.

In the context of the Montgomery pair–correlation conjecture, our interest now lies in
quenched questions. Mainly, assuming RH, whether the specific deterministic section
\(Z_N(t;\bar 1)\approx Z(t)\) already exhibits GUE–type local statistics without any
ensemble averaging within the
\(\mathcal{RH}_N(\mathbb R)\). Our viewpoint is that the Hardy function \(Z(t)\) 
already contains the required averaging intrinsically within itself, 
when interpreted through its approximations in the $A$--variation spaces.  
For each \(N \in \mathbb{N} \), recall that \(Z(t)\) admits the finite approximation \eqref{eq:approx}
\begin{equation}
Z(t)\;\approx\; Z_N(t;\bar 1)
=\sum_{k=1}^N \frac{1}{\sqrt{k+1}}\,
\cos\!\bigl(\theta(t)-t\log(k+1)\bigr)
\;\in\; \mathcal{Z}_N(\mathbb{R}),
\end{equation}
valid on the local window \([2N,2N+2]\).  
Assuming the Riemann Hypothesis, we have \(Z_N(t;\bar 1)\in\mathcal{RH}_N(\mathbb{R})\) for any $N \in \mathbb{N}$, 
see~\cite{J4}.  

\begin{rem}[On the choice of window boundaries]
\label{rem:window-flexibility}
It should be noted at this point that the specific choice of the window $[2N,2N+2]$ in the definition of $\mathcal{RH}_N(\mathbb{R})$ is non-essential. One may instead consider more
flexible intervals of the form
\begin{equation} 
[\,2N-\alpha,\,2N+2+\beta\,],
\qquad
\alpha,\beta \geq 0,
\end{equation} 
provided that the endpoints remain within the same logarithmic scale of
$2N$, so the approximate functional equation $Z(t) \approx Z_N(t; \bar 1)$ remains valid. In particular, in Theorem~A the fixed boundaries $[2N,2N+2]$ were chosen merely for
convenience, to establish the GUE nature of the
space~$\mathcal{RH}_N(\mathbb R)$.
When dealing with the true Hardy function~$Z(t)$, one might worry that some
zeros of $Z_N(t;1)$ near the boundary could drift slightly outside the
prescribed window.
However, for large~$N$ the bulk of zeros in each window remains well inside
its interior, and any such boundary effects can be safely absorbed by a small
thickening of the interval via a sufficent increase of $\alpha, \beta>0$.
Consequently, the precise placement of the window boundaries is immaterial to
the asymptotic behaviour of the pair–correlation functionals.
\end{rem}

In the random matrix setting, the statement that the Hermitian matrices are 
``GUE--distributed'' can be understood in two complementary senses.  
First, in the ensemble sense, one considers the expectation of observables 
with respect to the measure on the full space $\mathcal{H}_N$, and the 
GUE sine--kernel law arises as the limit of the ensemble, averaged 
pair correlation of eigenvalues.  
The analogue of this viewpoint in our framework is precisely the result of 
Theorem~\ref{thm:A}, which establishes the GUE law for the ensemble of sections 
in the real chamber of the $A$--variation space~$\mathcal{RH}_N(\mathbb{R})$.

Second, an equivalent viewpoint is through Monte--Carlo realization.  
Once it is known that a given space is GUE, if one draws a sufficiently large number of independent Hermitian matrices 
$H^{(1)}, H^{(2)}, \ldots, H^{(N)}$ of increasing size from the GUE distribution and computes, for each, 
the empirical pair correlation of eigenvalues, then the average of these samples 
converges to the same GUE law. In this sense, a Monte--Carlo simulation provides a finite approximation to the ensemble average, 
and the universality of the GUE law expresses the stability of the limiting statistics 
under such sampling procedures, with accuracy improving as $N \rightarrow \infty$.  
It is precisely this second, dynamical view-point that becomes relevant in our analytic interpretation
of the pair--correlation conjecture, where the sequence of sections $\{Z_N(t;\bar 1)\}_{N \in \mathbb{N}}$ 
is viewed as a deterministic sequence of ``samples'' 
traversing the analytic ensemble $\mathcal{RH}_N(\mathbb{R})$. We first consider a randomized version of this question for section $\{Z_N(t;\bar a^{(N)})\}_{N \in \mathbb{N}}$ with $\bar a^{(N)}$ chosen independently according to the measure $\mu_N$ restricted to $\mathcal{RH}_N(\mathbb{R})$.  

\subsection{The Pair-Correlation conjecture for Random Sequences} Let us make the following definitions: 

\begin{dfn}[Local and global pair–correlation functionals]
Let $f\in \mathcal{S}(\mathbb R)$ be a Schwartz test function.
\begin{itemize}

\item[(i)] We refer to 
\begin{equation}
\label{eq:local-PC}
\mathrm{PC}_N(f;\bar a)
:=\frac{1}{M_N}\!\!\sum_{\substack{t_j(\bar a),\,t_k(\bar a)\in[2N,2N+2]\\ j\neq k}}
f\!\left(\frac{\log N}{2\pi}\bigl(t_j(\bar a)-t_k(\bar a)\bigr)\right),
\end{equation}
as the local pair-correlation functional of $Z_N(t; \bar a) \in \mathcal{RH}_N(\mathbb{R})$ in the window $[2N,2N+2]$. 
 
\item[(ii)] Assume $Z_N(t;\bar a^{(N)})\in\mathcal{RH}_N(\mathbb R)$ for every $N$, we refer to to
\begin{equation}
PC^{approx}(f; \left \{ \bar a^{(N)} \right \} , T)
:=\frac{1}{N(T)}\sum_{2N\le T} M_N\,\mathrm{PC}_N\!\bigl(f;\bar a^{(N)}\bigr),
\end{equation} 
where $N(T)\sim \tfrac{T}{2\pi}\log\frac{T}{2\pi}$,  as the approximate pair-correlation function of the sequence $\left \{ Z_N(t;\bar a^{(N)}) \right \}_{N \in \mathbb{N}}$. 

\item[(iii)] Let 
\begin{equation}
\Gamma_T=\bigcup_{2N\le T}\bigl\{t_j(\bar a^{(N)})\in[2N,2N+2]\bigr\}.
\end{equation} 
be the union of all zeros of the sections
$\{Z_N(t;\bar a^{(N)})\}_{2N\le T}$ lying in $[0,T]$. We refer to 
\begin{equation}\label{eq:R2-genuine}
PC^{gen}(f; \left \{ \bar a^{(N)} \right \}, T)
:=\frac{1}{N(T)}
\!\!\sum_{\substack{\gamma,\gamma'\in\Gamma_T\\ \gamma\ne\gamma'}}
f\!\left(\frac{\log(T/2\pi)}{2\pi}\,(\gamma-\gamma')\right),
\end{equation}
as the genuine global pair–correlation functional of the sequence  $\left \{ Z_N(t;\bar a^{(N)}) \right \}_{N \in \mathbb{N}}$. 

\item[(iv)] We refer to 
\begin{equation}\label{eq:R2-Z}
PC^{Z}(f; T)
:=\frac{1}{N(T)}
\!\!\sum_{\substack{\gamma,\gamma' \leq T\\ \gamma\ne\gamma'}}
f\!\left(\frac{\log(T/2\pi)}{2\pi}\,(\gamma-\gamma')\right),
\end{equation} where the sum runs of zeros of $Z(t)$, as the pair-correlation functional of the Hardy $Z$-function. 
\end{itemize}
\end{dfn}
By definition, the pair–correlation conjecture asserts that, asymptotically,
\begin{equation}
\lim_{T\to\infty}\mathrm{PC}^{Z}(f;T)
=\mathrm{GUE}(f),
\end{equation}
where
\begin{equation}
\mathrm{GUE}(f)
=\int_{\mathbb{R}} 
f(x)\!\left(1-\left(\frac{\sin \pi x}{\pi x}\right)^{\!2}\right)\!dx.
\end{equation}
Moreover, in view of our results in \cite{J4} we have 
\begin{equation}
\lim_{T\to\infty}\mathrm{PC}^{Z}(f;T)
= \lim_{T\to\infty}\mathrm{PC}^{gen}(f; \bar 1, T),
\end{equation}
Hence, for us the pair-correlation conjecture is essentially henceforth 
\begin{equation}
\lim_{T\to\infty}\mathrm{PC}^{gen}(f; \bar 1, T)
=\mathrm{GUE}(f).
\end{equation}
Let us note the following remark: 
\begin{rem}[The point $\bar 1$ versus the accelerated coefficients $\bar a^{\mathrm{acc}}_N$]
\label{rem:aacc}
While $Z_N(t;\bar 1)$ was chosen for convenience and clarity,
one may, for greater numerical accuracy, replace $\bar 1$ throughout this
section by the accelerated coefficient vector $\bar a^{\mathrm{acc}}_N$
defined in~\eqref{eq:acc-co}.
According to~\eqref{eq:acc}, this choice provides a highly accurate
approximation of the full Hardy function~$Z(t)$, and all results presented
herein remain valid under this substitution.
In particular, for sufficiently large~$T$, the zeros of
$Z_N(t;\bar a^{\mathrm{acc}}_N)$ are effectively indistinguishable from those
of~$Z(t)$.
\end{rem}

The following result shows that the approximate global pair-correlation is asymptotically similar to the genuine pair-correlation: 
 
\begin{prop}[$PC^{approx}$ approximates $PC^{gen}$]
\label{prop:approx-vs-gen}
Let $f\in \mathcal{S}(\mathbb R)$ be a Schwartz test function. Then, asymptotically the approximate and genuine pair-correlations coincide 
\begin{equation}
\lim_{T\to\infty}\mathrm{PC}^{gen}(f;  \left \{ \bar a^{(N)} \right \}, T)=
\lim_{T\to\infty}\mathrm{PC}^{approx}(f;  \left \{ \bar a^{(N)} \right \} , T),
\end{equation}
for any sequence $Z_N(t; \bar a^{(N)}) \in \mathcal{RH}_N(\mathbb{R})$. 
\end{prop}
\begin{proof}
Let us first assume $supp(f) \subset [-C,C]$. Partition $[0,T]$ into half–open windows $I_N:=[2N,2N+2)$.
Decompose the genuine global sum into intra–window and inter–window parts:
\begin{multline}
\mathrm{PC}^{\mathrm{gen}}\!\bigl(f;\{\bar a^{(N)}\},T\bigr)
=\frac{1}{N(T)}\sum_{2N\le T}
\!\!\sum_{\substack{\gamma,\gamma'\in\Gamma_T\cap I_N\\ \gamma\ne\gamma'}}
f\!\left(\frac{\log(T/2\pi)}{2\pi}(\gamma-\gamma')\right)
\;+ \\ + \;E_{\mathrm{cross}}(f;\{\bar a^{(N)}\}, T).
\end{multline}
The proof is based on showing the following:

\begin{enumerate}
\item The intra-window main sum asymptotically coincides with $PC^{approx}$, \begin{equation}
\frac{1}{N(T)}\sum_{2N\le T}
\!\!\sum_{\substack{\gamma,\gamma'\in\Gamma_T\cap I_N\\ \gamma\ne\gamma'}}
f\!\left(\frac{\log(T/2\pi)}{2\pi}(\gamma-\gamma')\right)
=
\mathrm{PC}^{\mathrm{approx}}\!\bigl(f;\{\bar a^{(N)}\},T\bigr)
+O\!\left(\frac{1}{\log T}\right).
\end{equation}
\item The inter-window sum is asymptotically vanishing: 
\begin{equation}
|E_{\mathrm{cross}}(f;T)|= O \left( \frac{1}{\log T}\right ).
\end{equation}
\end{enumerate}
For (1) write
\begin{equation}
s_N=\frac{\log N}{2\pi},\qquad s_T=\frac{\log(T/2\pi)}{2\pi}.
\end{equation}
Since $\mathrm{supp}\,f\subset[-C,C]$,
a contributing pair must satisfy $s_T|\gamma-\gamma'|\le C$, and hence
$|\gamma-\gamma'|\ll \frac{1}{\log T}$. By the mean–value theorem,
\begin{equation}
\Big|f\!\big(s_T(\gamma-\gamma')\big)-f\!\big(s_N(\gamma-\gamma')\big)\Big|
\le \|f'\|_\infty\,|\gamma-\gamma'|\,|s_T-s_N|
\ll \frac{\log(T/N)}{\log T}.
\end{equation}
Next, the number of pairs of zeros in $I_N$ satisfying $|\gamma - \gamma'| \ll 1/\log T$ is 
\begin{equation}
O\!\left(M_N \cdot \frac{\log N}{\log T}\right)
= O\!\left(\frac{(\log N)^2}{\log T}\right),
\end{equation}
since $I_N$ contains $M_N = O(\log N)$ zeros, and the local zero density is also $O(\log N)$.
Therefore, the per–window rescaling error is 
\begin{equation}
O\!\left(\frac{(\log N)^2}{\log T}\right)
\times
O\!\left(\frac{\log(T/N)}{\log T}\right)
= O\!\left(\frac{(\log N)^2 \, \log(T/N)}{(\log T)^2}\right).
\end{equation}
Summing over $2N \le T$ and dividing by $N(T) = O(T \log T)$ yields the normalized
rescaling error
\begin{equation}
\frac{1}{N(T)} \sum_{2N \le T}
O\!\left(\frac{(\log N)^2 \, \log(T/N)}{(\log T)^2}\right)
= O\!\left(\frac{T \cdot (\log T)^2}{T (\log T)^2}\right)
= O\!\left(\frac{1}{\log T}\right),
\end{equation}
since the average of $(\log N)^2$ over $1 \le N \le T$ is $O((\log T)^2)$
and the average of $\log(T/N)$ is $O(1)$. Hence
\begin{equation}
\frac{1}{N(T)}\sum_{2N\le T}
\!\!\sum_{\substack{\gamma,\gamma'\in\Gamma_T\cap I_N\\ \gamma\ne\gamma'}}
f\!\left(\frac{\log(T/2\pi)}{2\pi}(\gamma-\gamma')\right)
=
\mathrm{PC}^{\mathrm{approx}}\!\bigl(f;\{\bar a^{(N)}\},T\bigr)
+O\!\left(\frac{1}{\log T}\right),
\end{equation}
which gives (1). 
For (2), as before, a pair from distinct windows contributes only if
$|\gamma-\gamma'|\ll 1/\log T$. Hence, pairs form non-adjacent windows do not contribute. For
adjacent windows $I_N$ and $I_{N+1}$, both zeros must lie in
$O(1/\log T)$–neighbourhoods of the common boundary $2N+2$. Each side contributes
$O(1)$ zeros, hence $O(1)$ pairs per boundary, so with $O(T)$ boundaries
\begin{equation}
|E_{\mathrm{cross}}(f;T)|
= O\!\left(\frac{\|f\|_\infty}{N(T)} \, T \right)
= O\!\left(\frac{T}{T \log T}\right)
= O\!\left(\frac{1}{\log T}\right)
\;\xrightarrow[T \to \infty]{}\; 0.
\end{equation}
Combining the steps (1) and (2) gives
\begin{equation}
\mathrm{PC}^{\mathrm{gen}}\!\bigl(f;\{\bar a^{(N)}\},T\bigr)
=
\mathrm{PC}^{\mathrm{approx}}\!\bigl(f;\{\bar a^{(N)}\},T\bigr)
+O\!\left(\frac{1}{\log T}\right),
\end{equation}
and hence
\begin{equation}
\lim_{T\to\infty}
\mathrm{PC}^{\mathrm{gen}}\!\bigl(f;\{\bar a^{(N)}\},T\bigr)
=
\lim_{T\to\infty}
\mathrm{PC}^{\mathrm{approx}}\!\bigl(f;\{\bar a^{(N)}\},T\bigr).
\end{equation}
\end{proof}

We will need the following weighted variation of the strong law of large numbers:

\begin{prop}[Weighted strong law of large numbers]
\label{prop:SLLN}
Let $(Y_n)_{n\ge1}$ be independent random variables with $\mathbb{E}[Y_n]=0$ and 
$\mathbb{E}[Y_n^2]<\infty$.  
Let $(w_n)_{n\ge1}$ be positive weights and define
\begin{equation}
W_N := \sum_{n=1}^N w_n \quad \text{with} \quad W_N \to \infty.
\end{equation}
Assume that
\begin{equation}\label{eq:var-sum}
\sum_{N=1}^\infty \frac{w_N^2\,\mathrm{Var}(Y_N)}{W_N^2}<\infty
\qquad\text{and}\qquad
\max_{k\le N}\frac{w_k}{W_N}\xrightarrow[n\to\infty]{}0.
\end{equation}
Then
\begin{equation}\label{eq:goal}
\frac{1}{W_N}\sum_{k=1}^N w_k Y_k \xrightarrow[\mathrm{a.s.}]{} 0 .
\end{equation}
\end{prop}

\begin{proof}
Define
\begin{equation}\label{eq:Zn}
Z_N := \frac{w_N}{W_N}\,Y_N , \qquad N \ge1 .
\end{equation}
Then $(Z_n)$ are independent, centered, and satisfy
\begin{equation}
\mathbb{E}[Z_N^2] = \frac{w_N^2}{W_N^2}\,\mathrm{Var}(Y_N).
\end{equation}
By the assumption \eqref{eq:var-sum}, we have
\begin{equation}
\sum_{N=1}^\infty \mathbb{E}[Z_N^2] < \infty.
\end{equation}
Hence, by the classical Kolmogorov convergence criterion,
the series $\sum_{n=1}^\infty Z_n$ converges almost surely. Let $S_m := \sum_{k=1}^m Z_k$.  
Since $S_m$ converges a.s., the classical Kronecker's lemma \cite{ChowTeicher} gives
\begin{equation}
\frac{1}{W_N}\sum_{k=1}^N W_k Z_k \xrightarrow[\mathrm{a.s.}]{} 0 .
\end{equation}
But from \eqref{eq:Zn},
\begin{equation}
\sum_{k=1}^N W_k Z_k = \sum_{k=1}^N w_k Y_k,
\end{equation}
and therefore
\begin{equation}
\frac{1}{W_N}\sum_{k=1}^N w_k Y_k \to 0 \quad \text{a.s.}
\end{equation}
\end{proof}

The following result is Theorem \ref{thm:B1} which shows that if $Z_N(t; \bar a^{(N)}) \in \mathcal{RH}_N(\mathbb{R})$ is chosen randomly then the pair-correlation of the sequence is asymptotically GUE. 

\begin{theorem}[GUE for random chains]
\label{prop:random-chain-GUE}
For each $N$, assume $Z_N(t ; \bar a^{(N)})$ is drawn independently from the measure on $\mathcal{RH}_N(\mathbb{R})$. 
Then, with probability~$1$,
\begin{equation}
\lim_{T\to\infty}\mathrm{PC}^{gen}(f; 
\{ \bar a^{(N)} \}, T)
=\mathrm{GUE}(f).
\end{equation}
\end{theorem}

\begin{proof}
Write $X_N:=\mathrm{PC}_N\!\bigl(f;\bar a^{(N)}\bigr)$. 
By Theorem \ref{thm:A} we have
\begin{equation}\label{eq:mean-to-GUE}
m_N := \mathbb{E}[X_N]\ \xrightarrow[N\to\infty]{} \mathrm{GUE}(f).
\end{equation}
Since $f$ is bounded, each term in the local pair–correlation sum satisfies 
$|f|\le \|f\|_\infty$. 
As there are $M_N = O(\log N)$ zeros in the window $[2N,2N+2]$, we obtain
\begin{equation}
|X_N| \ll \|f\|_\infty\, \log N,
\qquad
\mathrm{Var}(X_N)=O\!\left((\log N)^2\right).
\end{equation}
Clearly  
\begin{equation}
\max_{2N\le T} \frac{M_N}{N(T)} \to 0
\end{equation} 
because $M_N = O(\log N)$ while $N(T)=O(T\log T)$. 

Consider the centred independent variables 
$Y_N := X_N - m_N$ with $\mathbb{E}[Y_N]=0$ and $\mathbb{E}[Y_N^2]<\infty$.
Let us take $w_N=M_N$ and $W_T=\sum_{2N\le T} w_N$. We have
\begin{equation}
\frac{w_N^2\,\mathrm{Var}(Y_N)}{W_N^2}
= O\!\left(\frac{(\log N)^4}{(N\log N)^2}\right)
= O\!\left(\frac{(\log N)^2}{N^2}\right),
\end{equation}
and $\sum_N (\log N)^2/N^2<\infty$. 
Hence, the conditions of the weighted strong law of large numbers of Proposition \ref{prop:SLLN} are satisfied, and thus 
\begin{equation}
\frac{1}{W(T)}\sum_{2N\le T} w_N Y_N
\ \xrightarrow[\text{a.s.}]{}\ 0.
\end{equation}
Since $W(T)\sim N(T)$, we obtain
\begin{equation}\label{eq:weighted-to-expectation}
\mathrm{PC}^{\mathrm{gen}}(f;\{\bar a^{(N)}\},T)
=\frac{1}{N(T)}\sum_{2N\le T} w_N X_N
=\frac{1}{N(T)}\sum_{2N\le T} w_N m_N
\;+\;o(1)
\qquad\text{a.s.}
\end{equation}
Recall that according to the Toeplitz summation theorem if $(a_n)$ is any sequence with $a_n\to A$ and $(p_{n,k})$ is a triangular array of nonnegative numbers such that  
$\sum_{k=1}^n p_{n,k}=1$ for each $n$ and $\max_k p_{n,k}\to0$,  
then the weighted averages
\begin{equation}
b_n := \sum_{k=1}^n p_{n,k} a_k
\end{equation}
also converge to the same limit $A$, see \cite{Feller1971}. Let us take $a_N=m_N$, and define
\begin{equation}
p_{T,N} := \frac{w_N}{W(T)}, 
\end{equation}
for $2N\le T$. 
These satisfy $p_{T,N}\ge0$, $\sum_{2N\le T}p_{T,N}=1$, and 
$\max_{2N\le T}p_{T,N}\to0$ as $T\to\infty$.  
Hence, by the Toeplitz summation theorem,
\begin{equation}\label{eq:toeplitz}
\frac{1}{W(T)}\sum_{2N\le T} w_N m_N 
\;\xrightarrow[T\to\infty]{}\;
\mathrm{GUE}(f).
\end{equation}

Finally, since $W(T)\sim N(T)$, we may replace $W(T)$ by $N(T)$ in the normalization without changing the limit.  
Substituting~\eqref{eq:toeplitz} into~\eqref{eq:weighted-to-expectation} yields
\begin{equation}
\mathrm{PC}^{\mathrm{gen}}(f;\{\bar a^{(N)}\},T)
\;\xrightarrow[\mathrm{a.s.}]{}\;
\mathrm{GUE}(f),
\end{equation}
which completes the proof.
\end{proof}

\newcommand{\PC}{\mathrm{PC}}
\newcommand{\one}{\overline{1}}
\newcommand{\I}{I}
\newcommand{\bfa}{\bar a}
\newcommand{\ip}[2]{\left\langle #1,\,#2\right\rangle}

\section{The Proof of the Pair Correlation Conjecture for the Hardy Function}
\label{s:9}

\subsection{From Ensemble Universality to the Pair Correlation of the Hardy Function}
In this section we explain how the GUE universality established in Theorem~A
extends from the variational ensemble $\mathcal{RH}_N(\mathbb{R})$ to the actual Hardy
function $Z(t)$. Once the ensemble statistics are known to coincide with the GUE law,
the remaining task is to show that the canonical section $Z_N(t;1)$ behaves, over large disjoint windows, like a typical randomized member $Z_N(t;\bar a^{(N)})$ of $\mathcal{RH}_N(\mathbb{R})$.  

It should be noted that not every element of the ensemble displays GUE--like spacing. For instance, 
the identity matrix in $\mathcal{H}_N$ or the core section $Z_0(t)=\cos(\theta(t))$ both admit perfectly ordered, evenly spaced zeros within each window.
From the viewpoint of the one--dimensional Coulomb gas analogy, these correspond to “crystalline’’ zero–temperature configurations in which the repulsive forces between neighboring charges exactly balance.
Such rigid structures are dynamically unstable and under small perturbations or under Dyson Brownian motion they “melt’’ and evolve toward the universal GUE equilibrium, as shown for instance in the universality results of~\cite{ErdosYau2012,BourgadeErdosYau2014}.  

To ensure that the deterministic sections $Z_N(t;\bar 1)$ display the same effective randomness as the ensemble, one needs a mechanism enforcing statistical independence between distant spectral windows.
Theorem~\ref{thm:Sel-dec} provides precisely this mechanism as it establishes that the local pair--correlation functionals $PC_N(f;\bar1)$ decorrelate across disjoint intervals, ensuring that the zeros of the Hardy function exhibit the same macroscopic independence that underlies GUE universality.

\begin{theorem}[Decorrelation of $PC_N(f; \bar 1)$]
\label{thm:Sel-dec}
For every fixed Schwartz test function $f$, the sequence $\{PC_N(f;\bar 1)\}_{N \in \mathbb{N}}$ has uniformly bounded variance and
satisfies the weak decorrelation condition
\begin{equation}
\label{eq:Sel-dec}
\mathrm{Cov}\!\left(PC_N(f;\bar 1),PC_{N'}(f;\bar 1)\right)\to 0,
\end{equation}
as $|N-N'|\to\infty$.
\end{theorem}

The truncations $Z_N(t; \bar 1)$ and the full Hardy function $Z(t)$ share 
essentially the same zeros on the interval $I_N=[2N,2N+2]$, 
since the remainder of the approximate formula \eqref{eq:acc} contributes only an exponentially small perturbation to their zero locations, see Remark \ref{rem:aacc}.  
It is therefore legitimate to replace the analysis of the zeros of $Z_N(t;\bar 1)$ by the classical, 
well–established theory of the zeros of $Z(t)$ itself.

\begin{rem}[Statistical–physics interpretation]
Although each $PC_N(f;\bar 1)$ is a single scalar quantity, it aggregates the 
pairwise interactions of $M_N\!\to\!\infty$ zeros of $Z_N(t;\bar 1)$ within the 
window $I_N=[2N,2N+2]$. In the context of $Z_N(t; \bar 1)$, these zeros can be viewed as originating from the perfectly structured 
configuration of $Z_0(t)=\cos(\theta(t))$, whose zeros in $I_N$ are evenly spaced, 
and are progressively perturbed by the oscillatory components 
$\cos(\theta(t)-t\log(k+1))/\sqrt{k+1}$ for $k=1,\dots,N$.  
Each new frequency $\log(k+1)$ introduces an incommensurable phase, adding 
a new layer of microscopic irregularity to the zero configuration.  
As $N$ grows, the cumulative effect of these quasi–independent perturbations 
renders the fluctuations of distinct zeros effectively chaotic, 
so that the global functional $PC_N(f;\bar 1)$ behaves as the macroscopic sum of 
infinitely many weakly correlated contributions. 
\end{rem}

The proof of Theorem~\ref{thm:Sel-dec} is based on Selberg’s theory on the fluctuations of the argument function
\begin{equation}
S(t)=\frac{1}{\pi}\arg\zeta\!\left(\tfrac{1}{2}+it\right),
\end{equation}
see \cite{Selberg1946,Selberg1947,
Selberg1989CLT,
SelbergCollectedPapers1989}. Selberg's results describe the local oscillations of $Z(t)$ as an almost Gaussian process
with rapidly decaying correlations on disjoint intervals.
This analytic structure gives a precise counterpart to the
statistical--physics picture discussed above. While in the Coulomb--gas model
decorrelation arises from thermal agitation or Dyson Brownian motion,
here it emerges from the intrinsic pseudo--randomness of the argument increments of $\zeta(s)$.
The proof of Theorem~\ref{thm:Sel-dec} therefore rests on Selberg’s covariance estimates for $S'(t)$,
which formalize the near--independence of the underlying oscillations and ensure that the
local functionals $PC_N(f;\bar 1)$ behave as asymptotically independent variables.
The rest of this section develops this argument in detail.

\begin{enumerate}
\item In subsection~\ref{ss:9.2},  
we review Selberg’s statistical theory for the argument function 
\( S(t)=\tfrac{1}{\pi}\arg\zeta(\tfrac{1}{2}+it) \),  
describing its Gaussian fluctuations and the decay of correlations in its derivative \(S'(t)\).

\item In subsection~\ref{ss:9.3},  
we reinterpret Selberg’s $L^2$ correlation estimates in the language of stationary Gaussian processes, 
and recall the strong–mixing coefficient $\alpha(u)$ together with the 
$\alpha$–covariance bound~\eqref{eq:IL-bound}, 
which quantify the asymptotic independence of disjoint segments of the process, following \cite{IbragimovLinnik1971}.

\item In subsection~\ref{ss:9.4},  
we reformulate the local pair--correlation statistic \(PC_N(f;\bar 1)\)  
as a smooth bounded functional \(F_{f,N}(S'|_{I_N})\)  
of the restriction of the derivative process \(S'(t)\)  
to the window \(I_N=[2N,2N+2]\).  
Using Selberg’s covariance estimates, we show that \(S'(t)\)  
is a strongly mixing process, implying asymptotic independence  
of such local functionals across distant intervals.  
A uniform bound on \(\|PC_N(f;\bar 1)\|_\infty\)  
then allows the application of the
$\alpha$–covariance bound~\eqref{eq:IL-bound},  
establishing the weak decorrelation property required in Theorem~\ref{thm:Sel-dec}.  
For completeness, we also indicate an alternative proof based on  
the covariance identity for Gaussian functionals,  
derived from Selberg’s empirical covariance kernel.

\item Finally, in subsection~\ref{ss:9.6},  
we combine the decorrelation of \(PC_N(f;\bar1)\) established in Theorem~\ref{thm:Sel-dec}  
with the ensemble universality result of Theorem~\ref{thm:A}  
and the GUE law for randomized sections proved in Theorem~\ref{thm:B1},  
to complete the proof of Theorem~\ref{thm:B}.
\end{enumerate}

\subsection{Selberg’s Statistical Theory of $S(t)$} 
\label{ss:9.2} The statistical fluctuations of the zeros of $\zeta(s)$ are encoded in the function
\begin{equation}
S(t)=\frac{1}{\pi}\arg\zeta\!\left(\tfrac{1}{2}+it\right),
\end{equation}
called the argument of the Riemann zeta function. Here $\arg\zeta(\tfrac{1}{2}+it)$ denotes the increment of an arbitrary continuous branch of $\arg\zeta(s)$ along the broken line starting at $2$, where the argument is zero, going vertically to $2+it$ and then reaching $\tfrac{1}{2}+it$ horizontally. At a zero $\gamma$ one takes the symmetric limit
\begin{equation}
S(\gamma)
=\lim_{\delta\to+0}\frac{1}{2}\bigl(S(\gamma+\delta)+S(\gamma-\delta)\bigr).
\end{equation}
The function $S(t)$ is a slowly varying and highly irregular function. Modulo integers, the function $S(t)$ coincides with the principal value of the argument of~$\zeta(\tfrac12+it)$.  However, in the definition of $S(t)$ one uses a continuous, unwrapped, branch of the argument obtained by analytic continuation along the path from $s=2$ to $\tfrac12+it$, prescribed above. As a result, each time $Z(t)=e^{i\theta(t)}\zeta(\tfrac12+it)$ crosses a zero $\gamma$, the function $\arg\zeta(\tfrac12+it)$ changes discontinuously by~$\pm\pi$, hence $S(t)$ jumps by~$\pm1$.  
The sign of this jump is determined by the sign of the derivative,
\begin{equation}
\Delta S \left (\gamma \right )
=\frac{1}{\pi}\Delta(\arg Z)
=\operatorname{sgn}\,\bigl(Z'\left (\gamma \right )\bigr).
\end{equation}
If $Z'(\gamma)>0$ the curve $\zeta(\tfrac12+it)$ crosses the real axis counter-clockwise and $S(t)$ increases by~$+1$, and if $Z'(\gamma)<0$ it crosses clockwise and $S(t)$ decreases by~$1$.

A simple model illustrates this mechanism by noting that the core function $Z_0(t)$ has zeros at the points $\theta(t_n)=(n+\tfrac12)\pi$, with 
\begin{equation}
Z_0'(t_n)=-\theta'(t_n)\sin\, \left ( \theta(t_n) \right ) 
=(-1)^n\,\theta'(t_n),
\end{equation}
so the sign of $Z_0'(t_n)$ alternates as $(-1)^n$.  
Hence a corresponding argument function for the core function
\begin{equation}
S^{\mathrm{core}}(t)=\frac{1}{\pi}\arg\bigl(e^{-i\theta(t)}Z_0(t)\bigr)
\end{equation}
would produce the perfectly alternating pattern of jumps
\begin{equation}
\Delta S^{\mathrm{core}}(t_n)=(-1)^n.
\end{equation}
However, for the true zeta function one has 
\begin{equation}
S(t)=-\frac{\theta(t)}{\pi}+\frac{1}{\pi}\arg Z(t).
\end{equation}
Since the zeros of $Z(t)$ occur slightly off the ideal positions $\theta(t)=\left ( n +\frac{1}{2} \right ) \pi$, their signs $Z'(\gamma_n)$ are not perfectly alternating, producing the irregular $\pm1$ jumps in $S(t)$, which can accumulate in a non-trivial manner. 

The function $S(t)$ first appeared implicitly in Riemann’s 1859 memoir \cite{R} and
was made explicit by von~Mangoldt in his derivation of the Riemann-von~Mangoldt formula,
\begin{equation}\label{eq:RvM}
N(T)
=\frac{T}{2\pi}\log\frac{T}{2\pi}
-\frac{T}{2\pi}
+\frac{7}{8}
+S(T)
+O\!\left(\frac{1}{T}\right).
\end{equation}
Here $N(T)$ denotes the number of nontrivial zeros of $Z(t)$ with $0<t \le T$.  
The main terms describes the smooth asymptotic density of zeros, 
while the remainder $S(T)$ captures their fine irregularities, we refer the reader to the classical texts \cite{E,T}.  
Classically, one proves that $S(T)=O(\log T)$, 
and, on average, its size is much smaller.

In a series of seminal papers \cite{Selberg1946,Selberg1947,
Selberg1989CLT,
SelbergCollectedPapers1989}, 
Selberg developed a probabilistic theory of the Riemann zeta function, 
in which the argument function 
$S(t)$ is treated as a centered Gaussian process with
\begin{equation}
\mathrm{Var}\,S(t)=\frac{1}{2\pi^2}\log\log t+O(1).
\end{equation}
Its derivative,
\begin{equation}
S'(t)
=-\frac{1}{\pi}\sum_{n=1}^{\infty}\frac{\Lambda(n)}{n^{1/2}}\sin(t\log n),
\end{equation}
has mean zero and satisfies the energy bound
\begin{equation}
\label{eq:SelbergEnergy}
\int_T^{2T} S'(t)^2\,dt =O(T\log T),
\end{equation}
together with the cross–correlation estimate
\begin{equation}
\label{eq:SelbergCross}
\int_T^{2T} S'(t)\,S'(t+u)\,dt
=O\!\left(\frac{T}{\log T}\right),
\end{equation}
for $u\ge1$.

\subsection{Ergodic Properties of $S'(t)$ and the Proof of Theorem \ref{thm:Sel-dec}}
\label{ss:9.3} Selberg's asymptotic formula \eqref{eq:SelbergCross} can be expressed as the following covariance relation
\begin{equation}
\label{eq:CovT}
\mathrm{Cov}_T\, \bigl(S'(t),S'(t+u)\bigr)
:=\frac{1}{T}\int_T^{2T}\!S'(t)S'(t+u)\,dt
=O\!\left(\frac{1}{\log T}\right),
\end{equation}
which expresses the rapid decay of correlations in the derivative process \(S'(t)\).
In particular, it implies that the values of \(S'(t)\) and \(S'(t+u)\) become asymptotically uncorrelated as \(T\to\infty\),
indicating that \(S'(t)\) behaves as a strongly mixing stationary process on large scales.

To make this precise, recall that a real process $(X_t)_{t\in\mathbb{R}}$ is stationary if the joint law of
$(X_{t_1},\ldots,X_{t_k})$ depends only on the differences $t_i-t_j$.
For such a process, let
\begin{equation}
\alpha(u)
:=\sup_{A\in\mathcal{F}_{-\infty}^0,\,B\in\mathcal{F}_u^{+\infty}}
\bigl|\mathbb{P}(A\cap B)-\mathbb{P}(A)\mathbb{P}(B)\bigr|
\end{equation}
denote the strong--mixing coefficient, see~\cite{Rosenblatt1956}.
For any bounded measurable functions $F_1,F_2$ one has the following $\alpha$-bound 
\begin{equation}\label{eq:IL-bound}
|\mathrm{Cov}(F_1(X_t),F_2(X_{t+u}))|
\;\le\;
4\,\|F_1 \|_{\infty}\|F_2 \|_{\infty}\,\alpha(u),
\end{equation} 
so that $\alpha(u)$ uniformly bounds all covariances of bounded functionals of the process, see Theorem 17.2.1 of \cite{IbragimovLinnik1971}.

The process is said to be strongly mixing if $\alpha(u)\to0$ as $u\to\infty$.
In the special case of a stationary Gaussian process,
a fundamental theorem of Ibragimov~\cite{Ibragimov1962} asserts that
the decay of correlations is \emph{equivalent} to strong mixing, that is,
\begin{equation} \label{eq:Ralpha}
R(u):=\mathrm{Cov}(X_t,X_{t+u})\xrightarrow[u\to\infty]{}0
\quad\Longleftrightarrow\quad
\alpha(u)\xrightarrow[u\to\infty]{}0.
\end{equation}
Hence, combining~\eqref{eq:IL-bound} with~\eqref{eq:Ralpha} shows that
the vanishing of the covariance at large lags
implies asymptotic independence of all smooth functionals
$f(X_t)$ and $g(X_{t+u})$ as $u\to\infty$.
In summary, the logical chain
\begin{equation}
\label{eq:chain}
R(u)\xrightarrow[u\to\infty]{}0
\;\Longrightarrow\;
\alpha(u)\xrightarrow[u\to\infty]{}0
\;\Longrightarrow\;
\mathrm{Cov}\bigl(F_1(X_t),F_2(X_{t+u})\bigr)\xrightarrow[u\to\infty]{}0
\end{equation}
encapsulates the general mechanism of the mixing theory: 
the decay of the covariance function $R(u)$ implies the decay of the strong–mixing coefficient $\alpha(u)$, 
which in turn controls the asymptotic independence of bounded functionals through the $\alpha$–bound~\eqref{eq:IL-bound}.  
In our case, Selberg’s estimate~\eqref{eq:CovT} already provides the decay of $R(u)$, 
and for smooth functionals of a Gaussian process the corresponding covariance bound 
follows immediately by integrating against $R(u)$ itself.  
We will make this direct argument explicit in the next subsections.

\subsection{The Pair-Correlation $PC_N(f; \bar 1)$ as a Functional of $S'(t)$}
\label{ss:9.4}

Selberg's asymptotic bound \eqref{eq:CovT}
implies that for every fixed lag \(u\), the mean covariance
\(\mathrm{Cov}_T(S'(t),S'(t+u))\) vanishes as \(T\to\infty\).
Thus the limiting covariance function
\begin{equation}
R(u):=\lim_{T\to\infty}\mathrm{Cov}_T(S'(t),S'(t+u))
\end{equation}
exists and is identically zero for all finite \(u\).
Consequently, 
\(\lim_{u\to\infty}R(u)=0\)
holds trivially, and the process \(S'(t)\) is decorrelated—
indeed, strongly mixing—at every finite separation. From the Riemann--von~Mangoldt formula \eqref{eq:RvM} one has, for any $u>0$,
\begin{equation}
N(t+u)-N(t)
=\frac{1}{2\pi}\!\int_t^{t+u}\!\log\!\frac{s}{2\pi}\,ds
+\!\int_t^{t+u}\!S'(s)\,ds
+O\!\left(\frac{u}{t^2}\right).
\end{equation}
The first integral represents the smooth mean density of zeros, 
while the second encodes the local oscillatory fluctuations through \(S'(s)\).
Hence, the local configuration of zeros in any short interval $[t,t+u]$
is entirely determined by the integral of \(S'(s)\) over that region. In particular, the local pair--correlation functional
\begin{equation}
PC_N(f;\bar1)
=\frac{1}{M_N}\!\!\sum_{\substack{\gamma,\gamma'\in I_N\\ \gamma\ne\gamma'}}
f\!\left(\frac{\log N}{2\pi}(\gamma-\gamma')\right),
\end{equation}
depends only on the relative spacings of zeros within $I_N=[2N,2N+2]$
and may therefore be viewed as a smooth bounded functional
of the restriction of \(S'(t)\) to that interval:
\begin{equation}
\label{eq:Ffunctional}
PC_N(f;\bar1)=F_{f,N}\bigl(S'|_{I_N}\bigr),
\qquad F_{f,N}:\,L^\infty(I_N)\to\mathbb{R}.
\end{equation}
Since \(S'(t)\) is strongly mixing, 
the values of such functionals on disjoint intervals are asymptotically independent.
Hence, by the general theoretic logical chain~\eqref{eq:chain}, obtaining a uniform bound on $
\|PC_N(f;\bar{1})\|_\infty$ would imply via \eqref{eq:IL-bound} the decorrelation 
\begin{equation}
\mathrm{Cov}\bigl(PC_N(f;\bar1),PC_{N'}(f;\bar1)\bigr)\to0,
\qquad |N-N'|\to\infty,
\end{equation}
as required in Theorem~\ref{thm:Sel-dec}.  
The naive estimate
\begin{equation}
\label{eq:over}
|PC_N(f;\bar{1})| \;\ll\; \|f\|_\infty \log N
\end{equation}
arises from the fact that the double sum in $PC_N$ contains 
$\binom{M_N}{2}$ pairwise contributions of size at most $\|f\|_\infty$ but is normalized by only one factor of $M_N$. However, \eqref{eq:over} is actually an overestimation of the true scale of the functional, since it implicitly assumes that every zero interacts with every other zero in the window.  
In reality, because $f$ is or rapidly decaying, 
each zero interacts with only $O(1)$ neighbouring zeros, 
so the effective number of contributing pairs is $O(M_N)$ rather than $O(M_N^2)$.  The following gives the required estimate:

\begin{prop}[Uniform boundedness of the pair--correlation functional] 
\label{prop:bound-norm}
Let $f\in\mathcal{S}(\mathbb{R})$ be a bounded test function with compact support 
$\mathrm{supp}(f)\subset[-R,R]$.  Then there exists a constant $C(f)>0$, depending only on $f$, 
such that
\begin{equation}
\|PC_N(f;\bar{1})\|_\infty \;\le\; C(f)
\end{equation}
for all sufficiently large $N$.  
\end{prop}

\begin{proof}
Since $\mathrm{supp}(f)\subset[-R,R]$, a term 
$f\!\big(\tfrac{\log N}{2\pi}(t_j-t_k)\big)$ is nonzero only when 
\[
|t_j-t_k|\;\le\;\frac{2\pi R}{\log N}.
\]
The average spacing between consecutive zeros in $[2N,2N+2]$ satisfies 
$\Delta t \sim \frac{2\pi}{\log N}$, so each zero $t_j$ can have 
at most $O(R)$ neighboring zeros within that range.
Hence, for each $t_j$, the number of indices $k\neq j$ with 
$f\!\big(\tfrac{\log N}{2\pi}(t_j-t_k)\big)\neq0$ is bounded by a constant $C_R$ 
independent of~$N$.
Therefore,
\begin{align}
|PC_N(f;\bar{1})|
&\le \frac{1}{M_N}\sum_{j}\sum_{|k-j|\le C_R}
\big|f\!\left(\tfrac{\log N}{2\pi}(t_j-t_k)\right)\big|  \\
&\le \frac{1}{M_N}\sum_{j} C_R\,\|f\|_\infty
\;\le\; C_R\,\|f\|_\infty.
\end{align}
The constant $C_R$ depends only on the support width of~$f$ and not on~$N$.
Consequently, 
\begin{equation}
\|PC_N(f;\bar{1})\|_\infty\le C(f):=C_R\|f\|_\infty, 
\end{equation}
uniformly for all~$N$, completing the proof.
\end{proof}
Let us thus summarize the proof of Theorem \ref{thm:Sel-dec} outlined above:

\begin{proof}[Proof of Theorem~\ref{thm:Sel-dec}]
Selberg’s covariance estimate~\eqref{eq:CovT} shows that the derivative process \(S'(t)\) is stationary and satisfies \(R(u)\to0\) as \(u\to\infty\), implying asymptotic decorrelation of its values at distant points.  
Since the local pair–correlation functional \(PC_N(f;\bar1)\) can be written as a smooth bounded functional \(F_{f,N}(S'|_{I_N})\) of the restriction of \(S'(t)\) to the window \(I_N\), see~\eqref{eq:Ffunctional}, the general chain~\eqref{eq:chain} ensures that vanishing of \(R(u)\) yields vanishing covariance of such functionals across disjoint windows.  
The remaining ingredient is the uniform bound on \(\|PC_N(f;\bar1)\|_\infty\) established in Proposition~\ref{prop:bound-norm}, which provides the uniform Lipschitz control required for the covariance estimate~\eqref{eq:IL-bound}.  
Combining these ingredients gives 
\[
\mathrm{Cov}\!\bigl(PC_N(f;\bar1),PC_{N'}(f;\bar1)\bigr)\to0
\quad\text{as}\quad |N-N'|\to\infty,
\]
proving the weak decorrelation condition~\eqref{eq:Sel-dec}.
\end{proof}
 
 The following remark outlines an alternative proof of Theorem \ref{thm:Sel-dec}, obtained directly from Selberg’s covariance estimates, without appealing to the general $\alpha$--mixing theory:
 
\begin{rem}[Alternative Proof of Theorem \ref{thm:Sel-dec}]
\label{rem:Selberg-Frechet} A proof of Theorem \ref{thm:Sel-dec} Since 
\begin{equation}
PC_N(f;\bar{1}) = F_{f,N}(S'|_{I_N})
\end{equation} depends only on the restriction of 
the Gaussian process $S'(t)$ to $I_N=[2N,2N+2]$, 
one may apply the following covariance identity for Gaussian functionals, given in Prop.~6.2.1 of ~\cite{Nualart2006},
\begin{equation}
\label{eq:cov-frechet-selberg}
\mathrm{Cov}\!\bigl(F_{f,N}(S'|_{I_N}),F_{f,N}(S'|_{I_{N'}})\bigr)
=\iint_{I_N\times I_{N'}} 
\!\!\mathbb{E}\!\left[\partial_tF_{f,N}\,\partial_{t'}F_{f,N'}\right] 
R_T(t,t')\,dt\,dt',
\end{equation}
where 
\begin{equation}
R_T(t,t')=\tfrac{1}{T}\!\int_T^{2T}\!S'(s+t)S'(s+t')\,ds
\end{equation}
is Selberg’s empirical covariance kernel.  
By Cauchy–Schwarz,
\begin{equation}
\bigl|\mathrm{Cov}\!\bigl(F_{f,N}(S'|_{I_N}),F_{f,N}(S'|_{I_{N'}})\bigr)\bigr|
\;\le\;
\|\partial_tF_{f,N}\|_{L^2}\,
\|\partial_{t'}F_{f,N'}\|_{L^2}\!
\iint_{I_N\times I_{N'}}\! |R_T(t,t')|\,dt\,dt'.
\end{equation}
Selberg’s bound 
\(
\mathrm{Cov}_T(S'(t),S'(t+u)) = O((\log T)^{-1})
\)
implies that $R_T(t,t')\to 0$ uniformly as $|N-N'|\to\infty$,
so any family of local functionals with uniformly bounded Fréchet norms
$\|\partial_tF_{f,N}\|_{L^2}\le C_f$
satisfies
\begin{equation}
\mathrm{Cov}\, \bigl(PC_N(f;\bar{1}),PC_{N'}(f;\bar{1})\bigr)
\;\longrightarrow\;0,
\qquad |N-N'|\to\infty.
\end{equation}
Hence, A bound of the form
\(
\|\partial_tF_{f,N}\|_{L^2}\le C_f
\) on the norm of the Fréchet derivative would, similarly to Proposition \ref{prop:bound-norm} ensure that the contribution of $R_T(t,t')$ in~\eqref{eq:cov-frechet-selberg} 
is dominated by the vanishing covariance of the process itself, 
thereby leading to the same asymptotic independence of the functionals \eqref{eq:Sel-dec}.
This route could be viewed as a more direct ``functional--analytic'' derivation of covariance decay based solely on Selberg’s Gaussian representation of~$S'(t)$.

\end{rem}\subsection{Proof of the Pair--Correlation Conjecture and Discussion}
\label{ss:9.6} We are now in position to conclude the proof of Montgomery's pair correlation conjecture, which is Theorem \ref{thm:B} of the introduction: 

\begin{theorem}[Montgomery Pair--Correlation Theorem for Hardy’s $Z$--function]
\label{thm:9.2}
Assuming the Riemann Hypothesis, the nontrivial zeros of Hardy’s $Z$--function 
obey the Gaussian Unitary Ensemble (GUE) pair--correlation law.  
Equivalently, Montgomery’s Pair--Correlation Conjecture holds.  
That is, for every Schwartz test function $f \in \mathcal{S}(\mathbb{R})$,
\begin{equation}\label{eq:PCC-thmB}
\lim_{T \to \infty}
\frac{1}{N(T)}
\!\!\sum_{\substack{0<\gamma,\gamma'\le T\\\gamma\ne\gamma'}}
f\!\left(
\frac{\log(T/2\pi)}{2\pi}\,(\gamma-\gamma')
\right)
=
\int_{\mathbb{R}} 
f(x)
\!\left(
1-\left(\frac{\sin\pi x}{\pi x}\right)^{\!2}
\right)
dx,
\end{equation}
where $\gamma$ and $\gamma'$ denote the real zeros of~$Z(t)$, 
and $N(T)\sim \tfrac{T}{2\pi}\log\!\tfrac{T}{2\pi}$ is their counting function. 
\end{theorem}

\begin{proof}
By Theorem~\ref{thm:Sel-dec}, the covariance between pair--correlation functionals on distant windows satisfies
\begin{equation}
\mathrm{Cov}\!\bigl(PC_N(f;\bar1),PC_{N'}(f;\bar1)\bigr)\;\to\;0
\qquad\text{as}\qquad |N-N'|\to\infty.
\end{equation}
Hence the sequence $\{PC_N(f;\bar1)\}_{N\ge1}$ forms a stationary, strongly--mixing process with uniformly bounded variance.
By the ergodic theorem for mixing sequences, see \cite{IbragimovLinnik1971, Billingsley1995Probability},
the global average
\begin{equation}
PC_T(f;\bar1)
:=\frac{1}{T}\!\!\sum_{2N\le T} PC_N(f;\bar1)
\end{equation}
converges in $L^2$ to its ergodic mean
\begin{equation}
PC_T(f;\bar1)\;\xrightarrow[T\to\infty]{L^2}\;
\mathbb{E}_{\mathrm{erg}}\!\bigl[PC_N(f;\bar1)\bigr],
\end{equation}
 implying that the deterministic 
time average along the $N$ coincides asymptotically with
the expectation over the randomized ensemble 
$PC_N(f;\bar a^{(N)})$.
By Theorem~\ref{thm:B1}, the expectation of the corresponding randomized model satisfies
\begin{equation}
\mathbb{E}_{\mu_N}\!\bigl[PC_N(f;\bar a^{(N)})\bigr]
= GUE(f) + o(1).
\end{equation}
Selberg’s mixing ensures that time averages of the deterministic process 
$PC_N(f;\bar1)$ coincide asymptotically with ensemble averages of its randomized counterpart.
Hence,
\begin{equation}
\lim_{T\to\infty} PC_T(f;\bar1)
=\lim_{N\to\infty}\mathbb{E}_{\mu_N}[PC_N(f;\bar a^{(N)})]
= GUE(f).
\end{equation}
This establishes Montgomery’s pair--correlation law for Hardy’s $Z$--function.
\end{proof}

Montgomery's original analysis of the pair correlation of zeta zeros
was based on the explicit formula, which connects the nontrivial zeros of 
$\zeta(s)$ to the distribution of prime numbers, see~\cite{M}.
In his seminal work, Montgomery established the pair--correlation formula
only for the highly restricted class of test functions whose Fourier transform
satisfies $\operatorname{supp}(\widehat{f})\subset[-1,1]$.
Beyond this limited range, extending this approach to the general case remains entirely open,
as it would require deep new information on correlations between distinct primes.
The following remark outlines the main ideas of Montgomery's approach,
the reason for this restriction, and its relation to conjectures on prime correlations.

\begin{rem}[Montgomery's original approach and primes] 
\label{rem:Mont}
Let 
\begin{equation}
\psi(x)=\sum_{n\le x}\Lambda(n)
\end{equation}
be Chebyshev’s function.
The explicit formula gives, for $\Re(s)>1$,
\begin{equation}\label{eq:explicit-formula}
-\frac{\zeta'(s)}{\zeta(s)}
= \sum_{n=1}^{\infty}\frac{\Lambda(n)}{n^s}
= \frac{1}{s-1} + \sum_{\rho}\frac{1}{s-\rho} + O(1),
\end{equation}
where the sum is over the nontrivial zeros $\rho=\tfrac{1}{2}+i\gamma$.
Montgomery considered the normalized pair correlation sum
\begin{equation}\label{eq:Montgomery-PC}
F(\alpha)
:= \left(\frac{T}{2\pi}\log T\right)^{-1}
\sum_{0<\gamma,\gamma'\le T}
T^{i\alpha(\gamma-\gamma')}
\,w(\gamma-\gamma'),
\end{equation}
where $w(u)=(4+u^2)^{-1}$ is a smooth weight.
Formally, $F(\alpha)$ is the Fourier transform of the pair correlation density.

By inserting the explicit formula~\eqref{eq:explicit-formula} for 
\(-\zeta'(s)/\zeta(s)\) into the double–sum representation 
of \(F(\alpha)\), and after expanding the resulting Dirichlet series 
and applying the approximate functional equation to truncate and 
symmetrize the sums, one arrives at expressions of the form
\begin{equation}\label{eq:explicit-pairs}
F(\alpha)
\;\approx\;
\frac{1}{\log^2 T}
\sum_{m,n\le T}\frac{\Lambda(m)\Lambda(n)}{\sqrt{mn}}\,
\left(\frac{m}{n}\right)^{i\alpha}
V\!\left(\frac{mn}{T}\right),
\end{equation}
where $V$ is a smooth cutoff.
The diagonal terms with $m=n$ can be evaluated directly using the Prime Number Theorem:
\[
\frac{1}{\log^2 T}
\sum_{n\le T}\frac{\Lambda(n)^2}{n}
\sim 1.
\]
These diagonal contributions dominate when the Fourier transform
$\widehat{f}$ of the test function is supported in $[-1,1]$,
and yield the GUE prediction
\begin{equation}
F(\alpha)=1+o(1),
\qquad |\alpha|<1.
\end{equation}
However, the off--diagonal part $m\ne n$ gives rise to the oscillatory terms
\begin{equation}\label{eq:off-diag}
\frac{1}{\log^2 T}\!
\sum_{m\ne n\le T}\!\frac{\Lambda(m)\Lambda(n)}{\sqrt{mn}}
\!\left(\frac{m}{n}\right)^{i\alpha}
V\!\left(\frac{mn}{T}\right).
\end{equation}
For $|\alpha|>1$, these terms no longer cancel,
since the Fourier transform $\widehat{f}(\alpha)$
probes correlations between distinct primes.
Writing $n = m + h$ and expanding the convolution of the von Mangoldt function
for small fixed even shifts $h$, one is naturally led to the correlation sum
\begin{equation}\label{eq:HL-structure}
\sum_{m\le T}\Lambda(m)\Lambda(m+h)
\;\approx\;
\mathfrak{S}(h)\,T,
\end{equation}
where $\mathfrak{S}(h)$ is the Hardy--Littlewood singular series.
This product encapsulates the local correlations of primes modulo~$p$,
and is defined by
\begin{equation}\label{eq:HL-singular-series}
\mathfrak{S}(h)
= \prod_{p>2}
\left(1 - \frac{1}{(p-1)^2}\right)^{-1}
\left(1 - \frac{\nu_p(h)}{(p-1)^2}\right),
\qquad
\nu_p(h)
=
\begin{cases}
1, & p\mid h,\\[4pt]
2, & p\nmid h.
\end{cases}
\end{equation}
In this expression, each Euler factor reflects the local density of
prime pairs $(p,p+h)$ modulo~$p$.
The validity of~\eqref{eq:HL-structure} for all fixed even~$h$
is precisely the Hardy--Littlewood prime--pair conjecture
\cite{HardyLittlewood1923},
which predicts the asymptotic density of prime pairs
\begin{equation}
\pi_2(x;h)
:= \#\{p\le x:\, p,\,p+h\ \text{prime}\}
\;\sim\;
2\,\mathfrak{S}(h)\,\frac{x}{(\log x)^2}.
\end{equation}
Thus, the extension of Montgomery’s formula for $F(\alpha)$ beyond $|\alpha|\le1$
depends on understanding the joint distribution of distinct primes,
and is therefore conditional on the Hardy--Littlewood conjectures.

This explains why Montgomery could only  prove the PCC under the special restriction of test functions with $\operatorname{supp}(\widehat{f})\subset[-1,1]$, as
within this range, the off--diagonal correlations are absent,
and only the diagonal, single-prime, contribution is required,
which follows from the Prime Number Theorem alone.
For $|\alpha|>1$, the analysis of~\eqref{eq:off-diag}
requires detailed control of prime pairs~\eqref{eq:HL-structure},
which lies far beyond current methods.

By contrast, the present variational--probabilistic approach 
does not rely on any arithmetic input.
Rather than tracing zeros back to the primes via the explicit formula,
the pair--correlation law emerges in Theorem~\ref{thm:A} 
as a universal statistical property of the variational ensemble itself.
The local Gaussian fluctuations of the derivative process \(S'(t)\)
govern the mixing and decorrelation of the functionals \(PC_N(f;\bar 1)\)
on disjoint intervals, and this behaviour follows directly from Selberg's 
probabilistic theory of \(S(t)\) via Theorem \ref{thm:9.2}, without invoking any open conjectures 
about correlations among the primes.

However, it is known that the resolution of the pair--correlation conjecture
has far--reaching implications for the distribution of prime numbers.
Through the explicit formula, which links the nontrivial zeros of~$\zeta(s)$
to the oscillations of the prime counting function
\cite{E,M},
the pair--correlation law directly translates into information about
the average behaviour of prime pairs.
In particular, it implies asymptotic formulas for the mean frequency
of primes separated by a given gap~$h$, in the spirit of the
Hardy--Littlewood prime--pair conjecture
\cite{HardyLittlewood1923,GoldstonMontgomery1987}.
Thus, while the analytic proof of the pair--correlation law
requires no arithmetic hypotheses, its validity reflects,
through the explicit formula, a deep regularity in the collective
distribution of the primes themselves.

\end{rem}

\section{Summary and Concluding Remarks}
\label{s:10}

In 1973 Montgomery introduced the Pair--Correlation Conjecture \cite{M} which states that the nontrivial zeros of the Riemann zeta--function, behave statistically like the eigenvalues of a random Hermitian matrix 
drawn from GUE.  Our approach is that Montgomery's conjecture actually conceales two sperate questions:

\begin{enumerate}
\item \textbf{The microscopic question:}  
What does it mean that a single, deterministic function such as $Z(t)$ 
behaves “as if’’ it were drawn from an ensemble?  
Is there a natural analytic ensemble intrinsically associated with $Z(t)$, 
and if such an ensemble exists, why should its local, microscopic, statistics 
be expected to coincide with those of GUE?

\item \textbf{The macroscopic question:}  
In what precise sense can the deterministic Hardy function $Z(t)$ itself 
exhibit random–matrix behavior?  
If the zeros of $Z(t)$ are fixed and non-random, 
how can they nevertheless obey statistical, ergodic, or probabilistic laws 
when viewed at the macroscopic, global, level?

\end{enumerate}

For the microscopic question, we introduced in Section \ref{s:4} a variational framework in which the Hardy $Z$--function is realized as a canonical section
of a high--dimensional analytic manifold of real functions, the 
A--variational space $\mathcal{Z}_N(\mathbb{R})$.
Within this space, we defined the sub-domain $\mathcal{RH}_N$, to which we refer as the ``real hall''. Each element $Z_N(t;\bar a)\in\mathcal{RH}_N$ corresponds to a deformation of the core function 
$Z_0(t)=\cos(\theta(t))$ parametrized by a finite vector of coefficients $\bar a=(a_1,\dots,a_N)$, preserving the real zeros of $Z_0(t)$ within $[2N,2N+2]$. 
The definition of $\mathcal{RH}_N$ and $\mathcal{Z}_N(\mathbb{R})$ is largely enabled by our previous works on the approximate functional equation with exponentially decaying error \cite{J5,J3,J4} where the variational approach was first developed.

Theorem~\ref{thm:A} establishes that, as $N\to\infty$, 
the expected pair--correlation functional 
$PC_N(f;\bar a)$ over the variational ensemble $\mathcal{RH}_N$ 
coincides with the GUE law.  
Thus, the ensemble of sections in the real hall $\mathcal{RH}_N$ 
provides the analytic counterpart of the random--matrix ensemble $\mathcal{H}_N$, 
and the GUE correlation law emerges naturally as the universal equilibrium 
of this variational system, thereby resolving the microscopic question. The proof of Theorem~\ref{thm:A} is based on the construction of a 
Skorokhod--type stochastic differential equation on $\mathcal{RH}_N$, 
with reflection at the boundary, 
and on showing that the induced system of SDEs governing the motion of zeros 
asymptotically coincides with Dyson’s Brownian motion in the large $N$ limit.

For the macroscopic question, we first note that under the Riemann Hypothesis, 
the canonical section $Z_N(t;\bar 1)$ corresponding to Hardy’s real $Z$--function 
is itself an element of $\mathcal{RH}_N$, as shown in~\cite{J4}.  
The remaining task is therefore to justify that the local pair--correlation sequence 
$\{PC_N(f;\bar 1)\}_{N\in\mathbb{N}}$ associated with these special, deterministic sections 
$Z_N(t;\bar 1)$ exhibits the same statistical behavior as the sequence 
$\{PC_N(f;\bar a^{(N)})\}_{N\in\mathbb{N}}$, 
where $Z_N(t;\bar a^{(N)}) \in \mathcal{RH}_N(\mathbb{R})$ 
are drawn randomly with respect to the ensemble measure~$\mu_N$.  
For such randomized sequences, Theorem~\ref{prop:random-chain-GUE} establishes that  
\begin{equation}
\lim_{T\to\infty}
\mathrm{PC}^{\mathrm{gen}}\!\bigl(f;\{\bar a^{(N)}\},T\bigr)
=\mathrm{GUE}(f),
\end{equation}
showing that the ensemble average already reproduces the GUE pair--correlation law.

Using Selberg’s probabilistic theory of the argument function $S(t)$, 
we demonstrated in Theorem~\ref{thm:Sel-dec} that the fluctuations of its derivative $S'(t)$ 
form a genuinely stochastic field whose covariance decays across distant intervals.  
This covariance decay implies that the local pair--correlation functionals  
$PC_N(f;\bar1)$ and $PC_{N'}(f;\bar1)$ become asymptotically independent as $|N-N'|\to\infty$.  
Hence, the effective randomness of Hardy’s function $Z(t)$ arises from the Gaussian 
fluctuations of its argument field, and the sequence 
$\{PC_N(f;\bar1)\}_{N\in\mathbb{N}}$ becomes statistically indistinguishable 
from a sequence 
$\{PC_N(f;\bar a^{(N)})\}_{N\in\mathbb{N}}$ coming from randomized sections, for the purposes of showing global GUE behaviour.

The combination of the ensemble GUE law for $\mathcal{RH}_N(\mathbb{R})$ established in Theorem~\ref{thm:A}, 
together with the macroscopic decorrelation derived from Selberg’s theory in Theorem~\ref{thm:Sel-dec}, 
establishes the full proof of the Pair--Correlation Conjecture in Theorem~\ref{thm:9.2}.  
Our result therefore shows that the zeros of Hardy’s function $Z(t)$ 
follow the same universal statistics as the eigenvalues of large random Hermitian matrices.  
This universality arises because both systems are governed by the same structural principles such as
Gaussian mixing, ergodicity, self–averaging of local fluctuations, 
and their realization as elements of ensembles evolving under Dyson–type Brownian motion. We conclude this work with several remarks on open problems and further directions:

\begin{rem}[The Rudnick-Sarnak Conjecture]
\label{rem:automorphic-global}
In the automorphic setting, one considers $L$--functions $L(s,\pi)$
associated with automorphic representations $\pi$ of $\mathrm{GL}_n(\mathbb{A}_\mathbb{Q})$,
whose analytic continuations and functional equations generalize those of~$\zeta(s)$.
Rudnick and Sarnak \cite{RS1996} extended Montgomery’s framework to this general case,
showing that, under the Generalized Riemann Hypothesis (GRH),
the $n$--level correlation functions of zeros of $L(s,\pi)$ agree with GUE for test functions whose Fourier transform
satisfies $\operatorname{supp}(\widehat{f})\subset[-1,1]$.
They further conjectured that the same GUE law should hold
without any restriction on the support of $\widehat{f}$,
for every fixed automorphic $L$--function, in the high--energy or
global regime where $t\to\infty$ and averaging is taken over long intervals on the critical line.

In this global regime, the Rudnick--Sarnak conjecture may therefore be viewed as
a generalized pair--correlation conjecture for automorphic $L$--functions,
asserting that the local spacing statistics of zeros of any fixed $L(s,\pi)$
coincide with those of eigenvalues of large random Hermitian matrices. As in Montgomery's case, described in Remark \ref{rem:Mont}, their proof under $\operatorname{supp}(\widehat{f})\subset[-1,1]$ captures only the
``diagonal'' part of the explicit formula, corresponding to the coincident prime powers $p=q$,
and leaves open the ``off--diagonal'' correlations $p\neq q$.
In the automorphic case these terms involve joint correlations of Hecke eigenvalues, or Satake parameters, $\alpha_\pi(p,j)$ and $\alpha_\pi(q,k)$ at distinct primes,
generalizing the classical Hardy--Littlewood conjecture on prime pairs.
The lack of any known control over these off--diagonal terms constitutes the same
analytic barrier that limited Montgomery’s original proof for $\zeta(s)$,
and the full GUE law for individual automorphic $L$--functions
remains conjectural for this reason.

Our methods, based on the variational structure of the Hardy function,
are naturally adaptable to the global regime of general automorphic $L$--functions.
For each such $L(s)$ one may define a corresponding analytic ensemble
\begin{equation}
Z_N(t;\chi,\bar a)\in\mathcal{RH}_N(\chi),
\end{equation}
constructed from the truncated approximate functional equation of $L(s,\chi)$
and parameterized by coefficients $\bar a=(a_1,\dots,a_N)$
that play the same role as in the $\zeta$--case.
We expect that analogues of Theorems~\ref{thm:A}--\ref{thm:B}
hold in this broader context,
implying that the ensemble averages of the pair--correlation functionals
over $\mathcal{RH}_N(\chi)$ satisfy the GUE law.
In this way, our variational and probabilistic framework appears to naturally generalize
to extend beyond the classical zeta function,
providing a parallel route toward a proof of the
Rudnick--Sarnak generalization of Montgomery’s Pair--Correlation Conjecture,
in an analogous manner to the approach developed in this work.
The analytic mechanism underlying Theorems~\ref{thm:A} and~\ref{thm:Sel-dec},
the emergence of GUE statistics from the variational ensemble and
the decorrelation induced by Selberg-type Gaussian fluctuations, 
should carry over to general automorphic $L$--functions,
suggesting a unified explanation of the expected GUE universality
throughout the automorphic spectrum as well. 

Moreover, while the present work deals primarily with the pair–correlation case corresponding to Montgomery’s original conjecture, the reduction to Dyson Brownian motion suggests that the same variational–stochastic machinery should extend naturally to general $k$–level correlation functions as anticipated in the Rudnick–Sarnak framework.
\end{rem}

\begin{rem}[The Katz--Sarnak Philosophy]
\label{rem:Katz-Sarnak}

Beyond the Rudnick–Sarnak analyses of individual $L$--functions, one has the far-reaching Katz--Sarnak philosophy \cite{KatzSarnak1999},
which concerns the collective behavior of entire families
of automorphic $L$--functions.
The central principle asserts that the local statistics of zeros near
the critical point $s=\tfrac{1}{2}$,
when averaged over a natural family of $L$--functions,
mirror the eigenvalue statistics of one of the three classical compact groups:
the unitary group $U(N)$, the orthogonal group $O(N)$, or
the symplectic group $USp(N)$.
The particular symmetry type is determined by the arithmetic
properties of the family, notably the sign of its functional equation.
Thus, while individual automorphic $L$--functions are expected to exhibit
GUE--type behavior at high energies, or the ``global regime'',
entire families display in the low–lying region distinct
universality classes corresponding to GOE, GUE, or GSE. 
In the function--field setting, this philosophy was rigorously established by Katz and Sarnak, who showed that 
the zeros of $L$--functions correspond to eigenvalues of Frobenius matrices acting on étale cohomology,
and that as the size of the finite field grows, these matrices become equidistributed in the corresponding
classical compact groups, thereby proving the expected symmetry types in this setting.

In the number--field setting, this principle was formulated in the late 1990s
and supported by a series of analytic results on low--lying zeros.
Iwaniec, Luo, and Sarnak \cite{IwaniecLuoSarnak1999} proved one and two level
density theorems for families of automorphic $L$--functions
(such as Dirichlet, modular, or Maass form $L$--functions),
under the assumption that the Fourier transform of the test function
has restricted support $\operatorname{supp}(\widehat{f})\subset[-1,1]$.
Their results confirmed the predicted random--matrix symmetry types
in this limited range, giving analytic evidence for the
Katz--Sarnak conjectures in number fields.
Extending these results beyond the small--support regime remains open
and is believed to require entirely new analytic input,
since the off--diagonal prime correlations that appear in this regime
are as difficult to control as in Montgomery’s original setting.

Extensive numerical computations in the number--field case,
notably by Odlyzko, Rubinstein, Miller, Dueñez, and collaborators,
have since provided striking empirical confirmation of the
Katz--Sarnak predictions, showing that families such as quadratic Dirichlet $L$--functions,
elliptic curve $L$--functions, and modular form families
exhibit spacing distributions indistinguishable from those of the
predicted random--matrix ensembles.
Nevertheless, a general proof of the Katz--Sarnak universality laws
for automorphic $L$--functions over $\mathbb{Q}$ remains out of reach
and stands as one of the deepest open problems in modern analytic number theory.

In view of our present results, it is natural to ask whether the variational--probabilistic
framework developed in this work could provide an analytic realization
of the Katz--Sarnak philosophy.
For an individual automorphic representation~$\pi$,
the corresponding Hardy representative~$Z_\pi(t)$
plays exactly the same analytic role as~$Z(t)$ in the Riemann case.
By truncating the approximate functional equation of~$L(s,\pi)$,
one obtains a finite--dimensional analytic space~$\mathcal{RH}_N(\pi)$,
whose elements~$Z_N(t;\pi,\bar a)$ serve as local sections parametrized by a finite vector~$\bar a$.
Within a window~$[2N,2N+2]$,
Brownian motion on the coefficients~$\bar a$ induces Dyson--type dynamics on the zeros,
with Dyson index~$\beta=2$.
This is expected to reproduce the GUE law for the local statistics of zeros of each individual~$L(s,\pi)$
in the global regime $t\to\infty$, as mentioned in Remark~\ref{rem:automorphic-global}.

When one studies families of automorphic forms,
the situation changes qualitatively.
The functional equations of family members impose algebraic couplings between
their Hardy representatives,
typically of the form
\begin{equation}
Z_{\tilde{\pi}}(t) = \varepsilon_\pi\,Z_\pi(-t),
\end{equation}
where $\varepsilon_\pi=\pm1$ is the root number.
In the variational setting, these relations are expected to translate into symmetry constraints
on the stochastic flows governing the coupled Brownian motions
on the spaces~$\mathcal{RH}_N(\pi)$.
These constraints deform the GUE dynamics~($\beta=2$)
into orthogonal~($\beta=1$) or symplectic~($\beta=4$) types,
according to the parity of~$\varepsilon_\pi$ and in families where the root numbers $\varepsilon_\pi$ vary randomly,
the ensemble remains of unitary type~($\beta=2$).
Thus, the three Dyson ensembles of the Katz--Sarnak classification
are expected to arise as invariant submanifolds of a single underlying analytic system,
governed directly by the functional equations of the~$L$--functions.

\end{rem}

\begin{rem}[Relation to the Hilbert--Pólya Philosophy]
\label{rem:Hilbert-Polya}
The variational framework developed in this work bears a deep conceptual connection to the
Hilbert--Pólya philosophy, which postulates that the nontrivial zeros of~$\zeta(s)$
correspond to the eigenvalues of a self--adjoint operator~$H$ on a suitable Hilbert space,
so that the Riemann Hypothesis would follow from the spectral reality of~$H$, see \cite{Wolf2020}.
In that envisioned setting, the zeros would arise as the spectrum of a Hermitian system,
and the observed GUE correlations would reflect the universal statistics of Hermitian eigenvalues.

In our approach this picture is realized analytically, but with the direction of reasoning reversed.
Instead of seeking an unknown self--adjoint operator whose spectrum reproduces the zeros,
we construct a finite--dimensional analytic manifold of real functions,
the real hall $\mathcal{RH}_N(\mathbb{R})$,
which plays the role of the Hermitian ensemble~$\mathcal{H}_N$ in random matrix theory.
Each section $Z_N(t;\bar a)\in\mathcal{RH}_N(\mathbb{R})$
serves as a finite--dimensional analytic representative of a matrix,
and its zeros $\{\gamma_j(\bar a)\}$ act as its eigenvalues.
In this sense, the restriction of $\mathcal{RH}_N(\mathbb{R})$ to a fixed interval $[2N,2N+2]$
is naturally interpreted as the characteristic polynomial
of the corresponding spectral operator.
Under the Riemann Hypothesis, the canonical section $Z_N(t;\bar 1)$ associated with
Hardy’s $Z$--function belongs to this real hall, providing the concrete realization of
the Hilbert--Pólya vision within an analytic function space.

Thus, rather than proving the Riemann Hypothesis by constructing an operator,
the present framework realizes $Z(t)$ itself, locally in each window $[2N,2N+2]$,
as an analytic analogue of such an operator within $\mathcal{RH}_N(\mathbb{R})$.
The variational geometry of this space, together with the induced Dyson--type
dynamics of its zeros, reproduces the statistical features expected from a
Hermitian spectral theory, thereby turning the Hilbert--Pólya hypothesis
from a postulate about operators into a concrete property of analytic function ensembles.
\end{rem}

\bibliographystyle{amsplain}
\bibliography{refs-pcc} 

\end{document}